\documentclass[11pt]{aart}

\usepackage{titlesec}
\titleformat{\subsection}[runin]{\normalfont\bfseries}{\thesubsection.}{.5em}{}[.]\titlespacing{\subsection}{0pt}{2ex plus .1ex minus .2ex}{.8em}
\titleformat{\subsubsection}[runin]{\normalfont\itshape}{\thesubsubsection.}{.3em}{}[.]\titlespacing{\subsubsection}{0pt}{1ex plus .1ex minus .2ex}{.5em}

\usepackage[labelfont=sc,font=small,labelsep=period]{caption}
\usepackage[letterpaper, hmargin=1in, top=1.1in, bottom=1.3in, footskip=0.7in]{geometry}

\usepackage{amsmath} 
\usepackage{amssymb}
\usepackage{amsfonts}
\usepackage{latexsym}
\usepackage{amsthm}
\usepackage{amsxtra}
\usepackage{amscd}
\usepackage{bbm}
\usepackage{mathrsfs}
\usepackage{bm}

\usepackage{graphicx, color}

\definecolor{darkred}{rgb}{0.8,0,0.3}
\definecolor{darkblue}{rgb}{0,0.3,0.8}

\newcommand{\nc}{\normalcolor}

\usepackage{ifthen}
\def\comment#1{\ifthenelse{\isodd{\value{page}}}{\marginpar{\raggedright\scriptsize{\textcolor{darkred}{#1}}}}{\marginpar{\raggedleft\scriptsize{\textcolor{darkred}{#1}}}}}  

\definecolor{vdarkred}{rgb}{0.6,0,0.2}
\definecolor{vdarkblue}{rgb}{0,0.2,0.6}
\usepackage[pdftex, colorlinks, linkcolor=vdarkblue,citecolor=vdarkred]{hyperref}

\usepackage[nottoc,notlof,notlot]{tocbibind} 
\usepackage{cite} 

\numberwithin{equation}{section}
\numberwithin{figure}{section}
\theoremstyle{plain} 
\newtheorem{theorem}{Theorem}[section]
\newtheorem*{theorem*}{Theorem}
\newtheorem{lemma}[theorem]{Lemma}
\newtheorem*{lemma*}{Lemma}
\newtheorem{corollary}[theorem]{Corollary}
\newtheorem*{corollary*}{Corollary}
\newtheorem{proposition}[theorem]{Proposition}
\newtheorem*{proposition*}{Proposition}
\newtheorem{definition}[theorem]{Definition}
\newtheorem*{definition*}{Definition}

\newtheorem*{conjecture*}{Conjecture}

\newtheorem{assumption}[theorem]{Assumption}

\theoremstyle{definition} 
\newtheorem{example}[theorem]{Example}
\newtheorem*{example*}{Example}
\newtheorem{remark}[theorem]{Remark}
\newtheorem*{remark*}{Remark}

\newcommand{\f}[1]{\boldsymbol{\mathrm{#1}}} 
\newcommand{\bb}{\mathbb} 
\renewcommand{\r}{\mathrm}  
\renewcommand{\cal}{\mathcal}

\newcommand{\ul}[1]{\underline{#1} \!\,} 
\newcommand{\ol}[1]{\overline{#1} \!\,} 
\newcommand{\wh}{\widehat}
\newcommand{\wt}{\widetilde}
\newcommand{\txt}[1]{\text{\rm{#1}}}

\renewcommand{\P}{\mathbb{P}}
\newcommand{\E}{\mathbb{E}}
\newcommand{\R}{\mathbb{R}}
\newcommand{\C}{\mathbb{C}}
\newcommand{\N}{\mathbb{N}}
\newcommand{\Z}{\mathbb{Z}}

\newcommand{\me}{\mathrm{e}}
\newcommand{\ii}{\mathrm{i}}
\newcommand{\dd}{\mathrm{d}}
\newcommand{\col}{\mathrel{\vcenter{\baselineskip0.75ex \lineskiplimit0pt \hbox{.}\hbox{.}}}}
\newcommand*{\deq}{\mathrel{\vcenter{\baselineskip0.65ex \lineskiplimit0pt \hbox{.}\hbox{.}}}=}

\newcommand{\eqdist}{\overset{\text{d}}{=}}

\newcommand{\todist}{\overset{\text{d}}{\longrightarrow}}
\renewcommand{\leq}{\leqslant}
\renewcommand{\geq}{\geqslant}
\renewcommand{\epsilon}{\varepsilon}

\newcommand{\floor}[1] {\lfloor {#1} \rfloor}
\newcommand{\ceil}[1]  {\lceil  {#1} \rceil}

\newcommand{\qq}[1]{[\![{#1}]\!]}

\newcommand{\ind}[1]{\f 1 (#1)}

\newcommand{\indbb}[1]{\f 1 \pbb{#1}}

\newcommand{\p}[1]{({#1})}
\newcommand{\pb}[1]{\bigl({#1}\bigr)}
\newcommand{\pB}[1]{\Bigl({#1}\Bigr)}
\newcommand{\pbb}[1]{\biggl({#1}\biggr)}
\newcommand{\pBB}[1]{\Biggl({#1}\Biggr)}
\newcommand{\pa}[1]{\left({#1}\right)}

\newcommand{\q}[1]{[{#1}]}
\newcommand{\qb}[1]{\bigl[{#1}\bigr]}
\newcommand{\qB}[1]{\Bigl[{#1}\Bigr]}

\newcommand{\qBB}[1]{\Biggl[{#1}\Biggr]}
\newcommand{\qa}[1]{\left[{#1}\right]}

\newcommand{\h}[1]{\{{#1}\}}
\newcommand{\hb}[1]{\bigl\{{#1}\bigr\}}
\newcommand{\hB}[1]{\Bigl\{{#1}\Bigr\}}

\newcommand{\abs}[1]{\lvert #1 \rvert}
\newcommand{\absb}[1]{\bigl\lvert #1 \bigr\rvert}

\newcommand{\absbb}[1]{\biggl\lvert #1 \biggr\rvert}
\newcommand{\absBB}[1]{\Biggl\lvert #1 \Biggr\rvert}
\newcommand{\absa}[1]{\left\lvert #1 \right\rvert}

\newcommand{\norm}[1]{\lVert #1 \rVert}
\newcommand{\normb}[1]{\bigl\lVert #1 \bigr\rVert}

\newcommand{\scalar}[2]{\langle{#1} \mspace{2mu}, {#2}\rangle}
\newcommand{\scalarb}[2]{\bigl\langle{#1} \mspace{2mu}, {#2}\bigr\rangle}

\DeclareMathOperator{\diag}{diag}
\DeclareMathOperator{\tr}{Tr}

\DeclareMathOperator{\supp}{supp}
\DeclareMathOperator{\re}{Re}
\DeclareMathOperator{\im}{Im}

\DeclareMathOperator{\dist}{dist}

\DeclareMathOperator{\spec}{spec}

\newcommand{\bv}{{\bf{v}}}

\newcommand{\be}{\begin{equation}}
\newcommand{\ee}{\end{equation}}

\newcommand{\e}{{\varepsilon}}

\begin{document}

\title{Anisotropic local laws for random matrices}
\author{Antti Knowles\footnote{ETH Z\"urich, {\tt knowles@math.ethz.ch}.} \and Jun Yin\footnote{University of Wisconsin, {\tt jyin@math.wisc.edu}.}}
\maketitle

\begin{abstract}
We develop a new method for deriving local laws for a large class of random matrices. It is applicable to many matrix models built from sums and products of deterministic or independent random matrices. In particular, it may be used to obtain local laws for matrix ensembles that are \emph{anisotropic} in the sense that their resolvents are well approximated by deterministic matrices that are not multiples of the identity. For definiteness, we present the method for sample covariance matrices of the form $Q \deq T X X^* T^*$, where $T$ is deterministic and $X$ is random with independent entries. We prove that with high probability the resolvent of $Q$ is close to a deterministic matrix, with an optimal error bound and down to optimal spectral scales.

As an application, we prove the edge universality of $Q$ by establishing the Tracy-Widom-Airy statistics of the eigenvalues of $Q$ near the soft edges. This result applies in the single-cut and multi-cut cases. Further applications include the distribution of the eigenvectors and an analysis of the outliers and BBP-type phase transitions in finite-rank deformations; they will appear elsewhere.

We also apply our method to Wigner matrices whose entries have arbitrary expectation, i.e.\ we consider $W+A$ where $W$ is a Wigner matrix and $A$ a Hermitian deterministic matrix. We prove the anisotropic local law for $W+A$ and use it to establish edge universality.
\end{abstract}

\section{Introduction}

\subsection{Local laws}
The empirical eigenvalue density of a large random matrix typically converges to
a deterministic asymptotic density. For Wigner matrices this law is the celebrated Wigner semicircle law \cite{Wig} and for uncorrelated sample covariance matrices it is the Marchenko-Pastur law \cite{MP}. This convergence is best formulated using \emph{Stieltjes transforms}. Let $Q$ be an $M \times M$ Hermitian random matrix, normalized so that its eigenvalues are typically of order one, and denote by $R(z) \deq (Q - z I_M)^{-1}$ its resolvent. Here $z = E + \ii \eta$ is a spectral parameter with positive imaginary part $\eta$. Then the Stieltjes transform of the empirical eigenvalue density is equal to $M^{-1} \tr R(z)$, and the convergence mentioned above may be written informally as
\begin{equation} \label{intr_avg_ll}
\frac{1}{M} \tr R(z) \;=\; \frac{1}{M} \sum_{i = 1}^M R_{ii}(z) \;\approx\; m(z)
\end{equation}
for large $M$ and with high probability. Here $m(z)$ is the Stieltjes transform of the asymptotic eigenvalue density, which we denote by $\varrho$. We call an estimate of the form \eqref{intr_avg_ll} an \emph{averaged law}.

As may be easily seen by taking the imaginary part of \eqref{intr_avg_ll}, control of the convergence of $M^{-1} \tr R(z)$ yields control of an order $\eta M$ eigenvalues around the point $E$. A \emph{local law} is an estimate of the form \eqref{intr_avg_ll} for all $\eta \gg M^{-1}$. Note that the approximation \eqref{intr_avg_ll} cannot be correct at or below the scale $\eta \asymp M^{-1}$, at which the behaviour of the left-hand side of \eqref{intr_avg_ll} is governed by the fluctuations of individual eigenvalues. Such local laws have become a cornerstone of random matrix theory, starting from the work \cite{ESY2} where a local law was first established for Wigner matrices. Well known corollaries of a local law include bounds on the eigenvalue counting function as well as eigenvalue rigidity.  Moreover, local laws constitute the main tool needed to analyse (a) the distribution of eigenvalues (including the universality of the local spectral statistics), (b) eigenvector delocalization, (c) the distribution of eigenvectors, and (d) finite-rank deformations of $Q$.

In fact, for all of the applications (a)--(d), the averaged local law from \eqref{intr_avg_ll} is not sufficient. One has to control not only the normalized trace of $R(z)$ but the matrix $R(z)$ itself, by showing that $R(z)$ is close to some deterministic matrix depending on $z$, provided that $\eta \gg M^{-1}$. Such control was first obtained for Wigner matrices in \cite{EYY1}, where the closeness was established in the sense of individual matrix entries: $R_{ij}(z) \approx m(z) \delta_{ij}$. We call such an estimate an \emph{entrywise local law}. More generally, in \cite{KY2, BEKYY} this closeness was established in the sense of \emph{generalized matrix entries}:
\begin{equation} \label{intr_iso_ll}
\scalar{\f v}{R(z) \f w} \;\approx\; m(z) \scalar{\f v}{\f w}\,, \qquad \eta \;\gg\; M^{-1}\,, \qquad \abs{\f v}, \abs{\f w} \;\leq\; 1\,.
\end{equation}
Analogous results for uncorrelated sample covariance matrices were obtained in \cite{PY, BEKYY}. The estimate \eqref{intr_iso_ll} states that for large $M$ the resolvent $R(z)$ is approximately \emph{isotropic} (i.e.\ proportional to the identity matrix), and we accordingly call an estimate of the form \eqref{intr_iso_ll} an \emph{isotropic local law}. We remark that the basis-independent control in \eqref{intr_iso_ll} is crucial for many applications, including the distribution of eigenvectors and the study of finite-rank deformations of $Q$.

Unlike in the case of Wigner matrices and uncorrelated sample covariance matrices mentioned above, the resolvent $R(z)$ is in general not close to a multiple of the identity matrix, but rather to some general deterministic matrix $P(z)$. In that case \eqref{intr_iso_ll} is to be replaced with
\begin{equation} \label{intr_aniso_ll}
\scalar{\f v}{R(z) \f w} \;\approx\; \scalar{\f v}{P(z) \f w}\,, \qquad \eta \;\gg\; M^{-1}\,, \qquad \abs{\f v}, \abs{\f w} \;\leq\; 1\,.
\end{equation}
We call an estimate of the form \eqref{intr_aniso_ll} an \emph{anisotropic local law}. The goal of this paper is to develop a method yielding (anisotropic) local laws for many matrix models built from sums and products of deterministic or independent random matrices. Applications include all the four (a)--(d) listed above, some of which we illustrate in this paper.

\subsection{Sample covariance matrices}
For definiteness, and motivated by applications to multivariate statistics, in this paper we focus mainly on sample covariance matrices. (In Section \ref{sec:Wig}, we also explain how to apply our method to deformed Wigner matrices.) We consider sample covariance matrices of the form $Q = N^{-1} A A^*$, where $A$ is an $M \times N$ matrix. The columns of $A$ represent $N$ independent and identically distributed observations of some random $M$-dimensional vector $\f a$. We shortly outline the statistical interpretation of $Q$, and refer e.g.\ to \cite{KYY} for more details. A fundamental goal of multivariate statistics is to obtain information on the population covariance matrix $\Sigma \deq \E \f a \f a^*$ from $N$ independent samples $A = [\f a^{(1)} \cdots \f a^{(N)}]$ of the population $\f a$. Then $Q$ corresponds to $\Sigma$ with the expectation replaced by the empirical average over the $N$ samples. If the entries $\f a$ do not have mean zero, then the population covariance matrix $\Sigma$ reads $\E (\f a - \E \f a)(\f a - \E \f a)^*$, and the sample covariance matrix is accordingly obtained by subtracting the empirical average $\frac{1}{N} \sum_{\nu = 1}^N A_{i \nu}$ from each entry $A_{i \mu}$. We may therefore write the sample covariance matrix as $\dot Q = (N-1)^{-1} A (I_N - \f e \f e^*) A^*$, where we introduced the normalized vector $\f e \deq N^{-1/2}(1,1, \dots, 1)^* \in \R^N$. Since $\dot Q$ is invariant under the deterministic shift $A_{i \mu} \mapsto A_{i \mu} + f_i$, we may without loss of generality assume that $\E A_{i \mu} = 0$. We shall always makes this assumption.

For the sample vector $\f a$ we take a linear model $\f a = T \f b$, where $T$ is a deterministic $M \times \wh M$ matrix and $\f b$ is a random $\wh M$-dimensional vector with zero expectation. We assume that the entries of $\f b$ are independent. This model includes in particular the usual linear (or ``signal + noise'') models from multivariate statistics; see \cite[Section 1.2]{BEKYY} for more details. Note that, in addition to the assumption $\E b_i = 0$, we may without loss of generality assume that $\E \abs{b_i}^2 = 1$ by absorbing the variance of $b_i$ into the deterministic matrix $T$. Hence, without loss of generality, throughout the following we make the assumptions $\E b_i = 0$ and $\E \abs{b_i}^2 = 1$.

Letting $X$ be an $\wh M \times N$ matrix of independent entries satisfying $\E X_{i \mu} = 0$ and $\E \abs{X_{i \mu}}^2 = N^{-1}$, we may therefore write the matrices $Q$ and $\dot Q$ as
\begin{equation} \label{def_Q}
Q \;=\; T X X^* T^*\,, \qquad \dot Q \;=\; \frac{N}{N-1} T X (I_N - \f e \f e^*) X^* T^*\,.
\end{equation}
It is easy to check that in both cases the associated population covariance matrix is $\Sigma = \E Q = \E \dot Q = T T^*$. The matrix $Q$ was first studied in the seminal work of Marchenko and Pastur \cite{MP}, where it was proved that the Stieltjes transform $m$ of the asymptotic density $\varrho$ may be characterized as the solution of an integral equation, \eqref{def_m} below, depending on the spectrum of $\Sigma$. In the language of free probability, the asymptotic density $\varrho$ is the free multiplicative convolution of the famous Marchenko-Pastur law with the empirical eigenvalue density of $\Sigma$.

\subsection{Outline of results}
For simplicity, we focus on the matrix $Q$, bearing in mind that similar results also apply for $\dot Q$ (see Section \ref{sec:Q_dot}). We assume that the three matrix dimensions $M, \wh M, N$ are comparable, and that the entries of $\sqrt{N} X$ possess a sufficient number of bounded moments. Moreover, we assume that $\norm{\Sigma}$ is bounded, and that the spectrum of $\Sigma$ satisfies certain regularity conditions, given Definition \ref{def:regular} below, which essentially state that the connected components of the support of $\varrho$ are separated by some positive constant, and that the density of $\varrho$ has square root decay near its edges in $(0,\infty)$. Note that we do not assume that $T$ is square, and in particular $T$ may have many vanishing singular values; this allows us to cover e.g.\ the general linear model of multivariate statistics.

Our main result is the anisotropic local law for $Q$. Roughly, it states that \eqref{intr_aniso_ll} holds with
\begin{equation*}
P(z) \;\deq\; - (z (I_M + m(z) \Sigma))^{-1}\,,
\end{equation*}
where $m(z)$ is the Stieltjes transform of the asymptotic density $\varrho$. In fact, we prove a more general anisotropic local law that is more useful in applications. Its formulation is most transparent under the additional assumption that $T = T^* = \Sigma^{1/2}$, although this assumption is a mere convenience and may be easily removed (see Section \ref{sec:gen_T}). We prove an anisotropic local law of the form
\begin{equation} \label{intr_G_Pi}
\scalar{\f v}{G(z) \f w} \approx \scalar{\f v}{\Pi(z) \f w}\,,
\end{equation}
for $\eta \gg M^{-1}$ and $\abs{\f v}, \abs{\f w} \leq 1$, where we defined
\begin{equation*}
\quad G(z) \deq \begin{pmatrix}
-\Sigma^{-1} & X
\\
X^* & -zI_N
\end{pmatrix}^{-1}\,, \qquad
\Pi(z) \deq \begin{pmatrix}
-\Sigma(I_M + m(z) \Sigma)^{-1} & 0
\\
0 & m(z) I_N
\end{pmatrix}\,.
\end{equation*}
A simple application of Schur's complement formula to \eqref{intr_G_Pi} yields the anisotropic local law for the resolvent of $Q = TXX^*T^*$ and a similar result for the resolvent of the companion matrix $X^* T^* T X$. The estimate \eqref{intr_G_Pi} holds with high probability, and we give an explicit and optimal error bound. We remark that the anisotropic local law holds under very general assumptions on the distribution of $X$, the dimensions of $X$ and $T$, and the spectrum of $T T^*$. In particular, we make no assumptions on the singular vectors of $T$. We remark that, previously, an anisotropic global law, valid for $\eta \asymp 1$, was derived in \cite{HLNP} for a different matrix model.

As an application of the anisotropic local law, we prove the edge universality of the eigenvalues near the soft spectral edges, whereby the joint distribution of the eigenvalues is asymptotically governed by the Tracy-Widom-Airy statistics of random matrix theory. More precisely, we prove that the asymptotic distribution of the eigenvalues near the soft edges depends only on the nonzero spectrum of $T T^*$. This may be regarded as a universality result in both the distribution of the entries of $X$ and the (left and right) singular vectors of $T$, including their dimensions. We then conclude that the Tracy-Widom-Airy statistics hold near the soft edges by noting that they have been previously established \cite{HHN, ElK2, Ona, SL} for Gaussian $X$ and diagonal $T$.

We comment briefly on the history of edge universality for sample covariance matrices of the form $Q$. The Tracy-Widom-Airy statistics were first established near the rightmost spectral edge in the case of complex Gaussian $X$ in \cite{ElK2, Ona}. In \cite{BaoPanZhou}, this result was extended to general complex $X$ under the assumption that $\Sigma$ is diagonal, i.e.\ the population vector $\f a$ is uncorrelated. Very recently, in \cite{HHN} the Tracy-Widom-Airy statistics were established near all soft edges in the case of complex Gaussian $X$. Moreover, in \cite{SL}, building on a new comparison method developed in \cite{LSY}, the Tracy-Widom-Airy statistics near the rightmost edge were established also in the case of real Gaussian $X$, or general real $X$ and diagonal $\Sigma$. We remark that all proofs from \cite{ElK2, Ona, HHN} crucially rely on the integrable structure of $Q$ in the complex Gaussian case (the Harish-Chandra-Itzykson-Zuber formula and the determinantal form of the eigenvalue process); this structure is not available in the real case, and the method of \cite{SL} is different from that of \cite{HHN, ElK2, Ona}.

We also prove the rigidity of eigenvalues, as well as the complete delocalization of the eigenvectors with respect to an arbitrary deterministic basis. Further applications of the anisotropic local law, such as the distribution of the eigenvectors and an analysis of the outliers and BBP-type phase transitions in finite-rank deformations, will appear elsewhere.

Finally, we also apply our method to \emph{deformed Wigner matrices} of the form $W + A$, where $W$ is a Wigner matrix and $A$ a bounded Hermitian matrix. This model describes Wigner matrices whose entries may have arbitrary expectations. As for $Q$, we establish the Tracy-Widom-Airy statistics near the spectral edges of $W+A$. More precisely, we prove that the asymptotic distribution of the eigenvalues near the edges depends only on the asymptotic spectrum of $A$, which may be regarded as a universality result in the distribution of $W$ and the eigenvectors of $A$. We then conclude that the Tracy-Widom-Airy statistics hold near the edges by noting that they have been previously established \cite{LSY} for diagonal $A$.

\subsection{Overview of the proof}
We conclude this section by outlining some ideas of the proof of the anisotropic local law. Roughly, the proof proceeds in three steps: (A) the entrywise local law for Gaussian $X$ and diagonal $\Sigma$, (B) the anisotropic local law for Gaussian $X$ and general $\Sigma$, and (C) the anisotropic local law for general $X$ and general $\Sigma$. Steps (A) and (B) may be performed by adapting the methods of \cite{BEKYY}, and we do not comment on them any further.

The main argument, and the bulk of the proof, is Step (C). Its core is a \emph{self-consistent comparison method}, which yields the anisotropic local law for general $X$ assuming it has been proved for Gaussian $X$. Up to now, all interpolation or Lindeberg replacement arguments in random matrix theory have crucially relied on a local law as input. Since the local law is exactly what we are trying to prove, a different approach is clearly needed -- one that does not need a local law to work.
Our method yields a new way of deriving local laws for general random matrices expressed as polynomials of deterministic and random matrices whose entries are independent. It relies on two key novel ideas: (a) continous Green function comparison argument where the errors are controlled self-consistently, and (b) a bootstrapping of the Green function comparison on the spectral scale $\eta$.

In the remainder of this subsection we give a few details of our method, and in particular explain the ideas (a) and (b) in more detail. We construct a continuous family $(X^\theta)_{\theta \in [0,1]}$ of matrices, whereby $X^0$ is a Gaussian ensemble and $X^1$ is the ensemble in which we are interested.
Then we operate in the $(\theta, \eta)$-plane and perform simultaneously a continuous interpolation in $\theta$ and a discrete bootstrapping in $\eta$. (See Figure \ref{fig:bootstrap} below.)

Interpolation methods have been extensively used in random matrix theory, in the form of both discrete Lindeberg-type replacement schemes \cite{Chat, TV1, EYY1, EYY3} and continuous interpolations \cite{SL,LSY,LP}. Moreover, Dyson Brownian motion, as used e.g.\ in \cite{SL,LSY,ESY4,ESYY}, may be regarded as a form of continuous interpolation for the case that $X^0$ is Gaussian. The usual choice of interpolation, as used in \cite{SL,LP} and also resulting from Dyson Brownian motion, is $X^\theta = \sqrt{1 - \theta} \, X^0 + \sqrt{\theta} \, X^1$. In this paper, we instead interpolate using i.i.d.\ Bernoulli random variables:
\begin{equation} \label{bern_interpol}
X^\theta_{\alpha} \;\deq\; \chi^\theta_\alpha X_\alpha^1 + (1 - \chi^\theta_\alpha) X_\alpha^0\,,
\end{equation}
where $\alpha = (i, \mu)$ indexes the matrix entries and $(\chi^\theta_\alpha)$ is a family of i.i.d.\ Bernoulli-$\theta$ random variables. The choice \eqref{bern_interpol} is convenient for our purposes, since the expectation of a function of $X^\theta = (X_\alpha^\theta)$ is particularly easy to differentiate in $\theta$, resulting in explicit formulas that are the starting point of our analysis.
Indeed, as explained in Section \ref{sec:comparison1} (see Remark \ref{rem:Bernoulli}), for any smooth enough function $F$ and $\ell \in \N$ we have an identity of the form
\begin{equation} \label{interpol25}
\frac{\dd}{\dd \theta} \E F (X^\theta) \;=\;  \sum_{n = 1}^{\ell}  \sum_{\alpha} K_{n,\alpha}^\theta \, \E \pbb{\frac{\partial}{\partial X^\theta_{\alpha}}}^n F (X^\theta) + \cal E_{\ell}\,,
\end{equation}
where $K_{n,\alpha}^\theta$ is defined through the formal power series
\begin{equation} \label{K_n_def}
\sum_{n \geq 1} K_{n,\alpha}^\theta \, t^n \;\deq\; \frac{\E \me^{t X_\alpha^1} - \E \me^{t X_\alpha^0}}{\E \me^{t X_\alpha^\theta}}\,,
\end{equation}
and the error term $\cal E_{\ell}$ is negligible for large enough $\ell$.
In particular, $K_{n,\alpha}^\theta$ depends only on the first $n$ moments of $X^0_\alpha$ and $ X^1_\alpha$, and vanishes if the first $n$ moments of $X^0_\alpha$ and $X^1_\alpha$ match. We remark that the interpolation method from \eqref{bern_interpol}--\eqref{K_n_def} can be used to simplify several previous comparison arguments in random matrix theory, especially the proofs of edge universality in \cite{EYY3,EKYY2}.

Next, we explain the basic strategy behind the self-consistent comparison method mentioned in (a) above. We choose a function
\begin{equation*}
F_{\f v}^p(X,z) \;\deq\; \absb{\scalar{\f v}{(G(z) - \Pi(z)) \f v}}^p
\end{equation*}
that controls the error in the anisotropic low. Here $\f v$ is a deterministic vector on the unit sphere $\bb S$ and $p$ a fixed integer. Our goal is to prove that $\absb{\scalar{\f v}{(G(z) - \Pi(z)) \f v}}   \leq  N^{o(1)} \nc \Psi(z)$ with high probability, where $\Psi$ is an explicit error parameter defined in \eqref{def_Psi} below. By Markov's inequality, it suffices to prove $\E F_{\f v}^p(X^1,z)   \leq N^{o(1)} \nc \Psi(z)^p$ for large but fixed $p$. Since we know (by assumption on the Gaussian case) that $\E F_{\f v}^p(X^0,z)   \leq N^{o(1)} \nc \Psi(z)^p$, it suffices to estimate the derivative as $\frac{\dd}{\dd \theta} \E F_{\f v}^p(X^\theta,z)  \leq N^{o(1)} \nc \Psi(z)^p$. In fact, by Gr\"onwall's inequality, it suffices to establish the weaker \emph{self-consistent} estimate
\begin{equation} \label{intr_sc}
\frac{\dd}{\dd \theta} \E F_{\f v}^p(X^\theta,z)  \;\leq\; N^{o(1)} \nc \Psi(z)^p + C \sup_{\f w \in \bb S} \E F_{\f w}^p(X^\theta, z)\,.
\end{equation} 
Note that \eqref{intr_sc} is self-consistent in the sense that the right-hand side depends on the quantities to be estimated. Most of the work is therefore to estimate the right-hand side of \eqref{interpol25} with $F \deq F_{\f v}^p$ by the right-hand side of \eqref{intr_sc}. The derivatives on the right-hand side of \eqref{interpol25} are polynomials of generalized matrix entries of $G$. Their structure has to be carefully tracked, which we do in Sections \ref{sec:comparison1} and \ref{sec:self-const_II} by encoding them using an appropriately constructed family of words.

It turns out, however, that in general these polynomials involve factors that cannot be estimated by either term on the right-hand side of \eqref{intr_sc}. To handle these bad terms, as mentioned in (b) above, we make use of a \emph{bootstrapping in the spectral scale $\eta$}. The basic idea is to perform the interpolation from $\theta = 0$ to $\theta = 1$ on a spectral scale $\tilde \eta$, and to use the obtained anisotropic local law at scale $\tilde \eta$ to deduce rough a priori bounds on the generalized entries of $G$ on the smaller spectral scale $\eta$. The bootstrapping is started at $\eta = 1$, where the trivial a priori bound $\eta^{-1} = 1$ on the entries of $G$ is sufficient. The bootstrapping proceeds in multiplicative increments of $N^{-\delta}$, where $\delta > 0$ is a fixed small exponent. Hence, the bootstrapping iteration needs to be performed only an order $O(1)$ times. (We remark that this bootstrapping is different from the stochastic continuity arguments commonly used to establish local laws for Wigner matrices \cite{EKYY1, EKYY4, ESY2}, which propagate much more precise estimates but have to be iterated an order $N^{O(1)}$ times.)

\begin{figure}[ht!]
\begin{center}
{\small 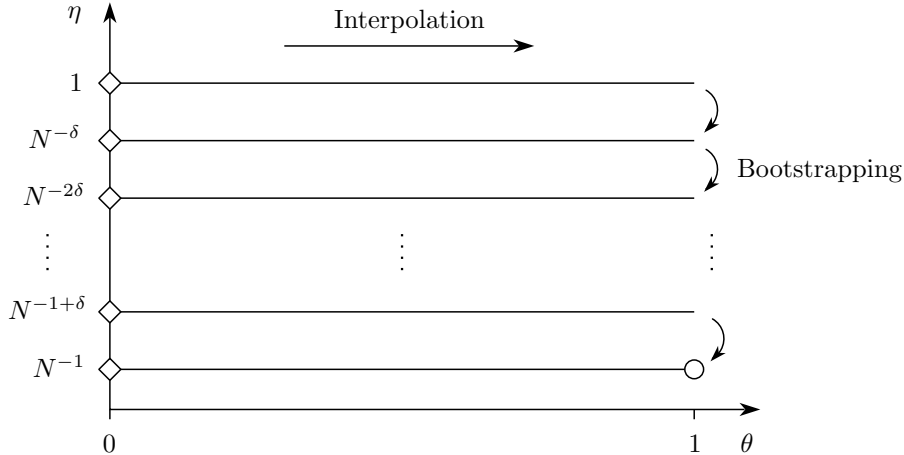}
\end{center}
\caption{
A schematic overview of the self-consistent comparison method. The horizontal axis depicts the interpolation parameter $\theta \in [0,1]$, with $\theta = 0$ corresponding to the Gaussian case and $\theta = 1$ to the general case. The vertical axis depicts the spectral scale $\eta$, with $\eta = 1$ corresponding to the macroscopic scale and $\eta = N^{-1}$ the microscopic scale. The self-consistent comparison argument takes place in the $(\theta, \eta)$-plane, and follows the horizontal lines from left to right and top to bottom, i.e.\ $(0,1) \to (1,1), (0,N^{-\delta}) \to (1, N^{-\delta}), (0,N^{-2\delta}) \to (1, N^{-2\delta}), \dots$. The diamonds on the vertical axis indicate that the anisotropic local law is known at $(\theta, \eta)$. The circle indicates the goal $(\theta,\eta) = (1,N^{-1})$.
\label{fig:bootstrap}}
\end{figure}

At each step of the bootstrapping, we use the anisotropic local law on the scale $\tilde \eta = N^\delta \eta$ and simple monotonicity properties of $G$ to obtain a priori bounds on $\abs{G_{\f v \f w}}$ and $\im G_{\f w \f w}$ on the scale $\eta$. Roughly, the smallness that allows us to propagate the bootstrapping arises from a certain structure of summation indices in the polynomials, which yields factors of $\im G_{\f w \f w}$. It is important to ensure that a sufficiently large number of such factors is generated, and that the a priori bound on them is close to optimal. In contrast, the a priori bound that one can use for general entries $\abs{G_{\f v \f w}}$ is very rough. Figure \ref{fig:bootstrap} gives a schematic representation of the self-consistent comparison method, which combines a continuous interpolation in $\theta$ and a discrete bootstrapping in $\eta$. We refer to Sections \ref{sec:sketch_compar} and \ref{sec:sketch_proof_7} for a more detailed outline of the proof.

The next section is devoted to the definition of the model and the statement of results. We give an outline of the structure of the paper in Section \ref{sec:outline} below.

\subsection*{Conventions}
The fundamental large parameter is $N$. All quantities that are not explicitly constant may depend on $N$; we almost always omit the argument $N$ from our notation.

We use $C$ to denote a generic large positive constant, which may depend on some fixed parameters and whose value may change from one expression to the next. Similarly, we use $c$ to denote a generic small positive constant. If a constant $C$ depends on some additional quantity $\alpha$, we indicate this by writing $C_\alpha$. For two positive quantities $A_N$ and $B_N$ depending on $N$ we use the notation $A_N \asymp B_N$ to mean $C^{-1} A_N \leq B_N \leq C A_N$ for some positive constant $C$. For $a < b$ we set $\qq{a,b} \deq [a,b] \cap \Z$. We use the notation $\f v = (\f v(i))_{i = 1}^n$ for vectors in $\C^n$, and denote by $\abs{\,\cdot\,}= \norm{\,\cdot\,}_2$ the Euclidean norm of vectors. We use $\norm{\,\cdot\,}$ to denote the Euclidean operator norm of a matrix. We often write the $k \times k$ identity matrix $I_k$ simply as $1$ when there is no risk of confusion.

We use $\tau > 0$ in various assumptions to denote a positive constant that may be chosen arbitrarily small. A smaller value of $\tau$ corresponds to a weaker assumption. All constants may $C,c$ depend on $\tau$, and we neither indicate nor track this dependence.

\section{Model} \label{sec:model}

In this section we define our model, list our key assumptions, and explain the basic structure of the asymptotic eigenvalue density $\varrho$.

\subsection{Definition of the model} \label{sec:def_model}
We consider the $M \times M$ matrix $Q \deq T X X^* T^*$, where $T$ is a deterministic $M \times \wh M$ matrix and $X$ a random $\wh M \times N$ matrix. We regard $N$ as the fundamental parameter and $M \equiv M_N$ and $\wh M \equiv \wh M_N$ as depending on $N$. Here, and throughout the following, we omit the index $N$ from our notation, bearing in mind that all quantities that are not explicitly constant (such as the constant $\tau$) may depend on $N$. For simplicity, we always make the assumption that
\begin{equation} \label{MsimN}
M \;\asymp\; \wh M \;\asymp\; N\,.
\end{equation} 
(This assumption may relaxed to $\log M \asymp \log \wh M \asymp \log N$ with some extra work; see \cite{KYY}. We do not pursue this direction here.) \nc We introduce the dimensional ratio
\begin{equation} \label{def_phi}
\phi \;\deq\; \frac{M}{N}\,.
\end{equation}
Note that \eqref{MsimN} implies
\begin{equation} \label{phi_bounded}
\tau \;\leq\; \phi \;\leq\; \tau^{-1}
\end{equation}
provided $\tau > 0$ is chosen small enough.

We assume that the entries $X_{i \mu}$ of the $\wh M \times N$ matrix $X$ are independent (but not necessarily identically distributed) real-valued random variables satisfying
\begin{equation} \label{cond on entries of X}
\E X_{i \mu}\;=\;0\,,\qquad \E \abs{X_{i \mu}}^2\;=\; \frac{1}{N}
\end{equation}
for all $i$ and $\mu$.
In addition, we assume that,  for all $p \in \N$, the random variables $\sqrt{N} X_{i \mu}$ have a uniformly bounded $p$-th moment. In other words, we assume that for all $p \in \N$ there is a constant $C_p$ such that
\begin{equation} \label{moments of X-1}
\E \absb{\sqrt{N} X_{i \mu}}^p \;\leq\; C_p
\end{equation}
for all $i$ and $\mu$.
The assumption that \eqref{moments of X-1} hold for all $p \in \N$ may be easily relaxed. For instance, it is easy to check that our results and their proofs remain valid, after minor adjustments, if we only require that \eqref{moments of X-1} holds for all $p \leq C$ for some large enough constant $C$. We do not pursue such generalizations. 

For definiteness, in this paper we focus mostly on the real symmetric case ($\beta = 1$), \nc where all matrix entries are real. We remark, however, that all of our results and proofs also hold, after minor changes, in the complex Hermitian case  ($\beta = 2$), \nc where $X_{i \nu} \in \C$ and in addition to \eqref{cond on entries of X} we have $\E X_{i \mu}^2 = 0$.

The population covariance matrix is defined as
\begin{equation*}
\Sigma \;\deq\; \E Q \;=\; T T^*\,.
\end{equation*}
We denote the eigenvalues of $\Sigma$ by
\begin{equation*}
\sigma_1 \;\geq\; \sigma_2 \;\geq\; \cdots \;\geq\; \sigma_M \;\geq\; 0\,.
\end{equation*}
Let
\begin{equation} \label{def_pi}
\pi \;\deq\; \frac{1}{M} \sum_{i = 1}^M \delta_{\sigma_i}
\end{equation}
denote the empirical spectral density of $\Sigma$.
We suppose that
\begin{equation} \label{bound_Sigma}
\sigma_1 \;\leq\; \tau^{-1}
\end{equation}
and that
\begin{equation} \label{Sigma_gap}
\pi ([0,\tau]) \;\leq\; 1 - \tau\,.
\end{equation}
This latter assumption means that the spectrum of $\Sigma$ cannot be concentrated at zero.

Sometimes it will be convenient to make the following stronger assumption on $T$:
\begin{equation} \label{T_diag}
\wh M = M \txt{ and } T = T^* = \Sigma^{1/2} > 0\,.
\end{equation}
The assumption \eqref{T_diag} will frequently simplify the presentation and the proofs. Thanks to our general assumptions \eqref{bound_Sigma} and \eqref{Sigma_gap}, it will always be relatively easy to remove \eqref{T_diag}. In particular, we emphasize that the assumption $\Sigma > 0$ is purely qualitative in nature, and is made in order to simplify expressions involving the inverse of $\Sigma$. The case of $\Sigma \geq 0$ may always be easily obtained by considering $\Sigma + \epsilon I_M$ and then taking $\epsilon \downarrow 0$ at fixed $N$. We refer to Section \ref{sec:gen_T} below for the details on how to remove the assumption \eqref{T_diag}.

To avoid repetition, we summarize our basic assumptions for future reference.
\begin{assumption} \label{ass:main}
We suppose that \eqref{MsimN}, \eqref{cond on entries of X}, \eqref{moments of X-1}, \eqref{bound_Sigma}, and \eqref{Sigma_gap} hold.
\end{assumption}

\subsection{Asymptotic eigenvalue density} \label{sec:structure of rho}

Next, we define the asymptotic eigenvalue density of $X^* T^* T X $, $\varrho$, via its Stieltjes transform, $m$. Let $\bb C_+$ denote the complex upper-half plane. 
\begin{lemma}
Let $\pi$ be a compactly supported probability measure on $\R$, and let $\phi > 0$. Then for each $z \in \bb C_+$ there is a unique $m \equiv m(z) \in \bb C_+$ satisfying
\begin{equation} \label{m_sc_eq}
\frac{1}{m} \;=\; -z + \phi \int \frac{x}{1 + m x} \, \pi(\dd x)\,.
\end{equation}
Moreover, $m(z)$ is the Stieltjes transform of a probability measure $\varrho$ with bounded support in $[0,\infty)$.
\end{lemma}
\begin{proof}
This is a well-known result, which may be proved using a fixed point argument in $\C_+$; the function $m(z)$ is the Stieltjes transform of the multiplicative free convolution of $\pi$ and the Marchenko-Pastur law with dimensional ratio \eqref{def_phi}. See e.g.\ \cite[Section 5]{BS4} and \cite{BSbook} for more details.
\end{proof}

\begin{definition}[Asymptotic density] \label{def:m}
We define the deterministic function $m \equiv m_{\Sigma, N} \col \bb C_+ \to \bb C_+$ as the unique solution $m(z)$ of \eqref{m_sc_eq} with $\phi$ defined in \eqref{def_phi} and $\pi$ defined in \eqref{def_pi}.
We denote by $\varrho \equiv \varrho_{\Sigma, N}$ the probability measure associated with $m$, and call it the asymptotic eigenvalue density.
\end{definition}

The rest of this subsection is devoted to a discussion of the basic properties of the asymptotic eigenvalue density $\varrho$. Much this discussion is well known; see e.g.\ \cite{BS2, SC, HHN}. The reader interested only in the principal components of $Q$, i.e.\ its top eigenvalues, can skip this subsection and proceed directly to the results in Theorem \ref{thm:edge} (with $k = 1$) and Corollary \ref{cor:TWA} (i).

Let $n \deq \abs{\supp \pi \setminus \{0\}}$ be the number of distinct nonzero eigenvalues of $\Sigma$, and write
\begin{equation*}
\supp \pi \setminus \{0\} \;=\; \{s_1 > s_2 > \dots > s_n\}\,.
\end{equation*}
Let $z \in \C_+$. Then $m \equiv m(z)$ may also be characterized as the unique solution of the equation
\begin{equation} \label{def_m}
z \;=\; f(m) \,, \qquad \im m > 0\,,
\end{equation}
where we defined
\begin{equation} \label{def_fx}
f(x) \;\deq\; - \frac{1}{x} + \phi \sum_{i = 1}^n \frac{\pi (\{s_i\})}{x + s_i^{-1}}\,.
\end{equation}
In Figures \ref{fig:plot_f} and \ref{fig:plot_f2}, we illustrate the graph of $f$ for the cases $\phi < 1$ and $\phi > 1$ respectively.
In Figure \ref{fig:density} we plot the density of $\varrho$ for the examples from Figures \ref{fig:plot_f} and \ref{fig:plot_f2}.
The behaviour of $\varrho$ may be entirely understood by an elementary analysis of $f$.

\begin{figure}[ht!]
\begin{center}
{\small 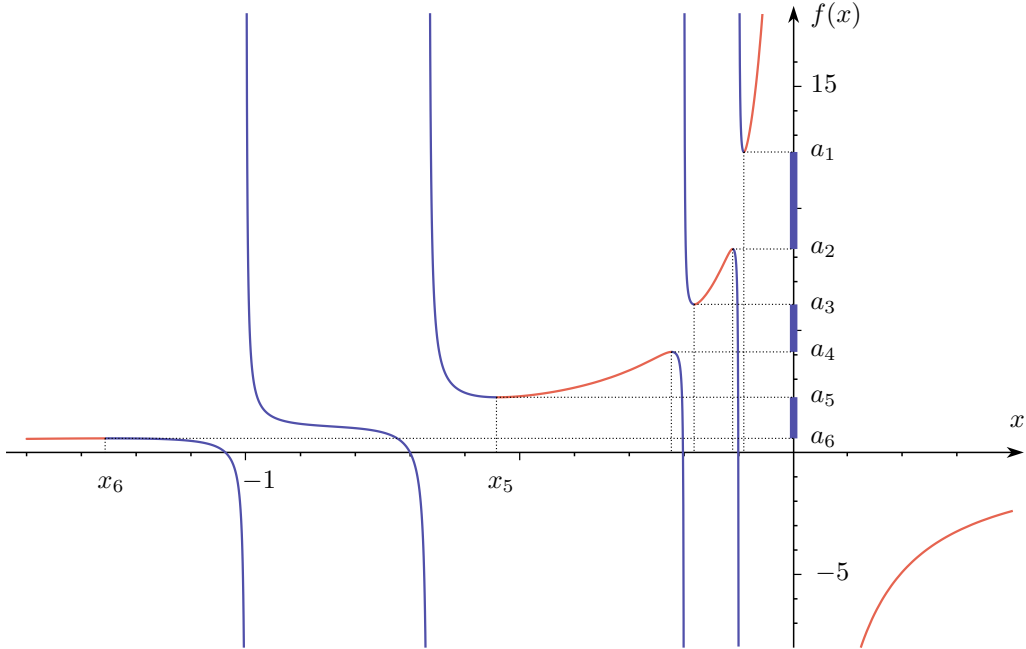}
\end{center}
\caption{The function $f(x)$ for $\phi \pi = 0.01 \, \delta_{10} + 0.01 \, \delta_{5} + 0.05 \, \delta_{1.5} + 0.03 \, \delta_1$.
Here, $\phi = 0.1$ and we have $p = 3$ connected components. The vertical asymptotes are located at $- \sigma^{-1}$ for $\sigma \in \supp \pi$. The support of $\varrho$ is indicated with thick blue lines on the vertical axis. The inverse of $m \vert_{\R \setminus \supp \varrho}$ is drawn in red. \label{fig:plot_f}}
\end{figure}

\begin{figure}[ht!]
\begin{center}
{\small 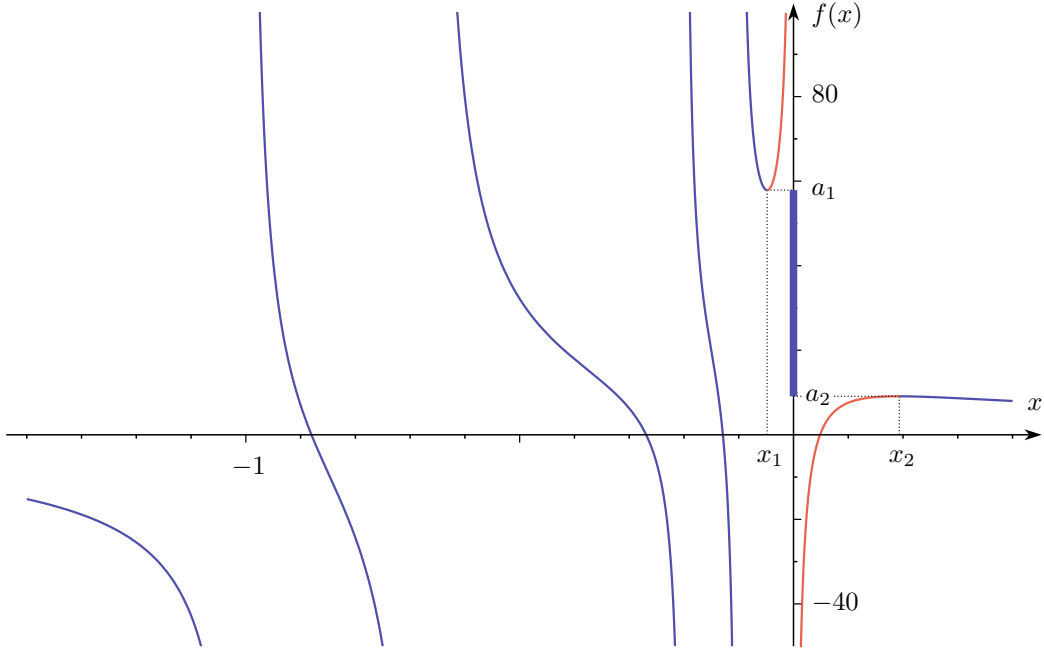}
\end{center}
\caption{The function $f(x)$ for $\phi \pi = \delta_{10} + \delta_{5} + 5 \, \delta_{1.5} + 3 \, \delta_1$.
Here, $\phi = 10$ and we have $p = 1$ connected component. The vertical asymptotes are located at $- \sigma^{-1}$ for $\sigma \in \supp \pi$. The support of $\varrho$ is indicated with a thick blue line. The inverse of $m \vert_{\R \setminus \supp \varrho}$ is drawn in red. \label{fig:plot_f2}}
\end{figure}

\begin{figure}[ht!]
\begin{center}
{\small 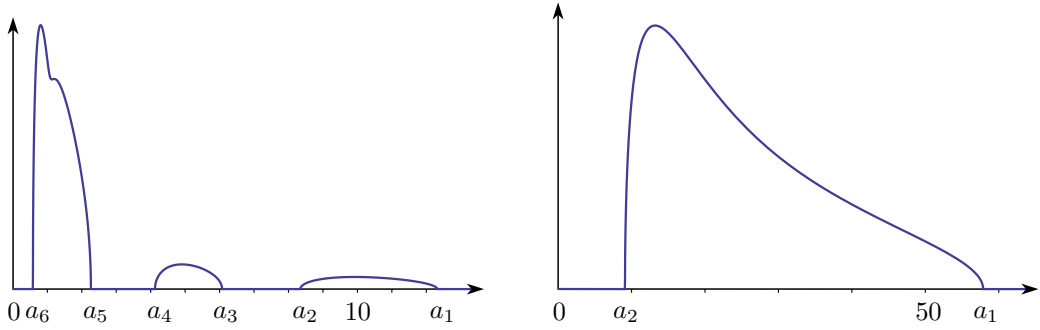}
\end{center}
\caption{The densities of $\varrho$ for the examples from Figures \ref{fig:plot_f} (left) and \ref{fig:plot_f2} (right). \label{fig:density}}
\end{figure}

It is convenient to extend the domain of $f$ to the real projective line $\ol \R = \R \cup \{\infty\} \cong S^1$. Clearly, $f$ is smooth on the $n+1$ open intervals of $\ol \R$ defined through
\begin{equation*}
\qquad I_1 \;\deq\; (-s_1^{-1},0)\,, \qquad I_i \;\deq\; (-s_i^{-1}, -s_{i-1}^{-1}) \quad (i = 2, \dots, n)\,, \qquad I_0 \;\deq\; \ol \R \setminus \bigcup_{i = 1}^n \ol I_i\,.
\end{equation*}
Next, we introduce the multiset $\cal C \subset \ol \R$ of critical points of $f$, using the conventions that a nondegenerate critical point is counted once and a degenerate critical point twice, and if $\phi = 1$ then $\infty$ is a nondegenerate critical point. The following three elementary lemmas are proved in Appendix \ref{sec:varrho}.

\begin{lemma}[Critical points] \label{lem:crit points}
We have $\abs{\cal C \cap I_0} = \abs {\cal C \cap I_1}= 1$ and $\abs{\cal C \cap I_i} \in \{0,2\}$ for $i = 2, \dots, n$.
\end{lemma}

We deduce from Lemma \ref{lem:crit points} that $\abs{\cal C} = 2p$ is even. We denote by $x_1 \geq x_2 \geq \cdots \geq x_{2p - 1}$ the $2p - 1$ critical points in $I_1 \cup \dots \cup I_n$, and by $x_{2p}$ the unique critical point in $I_0$. For $k = 1, \dots, 2p$ we define the critical values $a_k \deq f(x_k)$.

\begin{lemma}[Ordering of the critical values] \label{lem:a_order}
We have $a_1 \geq a_2 \geq \cdots \geq a_{2p}$. Moreover, for $k = 1, \dots, 2p$ we have $a_k = f(x_k)$ and $x_k = m(a_k)$, where $m$ is defined on $\R$ by continuous extension from $\C_+$, and we use the convention $m(0) \deq \infty$ for $\phi = 1$.
Finally, under the assumptions \eqref{phi_bounded} and \eqref{bound_Sigma} there exists a constant $C$ (depending on $\tau$) such that $a_k \in [0,C]$ for $k = 1, \dots, 2p$.
\end{lemma}

The following result gives the basic structure of $\varrho$.
\begin{lemma}[Structure of $\varrho$] \label{lem:rho_struct}
We have
\begin{equation} \label{rho_supp}
\supp \varrho \cap (0,\infty)\;=\; \pBB{\bigcup_{k = 1}^p [a_{2k}, a_{2k - 1}]} \cap (0,\infty)\,.
\end{equation}
\end{lemma}

Our assumptions on $\pi$ (i.e.\ on the spectrum of $\Sigma$) take the form of the following regularity conditions.

\begin{definition}[Regularity] \label{def:regular}
Fix $\tau > 0$.
\begin{enumerate}
\item
We say that the edge $k = 1, \dots, 2p$ is \emph{regular} if
\begin{equation} \label{reg: edge}
a_k \;\geq\; \tau \,, \qquad \min_{l \neq k} \abs{a_k - a_l} \;\geq\; \tau\,, \qquad \min_{i} \abs{x_k + s_i^{-1}} \;\geq\; \tau\,.
\end{equation}
\item
We say that the bulk component $k = 1, \dots, p$ is \emph{regular} if for any fixed $\tau' > 0$ there exists a constant $c \equiv c_{\tau, \tau'} > 0$ such that the density\footnote{This means the density of the absolutely continuous part of $\varrho$. In fact, as explained after Lemma \ref{lem:gen_prop_m} below, $\varrho$ is always absolutely continuous on $\R \setminus \{0\}$.} of $\varrho$ in $[a_{2k} + \tau', a_{2k - 1} - \tau']$ is bounded form below by $c$.
\end{enumerate}
\end{definition}

The edge regularity condition from Definition \ref{def:regular} (i) has previously appeared (in a slightly different form) in several works on sample covariance matrices. For the rightmost edge $k = 1$, it was introduced in \cite{ElK2} and was subsequently used in the works \cite{SL, Ona, BaoPanZhou} on the distribution of eigenvalues near the top edge $a_1$. For general $k$, it was introduced in \cite{HHN}. The second condition of \eqref{reg: edge} states that the gap in the spectrum of $\varrho$ adjacent to the edge $a_k$ does not close for large $N$; the third condition of \eqref{reg: edge} ensures a regular square-root behaviour of the spectral density near $a_k$, and in particular rules out outliers.

The bulk regularity condition from Definition \ref{def:regular} (ii) imposes a lower bound on the density of eigenvalues away from the edges. Without it, one can have points in the interior of $\supp \varrho$ with an arbitrarily small density, and our results can be shown to be false. Such points may be easily constructed by taking $\pi$ with a degenerate critical point $x$, which corresponds to two components of $\varrho$ touching at $a = f(x)$. Then it is not hard to check that one can construct an arbitrarily small perturbation $\tilde \pi$ of $\pi$ such that the point $a$ is well-separated from the edges of the support of the associated density $\tilde \varrho$, and the density $\tilde \varrho$ at $a$ is an arbitrarily small positive number. The condition from Definition \ref{def:regular} (ii) is stable under perturbation of $\pi$; see Remark \ref{rem:bulk_stab} below.

We conclude this subsection with a couple of examples verifying the regularity conditions of Definition \ref{def:regular}.

\begin{example}[Bounded number of distinct eigenvalues]
We suppose that $n$ is fixed, and that $s_1, \dots, s_n$ and $\phi \pi(\{s_1\}), \dots, \phi \pi(\{s_n\})$ all converge in $(0,\infty)$ as $N \to \infty$.
We suppose that all critical points of $\lim_N f$ are nondegenerate, and that  $\lim_N a_i > \lim_N a_{i+1}$ for $i = 1, \dots, 2p$. Then it is easy to check that, for small enough $\tau$, all edges $k = 1, \dots, 2p$ and bulk components $k = 1, \dots, p$ are regular in the sense of Definition \ref{def:regular}.
\end{example}

\begin{example}[Continuous limit]
We suppose that  $\phi$ converges in $(0,\infty) \setminus \{1\}$. Moreover, we suppose that $\pi$ is supported in some interval $[a, b] \subset (0,\infty)$, and that $\pi$ converges in distribution to some measure $\pi_\infty$ that is absolutely continuous and whose density satisfies $\tau \leq \dd \pi_\infty (E) / \dd E \leq \tau^{-1}$ for $E \in [a,b]$.  It is easy to check that in this case $p=1$, and that the two edges and the single bulk component are regular in the sense of Definition \ref{def:regular}.
\end{example}

\section{Results}

In this section we state our main results for the sample covariance matrix $Q$ defined in Section \ref{sec:model}. To ease readability, we state the anisotropic local law under three different sets of increasingly weak assumptions. In Section \ref{sec:res_full}, we assume that all edges and bulk components are regular in the sense of Definition \ref{def:regular}. In Section \ref{sec:res_detail}, we assume only regularity of a single edge or bulk component, and state the anisotropic local law in the vicinity of the corresponding edge or bulk component. As an application, we prove the edge universality near a regular edge. Finally, in Section \ref{sec:general_results} we give a general anisotropic law, where the concrete regularity assumptions from Definition \ref{def:regular} are replaced with a more general but abstract stability condition. The reader interested only in the principal components of $Q$ can proceed directly to the results in Theorem \ref{thm:edge} (with $k = 1$) and Corollary \ref{cor:TWA} (i).

\subsection{Basic definitions}
In this preliminary subsection we introduce some basic notations and definitions.
Our main results take the form of \emph{local laws}, which relate the resolvents
\begin{equation} \label{def_RM_RN}
R_N(z) \;\deq\; (X^* T^* T X - z)^{-1} \qquad \txt{and} \qquad R_M(z) \;\deq\; (T X X^* T^* - z)^{-1}
\end{equation}
to the Stieltjes transform $m$ of the asymptotic density $\varrho$. 
These local laws may be formulated in a simple, unified fashion under the assumption \eqref{T_diag} using an $(M + N) \times (M + N)$ block matrix, which is a linear function of $X$.

\begin{definition}[Index sets]
We introduce the index sets
\begin{equation*}
\cal I_M \;\deq\; \qq{1, M}\,, \qquad \cal I_N \;\deq\; \qq{M+1, M+N}\,, \qquad \cal I \;\deq\; \cal I_M \cup \cal I_N \;=\; \qq{1, M+N}\,.
\end{equation*}
We consistently use the letters $i,j \in \cal I_M$, $\mu,\nu \in \cal I_N$, and $s,t \in \cal I$. We label the indices of the matrices according to
\begin{equation*}
X \;=\; (X_{i \mu} \col i \in \cal I_M, \mu \in \cal I_N) \,, \qquad \Sigma \;=\; (\Sigma_{ij} \col i,j \in \cal I_M)\,.
\end{equation*}
\end{definition}

\begin{definition}[Linearizing block matrix]
Under the condition \eqref{T_diag} and for $z \in \bb C_+$ we define the $\cal I \times \cal I$ matrix
\begin{equation} \label{def_G}
G(z) \;\equiv\; G^\Sigma(X, z) \;\deq\; \begin{pmatrix}
-\Sigma^{-1} & X
\\
X^* & -z
\end{pmatrix}^{-1}\,.
\end{equation}
\end{definition}

The motivation behind this definition is that, assuming \eqref{T_diag}, a control of $G$ immediately yields control of the resolvents $R_N$ and $R_M$
via the identities
\begin{equation} \label{G_M_ident}
G_{ij} \;=\; \p{z \Sigma^{1/2} R_M \Sigma^{1/2}}_{ij}
\end{equation}
for $i,j \in \cal I_M$ and
\begin{equation} \label{G_N_ident}
G_{\mu \nu} \;=\; (R_N)_{\mu \nu}
\end{equation}
for $\mu,\nu \in \cal I_N$. Both of these identities may be easily checked using Schur's complement formula.

Next, we introduce a deterministic matrix $\Pi$, which we shall prove is close to $G$ with high probability (in the sense of Definition \ref{def:stocdom} (iii) below).
\begin{definition}[Deterministic equivalent of $G$]
For $z \in \bb C_+$ we define the $\cal I \times \cal I$ deterministic matrix
\begin{equation} \label{def_Pi}
\Pi(z) \;\equiv\; \Pi^\Sigma(z) \;\deq\;
\begin{pmatrix}
-\Sigma(1 + m(z) \Sigma)^{-1} & 0
\\
0 & m(z)
\end{pmatrix}\,.
\end{equation}
\end{definition}
We also extend $\Sigma$ to an $\cal I \times \cal I$ matrix
\begin{equation} \label{def_ul_Sigma}
\ul \Sigma \;\deq\; \begin{pmatrix}
\Sigma & 0
\\
0 & 1
\end{pmatrix}\,.
\end{equation}

We consistently use the notation $z = E + \ii \eta$
for the spectral parameter $z$. Throughout the following we regard the quantities $E(z)$ and $\eta(z)$ as functions of $z$ and usually omit the argument unless it is needed to avoid confusion. The spectral parameter $z$ will always lie in the fundamental domain
\begin{equation} \label{def_D}
\f D \;\equiv\; \f D(\tau,N) \;\deq\; \hb{z \in \bb C_+ \col \abs{z} \geq \tau \,,\, \abs{E} \leq \tau^{-1} \,,\, N^{-1 + \tau} \leq \eta \leq \tau^{-1}}\,,
\end{equation}
where $\tau > 0$ is a fixed parameter.

The following notion of a high-probability bound was introduced in \cite{EKY2}, and has been subsequently used in a number of works on random matrix theory. It provides a simple way of systematizing and making precise statements of the form ``$\xi$ is bounded with high probability by $\zeta$ up to small powers of $N$''.

\begin{definition}[Stochastic domination]\label{def:stocdom}
\begin{enumerate}
\item
Let
\begin{equation*}
\xi \;=\; \pb{\xi^{(N)}(u) \col N \in \N, u \in U^{(N)}} \,, \qquad
\zeta \;=\; \pb{\zeta^{(N)}(u) \col N \in \N, u \in U^{(N)}}
\end{equation*}
be two families of nonnegative random variables, where $U^{(N)}$ is a possibly $N$-dependent parameter set. 

We say that $\xi$ is \emph{stochastically dominated by $\zeta$, uniformly in $u$,} if for all (small) $\epsilon > 0$ and (large) $D > 0$ we have
\begin{equation*}
\sup_{u \in U^{(N)}} \P \qB{\xi^{(N)}(u) > N^\epsilon \zeta^{(N)}(u)} \;\leq\; N^{-D}
\end{equation*}
for large enough $N\geq N_0(\e, D)$.
Throughout this paper the stochastic domination will always be uniform in all parameters (such as matrix indices, deterministic vectors, and spectral parameters $z \in \f D$) that are not explicitly fixed. Note that $N_0(\e, D)$ may depend on quantities that are explicitly constant, such as $\tau$ and $C_p$ from \eqref{moments of X-1}.
\item
If $\xi$ is stochastically dominated by $\zeta$, uniformly in $u$, we use the notation $\xi \prec \zeta$. Moreover, if for some complex family $\xi$ we have $\abs{\xi} \prec \zeta$ we also write $\xi = O_\prec(\zeta)$.
\item
We extend the definition of $O_\prec(\,\cdot\,)$ to matrices in the weak operator sense as follows. Let $A$ be a family of complex square random matrices and $\zeta$ a family of nonnegative random variables. Then we use $A = O_\prec(\zeta)$ to mean $\abs{\scalar{\f v}{A \f w}} \prec \zeta \abs{\f v}\abs{\f w}$ uniformly for all deterministic vectors $\f v$ and $\f w$.
\end{enumerate}
\end{definition}

\begin{remark}
Because of \eqref{MsimN}, all (or some) factors of $N$ in Definition \ref{def:stocdom} could be replaced with $M$ or $\wh M$ without changing the definition of stochastic domination. 
\end{remark}

\subsection{The full spectrum} \label{sec:res_full}

As explained at the beginning of this section, we first state the anisotropic local law for all arguments $z$ under the assumption that all edges and bulk components are regular in the sense of Definition \ref{def:regular}. In the next subsection, we relax these assumptions by restricting the domain of $z$ to the vicinity of an edge or bulk component, and only requiring the regularity of the corresponding edge or bulk component.

Our main result is the following anisotropic local law. We introduce the fundamental control parameter
\begin{equation} \label{def_Psi}
\Psi(z) \;\deq\; \sqrt{\frac{\im m(z)}{N \eta}} + \frac{1}{N \eta}
\end{equation}
and the Stieltjes transform of the empirical eigenvalue density of $X^* T^* T X$,
\begin{equation} \label{def_m_N}
m_N(z) \;\deq\; \frac{1}{N} \sum_{\mu \in \cal I_N} (R_N)_{\mu \mu}(z)\,.
\end{equation}

\begin{theorem}[Anisotropic local law] \label{thm:LL_G}
Fix $\tau > 0$. Suppose that \eqref{T_diag} and Assumption \ref{ass:main} hold. Suppose moreover that every edge $k = 1, \dots, 2p$ satisfying $a_k \geq \tau$ and every bulk component $k = 1, \dots, p$ is regular in the sense of Definition \ref{def:regular}. Then
\begin{equation} \label{edge_G_result}
\ul \Sigma^{-1} \pb{G(z) - \Pi(z)} \ul \Sigma^{-1} \;=\; O_\prec(\Psi(z))
\end{equation}
and
\begin{equation} \label{edge_G_avg_result}
m_N(z) - m(z) \;=\; O_\prec \pbb{\frac{1}{N \eta}}
\end{equation}
uniformly in $z \in \f D$.
\end{theorem}
(Note that the presence of the factors $\ul \Sigma$ in \eqref{edge_G_result} strengthens the result; by \eqref{bound_Sigma}, they can be trivially dropped to obtain a weaker estimate. They ensure that the control of the error is stronger in directions where the covariance $\Sigma$ is small.)

Outside the support of the asymptotic spectrum, one has stronger control all the way down to the real axis.

\begin{theorem}[Anisotropic local law outside of the spectrum] \label{thm:G_outside}
Fix $\tau > 0$. Suppose that \eqref{T_diag} and Assumption \ref{ass:main} hold. Suppose moreover that every edge $k = 1, \dots, 2p$ satisfying $a_k \geq \tau$ is regular in the sense of Definition \ref{def:regular} (i). Then
\begin{equation} \label{G_outside}
\ul \Sigma^{-1} \pb{G(z) - \Pi(z)} \ul \Sigma^{-1} \;=\; O_\prec \pBB{\sqrt{\frac{\im m(z)}{N \eta}}}
\end{equation}
uniformly in $z \in [\tau, \tau^{-1}] \times (0, \tau^{-1}]$ satisfying $\dist(E, \supp \varrho) \geq N^{-2/3 + \tau}$.
\end{theorem}

An explicit expression for the error term in \eqref{G_outside} may be obtained from \eqref{m_sq_root} below.

\begin{remark}
Theorem \ref{thm:G_outside} can be used to obtain a complete picture of the \emph{outlier eigenvalues} of $Q$ in the case where a bounded number of eigenvalues of $\Sigma$ are changed to some arbitrary values (in particular possibly violating the regularity assumption from Definition \ref{def:regular} (i)). We note that the outliers may also lie between bulk components of $\varrho$. The analysis is similar to the one performed in \cite{KYY} for the case $\Sigma = I_M$; we omit the details.
\end{remark}

In the remainder of this subsection, we state several corollaries of Theorem \ref{thm:LL_G} where the assumption \eqref{T_diag} is removed. From Theorem \ref{thm:LL_G} it is not hard to deduce the following result on the resolvents $R_N$ and $R_M$, defined in \eqref{def_RM_RN}.

\begin{corollary}[Anisotropic local laws for $X^* T^* T X$ and $T X X^* T^*$] \label{cor:GM_GN}
Fix $\tau > 0$. Suppose that Assumption \ref{ass:main} holds. Suppose moreover that every edge $k = 1, \dots, 2p$ satisfying $a_k \geq \tau$ and every bulk component $k = 1, \dots, p$ is regular in the sense of Definition \ref{def:regular}. Then
\begin{equation} \label{R_N_edge}
R_N - m(z) \;=\; O_\prec(\Psi(z))
\end{equation}
and 
\begin{equation} \label{R_M_edge}
\Sigma^{-1/2} \pbb{R_M - \frac{-1}{z (1 + m(z) \Sigma)}} \Sigma^{-1/2} \;=\; O_\prec(\Psi(z))
\end{equation}
uniformly in $z \in \f D$.
Moreover, \eqref{edge_G_avg_result} holds uniformly in $z \in \f D$.
\end{corollary}

\begin{remark}
Theorem \ref{thm:G_outside} has an analogous corollary for $R_N$ and $R_M$ for $z$ satisfying $\dist(E, \supp \varrho) \geq N^{-2/3 + \tau}$, whereby the right-hand sides of \eqref{R_N_edge} and \eqref{R_M_edge} are replaced with the right-hand side of \eqref{G_outside}; we omit the precise statement.
\end{remark}

\begin{remark}
The following consistency check may be applied to the deterministic matrices on the left-hand sides of \eqref{R_N_edge} and \eqref{R_M_edge}.
If, in the identity $\frac{1}{M} \tr R_M = \frac{1}{M} \tr R_N + \frac{\phi - 1}{\phi} \frac{1}{z}$, we replace $\tr R_M$ and $\tr R_N$ with the corresponding deterministic matrices from the left-hand sides of \eqref{R_N_edge} and \eqref{R_M_edge}, we recover \eqref{def_m}.
\end{remark}

We conclude this subsection with another consequence of Theorem \ref{thm:LL_G} -- eigenvalue rigidity. We denote by
\begin{equation*}
\lambda_1 \;\geq\; \lambda_2 \;\geq\; \cdots \;\geq\; \lambda_{M \wedge N}
\end{equation*}
the nontrivial eigenvalues of $Q$, and by $\gamma_1 \geq \gamma_2 \geq \cdots \geq \gamma_{M \wedge N}$ the \emph{classical eigenvalue locations} defined through
\begin{equation} \label{def:gamma_alpha}
N \int_{\gamma_i}^\infty \dd \varrho \;=\; i-1/2\,.
\end{equation}
If $p \geq 1$, it is convenient to relabel $\lambda_i$ and $\gamma_i$ separately for each bulk component $k = 1, \dots, p$. To that end, for $k = 1, \dots, p$ we define the \emph{classical number of eigenvalues in the $k$-th bulk component} through
\begin{equation} \label{def_N_k}
N_k \;\deq\; N \int_{a_{2k}}^{a_{2k - 1}} \dd \varrho\,.
\end{equation}
In Lemma \ref{lem:N_i} below we prove that $N_k \in \N$ for $k = 1, \dots, p$. For $k = 1, \dots, p$ and $i = 1, \dots, N_k$ we introduce the relabellings
\begin{equation*}
\lambda_{k,i} \;\deq\; \lambda_{i + \sum_{l < k} N_l}\,, \qquad \gamma_{k,i} \;\deq\; \gamma_{i + \sum_{l < k} N_l} \;\in\; (a_{2k}, a_{2k - 1})\,.
\end{equation*}
Note that we may also characterize $\gamma_{k,i}$ through $N \int_{\gamma_{k,i}}^{a_{2k-1}} \dd \varrho = i-1/2$.

\begin{theorem}[Eigenvalue rigidity] \label{thm:rigidity}
Fix $\tau > 0$. Suppose that Assumption \ref{ass:main} holds. Suppose moreover that every edge $k = 1, \dots, 2p$ satisfying $a_k \geq \tau$ and every bulk component $k = 1, \dots, p$ is regular in the sense of Definition \ref{def:regular}. Then we have for all $k = 1, \dots, p$ and $i = 1, \dots, N_k$ satisfying $\gamma_{k,i} \geq \tau$ that
\begin{equation} \label{rigidity}
\abs{\lambda_{k,i} - \gamma_{k,i}} \;\prec\; \pb{i \wedge (N_k + 1 - i)}^{-1/3} N^{-2/3}\,.
\end{equation}
\end{theorem}

\begin{remark}
As in \cite[Theorem 2.8]{BEKYY}, Corollary \ref{cor:GM_GN} and Theorem \ref{thm:rigidity} imply the complete delocalization, with respect to an arbitrary deterministic basis, of the eigenvectors of $TXX^*T^*$ and $X^* T^* T X$ associated with eigenvalues $\lambda_i$ satisfying $\gamma_i \geq \tau$.
\end{remark}

Note that Theorem \ref{thm:rigidity} in particular implies an exact separation of the eigenvalues into connected components, whereby the number of eigenvalues in the $k$-th connected component is with high probability equal to the deterministic number $N_i$. This phenomenon of exact separation was first established in \cite{BS2, BS3}.

\subsection{Individual spectral regions and edge universality} \label{sec:res_detail}
In this subsection we elaborate on the results of Subsection \ref{sec:res_full} by only requiring the regularity of the edge or bulk component near which the spectral parameter $z$ lies. To that end, for fixed $\tau ,\tau' > 0$ we define the subdomains
\begin{align*}
\f D^e_k &\;\equiv\; \f D^e_k(\tau, \tau', N) \;\deq\; \hb{z \in \f D(\tau, N) \col E \in [a_k - \tau', a_k + \tau'] } \qquad (k = 1, \dots, 2p)\,,
\\
\f D^b_k &\;\equiv\; \f D^b_k(\tau, \tau', N) \;\deq\; \hb{z \in \f D(\tau, N) \col E \in [a_{2k} + \tau', a_{2k - 1} - \tau']} \qquad (k = 1, \dots, p)\,,
\\
\f D^o &\;\equiv\; \f D^o(\tau, \tau', N) \;\deq\; \hb{z \in \f D(\tau, N) \col \dist(E, \supp \varrho) \geq \tau'}\,.
\end{align*}
The superscripts $e$, $b$, and $o$ stand for ``edge'', ``bulk'', and ``outside'', respectively.

Roughly, the results from Section \ref{sec:res_full} hold in each one of these smaller domains (instead of the full domain $\f D$) provided that just the associated edge (for $z \in \f D^e_k$) or bulk component (for $z \in \f D^b_k$) is regular. Results outside of the spectrum (for $z \in \f D^o$) do not need any regularity assumption.

\begin{theorem}[Regular edge] \label{thm:edge}
Fix $\tau > 0$. Suppose that Assumption \ref{ass:main} holds. Suppose that the edge $k = 1, \dots, 2p$ is regular in the sense of Definition \ref{def:regular} (i). Then there exists a constant $\tau' > 0$, depending only on $\tau$, such that the following holds.
\begin{enumerate}
\item
Under the assumption \eqref{T_diag} the estimates \eqref{edge_G_result} and \eqref{edge_G_avg_result} hold uniformly in $z \in \f D^e_k$.
\item
The estimates \eqref{R_N_edge} and \eqref{R_M_edge} hold uniformly in $z \in \f D^e_k$.
\item
Let $\hat k \deq \floor{(k+1)/2}$ be the bulk component to which the edge $k$ belongs. Then for all $i = 1, \dots, N_{\hat k}$ satisfying $\gamma_{\hat k, i} \in [a_k - \tau', a_k + \tau']$ we have
\begin{equation*}
\abs{\lambda_{\hat k,i} - \gamma_{\hat k,i}} \;\prec\; \pb{i \wedge (N_{\hat k} + 1 - i)}^{-1/3} N^{-2/3}\,.
\end{equation*}
\end{enumerate}
\end{theorem}

\begin{theorem}[Regular bulk] \label{thm:bulk}
Fix $\tau,\tau'>0$. Suppose that Assumption \ref{ass:main} holds. Suppose that the bulk component $k = 1, \dots, 2p$ is regular in the sense of Definition \ref{def:regular} (ii).
\begin{enumerate}
\item
Under the assumption \eqref{T_diag} the estimates \eqref{edge_G_result} and \eqref{edge_G_avg_result} hold uniformly in $z \in \f D^b_k$.
\item
The estimates \eqref{R_N_edge} and \eqref{R_M_edge} hold uniformly in $z \in \f D^b_k$.
\item
Suppose that at least one of the two edges $2k$ and $2k-1$ is regular in the sense of Definition \ref{def:regular} (i). Then for all $i = 1, \dots, N_{k}$ satisfying $\gamma_{k,i} \in [a_{2k} + \tau', a_{2k - 1} - \tau']$ we have $\abs{\lambda_{k,i} - \gamma_{k,i}} \prec N^{-1}$.
\end{enumerate}
\end{theorem}

\begin{theorem}[Outside of the spectrum] \label{thm:outside}
Fix $\tau,\tau'>0$. Suppose that Assumption \ref{ass:main} holds.
\begin{enumerate}
\item
Under the assumption \eqref{T_diag} the estimates \eqref{edge_G_result} and \eqref{edge_G_avg_result} hold uniformly in $z \in \f D^o$.
\item
The estimates \eqref{R_N_edge} and \eqref{R_M_edge} hold uniformly in $z \in \f D^o$.
\end{enumerate}
\end{theorem}

\begin{remark}
As in Section \ref{sec:res_full}, if $\dist(E, \supp \varrho) \geq N^{-2/3 + \tau}$ then the error parameters $\Psi(z)$ on the right-hand sides of \eqref{edge_G_result}, \eqref{R_N_edge}, and \eqref{R_M_edge} in (i) and (ii) of Theorems \ref{thm:edge} and \ref{thm:outside} can be replaced with the smaller quantity $\sqrt{\im m(z) /(N \eta)}$ and the lower bound $\eta \geq N^{-1 + \tau}$ relaxed to $\eta > 0$. (See Theorem \ref{thm:G_outside}.) We omit the detailed statement.
\end{remark}

Finally, as a concrete application of the anisotropic local law, we establish the edge universality of $Q$. For a regular edge $k = 1, \dots, 2p$ we define $\varpi_k \deq (\abs{f''(x_k)}/2)^{1/3}$, which has the interpretation of the curvature of the density of $\varrho$ near the edge $k$. In Lemma \ref{lem:edge_tools} below we prove that if $k$ is a regular edge then $\varpi_k \asymp 1$. For any fixed $l \in \N$ and bulk component $k = 1, \dots, p$ we define the random vectors
\begin{align*}
\f q_{2k - 1,l} &\;\deq\; \frac{N^{2/3}}{\varpi_{2k - 1}} \pb{\lambda_{k,1} - a_{2k - 1}\,,\dots, \,\lambda_{k,l} - a_{2k - 1}}\,,
\\
\f q_{2k,l} &\;\deq\; -\frac{N^{2/3}}{\varpi_{2k}} \pb{\lambda_{k,N_k} - a_{2k}\,, \dots,\, \lambda_{k,N_k - l+1} - a_{2k}}\,.
\end{align*}
The interpretation of $\f q_k$ is the appropriately centred and normalized family of $l$ extreme eigenvalues of $Q$ near the edge $k$.

\begin{theorem}[Edge universality] \label{thm:edge_univ}
Fix $\tau > 0$ and $l \in \N$. Suppose that Assumption \ref{ass:main} holds. Suppose that the edge $k = 1, \dots, 2p$ is regular in the sense of Definition \ref{def:regular} (i). Then the distribution of $\f q_{k,l}$ depends asymptotically only on $\pi$. More precisely, for any fixed continuous bounded function $h \in C_b(\R^l)$ there exists $b(h,\pi) \equiv b_N(h,\pi)$, depending only on $\pi$, such that $\lim_{N \to \infty} (\E h(\f q_{k,l}) - b(h,\pi)) = 0$.
\end{theorem}

Theorem \ref{thm:edge_univ} says in particular that the asymptotic distribution of $\f q_{k,l}$ does not depend on the distribution of the entries of $X$, on the left and right singular vectors of $T$, or on the dimensions of $T$. In particular, we may choose the entries of $X$ to be Gaussian and $T = T^*$ to be diagonal. We remark that Theorem \ref{thm:edge_univ} also holds for the joint distribution at several regular edges, i.e.\ for the random vector $(\f q_{k,l})_{k \in K}$, where each $k \in K$ is a regular edge.

Combining Theorem \ref{thm:edge_univ} with the works \cite{HHN, ElK2, Ona, SL} on the Gaussian case, we deduce convergence of $\f q_{k,l}$ to the \emph{Tracy-Widom-Airy random vector} $\f q_l^\beta$, where $\beta = 1,2$ is the customary symmetry index of random matrix theory, equal to $1$ in the real symmetric case and $2$ in the complex Hermitian case   (see Section \ref{sec:def_model}). \nc Explicitly, $\f q_l^\beta$ may be expressed as the limit in distribution of the appropriately centred and scaled top $l$ eigenvalues of a GOE/GUE matrix:
\begin{equation*}
N^{2/3} (\mu^\beta_1 - 2, \dots, \mu^\beta_l - 2) \;\todist\; {\f q}_l^\beta \qquad (N \to \infty)\,,
\end{equation*}
where $\mu^\beta_1 \geq \mu^\beta_2 \geq \cdots \geq \mu^\beta_N$ denote the eigenvalues of GOE ($\beta = 1$) or GUE ($\beta = 2$). For instance, $\f q_1^\beta \in \R$ is a random variable with Tracy-Widom-$\beta$ distribution \cite{TW1,TW2}.

\begin{corollary}[Tracy-Widom-Airy statistics near a regular edge] \label{cor:TWA}
Fix $\tau > 0$ and $l \in \N$. Suppose that Assumption \ref{ass:main} holds.
\begin{enumerate}
\item
Suppose that the rightmost edge $k = 1$ is regular in the sense of Definition \ref{def:regular} (i).
Then, for real symmetric $Q$ ($\beta = 1$) or complex Hermitian $Q$ ($\beta = 2$), the random vector $\f q_{1,l}$ converges in distribution to the Tracy-Widom-Airy random vector $\f q_l^\beta$.
\item
Suppose that the edge $k = 1, \dots, 2p$ is regular in the sense of Definition \ref{def:regular} (i). Then, for complex Hermitian $Q$ ($\beta = 2$), the random vector $\f q_{k,l}$ converges in distribution to the Tracy-Widom-Airy random vector $\f q_l^2$.
\end{enumerate}
\end{corollary}

Like Theorem \ref{thm:edge_univ}, Corollary \ref{cor:TWA} (ii) also holds for the joint distribution at several regular edges. In particular, for the case $\beta = 2$ and under the assumption that the top and bottom edges of $\varrho$ are regular, we obtain the universality of the condition number of $Q$.

Corollary \ref{cor:TWA} (i) for $\beta = 1,2$ was previously established in \cite{SL} under the assumption that $\Sigma$ is diagonal, corresponding to uncorrelated population entries. Before that, Corollary \ref{cor:TWA} (i) for $\beta = 2$ and diagonal $\Sigma$ was established in \cite{BaoPanZhou}, following the results of \cite{ElK2, Ona} in the complex Gaussian case. Corollary \ref{cor:TWA} (ii) was recently established for Gaussian $X$ in \cite{HHN}.

\subsection{The general anisotropic local law} \label{sec:general_results}

In this subsection we conclude the statement of our results with a general anisotropic local law, which takes the form of a black box yielding the anisotropic local law for $Q$ assuming it has been established for a much simpler matrix. This latter result may be proved independently. In particular, this black box formulation may be used to establish the anisotropic local law in cases where the regularity assumptions from Definition \ref{def:regular} fail; we do not pursue such generalizations here. Aside from its great generality, this black box formulation also makes precise the three Steps (A) -- (C) mentioned in the introduction, which constitute the basic strategy of our proof.

We begin by introducing some basic terminology.
\begin{definition}[Local laws] \label{def:local_laws}
We call a subset $\f S \equiv \f S(N) \subset \f D(\tau, N)$ a \emph{spectral domain} if for each $z \in \f S$ we have $\h{w \in \f D \col \re w = \re z, \im w \geq \im z} \subset \f S$.

Let $\f S \subset \f D$ be a spectral domain.
\begin{enumerate}
\item
We say that the \emph{entrywise local law holds with parameters $(X, \Sigma, \f S)$} if
\begin{equation} \label{entrywise law}
\pB{\ul \Sigma^{-1} \pb{G(z) - \Pi(z)} \ul \Sigma^{-1}}_{st} \;=\; O_\prec(\Psi(z))
\end{equation}
uniformly in $z \in \f S$ and $s,t \in \cal I$.
\item
We say that the \emph{anisotropic local law holds with parameters $(X, \Sigma, \f S)$} if
\begin{equation} \label{isotropic law}
\ul \Sigma^{-1} \pb{G(z) - \Pi(z)} \ul \Sigma^{-1} \;=\; O_\prec(\Psi(z))
\end{equation}
uniformly in $z \in \f S$.
\item
We say that the \emph{averaged local law holds with parameters $(X, \Sigma, \f S)$} if
\begin{equation*}
m_N(z) - m(z) \;=\; O_\prec \pbb{\frac{1}{N \eta}}
\end{equation*}
uniformly in $z \in \f S$.
\end{enumerate}
\end{definition}

The main conclusion of this paper is that the anisotropic local law holds for general $X$ and $T$ provided that the entrywise local law holds for Gaussian $X$ and  diagonal $T$. This latter case may be established independently, as we illustrate in Section \ref{sec:diag_Sigma} and Appendix \ref{sec:varrho}.

Aside from Assumption \ref{ass:main}, the only assumption that we shall need is
\begin{equation} \label{1+xm}
\abs{1 + m(z) \sigma_i} \;\geq\; \tau \qquad \txt{ for all } z \in \f S \txt{ and } i \in \cal I_M\,.
\end{equation}
This assumption holds for instance under the regularity assumptions of Definition \ref{def:regular} (see Lemmas \ref{lem:stab_edge1}, \ref{lem:stab_bulk}, and \ref{lem:stab_outside} below). Clearly, we always have $\abs{1 + m(z) \sigma_i} > 0$ (see \eqref{def_m}), and \eqref{1+xm} is a uniform version of this bound. Generally, the assumption \eqref{1+xm} is necessary to guarantee that the generalized matrix entries of $(Q - z)^{-1}$ (or, alternatively, of $G(z)$) remain bounded. Indeed, in Corollary \ref{cor:GM_GN} we saw that the generalized entries of $(Q - z)^{-1}$ are close to those of $-z^{-1} (1 + m(z) \Sigma)^{-1}$.

\begin{theorem}[General local law] \label{thm:gen_iso}
Fix $\tau > 0$. Suppose that $X$ and $\Sigma$ satisfy \eqref{T_diag} and Assumption \ref{ass:main}. Let $X^{\txt{Gauss}}$ be a Gaussian matrix satisfying \eqref{cond on entries of X}.  Let $\f S \subset \f D$ be a spectral domain, and suppose that \eqref{1+xm} holds. Define the diagonalization of $\Sigma$ through
\begin{equation} \label{def_D_Sigma}
D \;\equiv\; D(\Sigma) \;\deq\; \diag(\sigma_1, \sigma_2, \dots, \sigma_M)\,.
\end{equation}
\begin{enumerate}
\item
If the entrywise local law holds with parameters $(X^{\txt{Gauss}}, D, \f S)$, then the anisotropic local law holds with parameters $(X, \Sigma, \f S)$.
\item
If the entrywise local law and the averaged local law hold with parameters $(X^{\txt{Gauss}}, D, \f S)$, then the averaged local law holds with parameters $(X, \Sigma, \f S)$.
\end{enumerate}
\end{theorem}

The next theorem shows that the hypotheses in (i) and (ii) of Theorem \ref{thm:gen_iso} may be verified under a \emph{stability condition} on the spectrum of $\Sigma$, made precise in Definition \ref{def:stability} below.

\begin{theorem}[General conditions for local law with diagonal $\Sigma$] \label{thm:diag_gen}
Fix $\tau > 0$. Let $\f S \subset \f D$ be a spectral domain. Suppose that $\Sigma$ is diagonal and that \eqref{T_diag}, \eqref{1+xm}, and Assumption \ref{ass:main} hold. Moreover, suppose that the equation \eqref{def_m} is stable on $\f S$ in the sense of Definition \ref{def:stability} below.
Then the entrywise local law holds with parameters $(X, \Sigma, \f S)$, and the averaged local law holds with parameters $(X, \Sigma, \f S)$.
\end{theorem}

Thus, the Steps (A) -- (C) outlined in the introduction may be expressed as follows. Let $\f S$ be a spectral domain (e.g.\ $\f S = \f D^e_k$ for some edge $k$). First, we prove that \eqref{1+xm} holds and that \eqref{def_m} is stable on $\f S$ in the sense of Definition \ref{def:stability} below. Then we use Theorem \ref{thm:diag_gen} to obtain the entrywise and averaged local laws with parameters $(X^{\txt{Gauss}}, D(\Sigma), \f S)$. We then feed this result into Theorem \ref{thm:gen_iso} to obtain the anisotropic and averaged local laws with parameters $(X, \Sigma, \f S)$.

\subsection{Outline of the paper} \label{sec:outline}
The bulk of this paper is devoted to the proof of the anisotropic local law, which is the content of Sections \ref{sec:tools}--\ref{sec:self-const_II}. For clarity of presentation, we first give all proofs under the assumption \eqref{T_diag}, and subsequently explain how to remove it. In Section \ref{sec:tools} we collect the basic tools that we shall use throughout the proofs; they consist of basic identities and estimates for the matrix $G$. In Section \ref{sec:diag_Sigma} we perform Step (A) of the proof, by proving the entrywise local law under the assumption that $\Sigma$ is diagonal (Theorem \ref{thm:diag_gen}). In Section \ref{sec:iso_gauss} we perform Step (B) of the proof, by proving the anisotropic local law for general $\Sigma$ and Gaussian $X$ (Theorem \ref{thm:gen_iso} (i) for Gaussian $X$). The main step, Step (C), of the proof is the content of Sections \ref{sec:comparison1} and \ref{sec:self-const_II}. In Section \ref{sec:comparison1} we explain the main ideas of the self-consistent comparison method and complete the proof under the additional assumption that $\E X_{i \mu}^3 = 0$. In Section \ref{sec:self-const_II} we give the additional arguments needed to complete the proof for general $X$. Together, these two sections complete the proof of Theorem \ref{thm:gen_iso} (i) for general $X$.

Having completed the proof of the anisotropic local law, we prove the averaged local law (Theorem \ref{thm:gen_iso} (ii)) in Section \ref{sec:avg_proof}. This will conclude the proof of Theorems \ref{thm:gen_iso} and \ref{thm:diag_gen}. At the end of Section \ref{sec:avg_proof}, we explain how to deduce Theorem \ref{thm:LL_G}, Corollary \ref{cor:GM_GN}, and Theorems \ref{thm:edge} (i)--(ii), \ref{thm:bulk} (i)--(ii), and \ref{thm:outside}.

In Section \ref{sec:rigidity} we prove the rigidity of the eigenvalues (Theorems \ref{thm:rigidity}, \ref{thm:edge} (iii), and \ref{thm:bulk} (iii)) and the universality of their joint distribution near the edges (Theorem \ref{thm:edge_univ}). Next, in Section \ref{sec:generalizations} we explain how to remove the assumption \eqref{T_diag} and how to extend all of our results from the matrix $Q$ to the matrix $\dot Q$.

In Section \ref{sec:Wig}, as a further illustration of the self-consistent comparison method, we present and prove analogous results for deformed Wigner matrices.

The first part of Appendix \ref{sec:varrho} is devoted to the proof the basic properties of the asymptotic eigenvalue density $\varrho$ stated in Section \ref{sec:structure of rho}. In the second part of Appendix \ref{sec:varrho}, we establish the two key assumptions of Theorem \ref{thm:diag_gen} -- \eqref{1+xm} and the stability of \eqref{def_m} -- on each subdomain $\f D^e_k$, $\f D^b_k$, and $\f D^o$ separately, under the regularity assumptions from Definition \ref{def:regular}.

\section{Basic tools} \label{sec:tools}

The rest of this paper is devoted to the proofs. In this preliminary section we collect various identities from linear algebra and simple estimates that we shall use throughout the paper.

We always use the following convention for matrix multiplication.

\begin{definition}[Matrix multiplication] \label{def:matrix_m}
We use matrices of the form $A = (A_{st} \col s \in l(A), t \in r(A))$, whose entries are indexed by arbitrary finite subsets of $l(A), r(A) \subset \N$. Matrix multiplication $A B$ is defined for $s \in l(A)$ and $t \in r(B)$ by
\begin{equation*}
(A B)_{st} \;\deq\; \sum_{q \in r(A) \cap l(B)} A_{sq} B_{qt}\,.
\end{equation*}
\end{definition}

\begin{definition}
Suppose \eqref{T_diag}.
Define the $\cal I \times \cal I$ matrices
\begin{equation*}
G(z) \;\deq\; H(z)^{-1}\,, \qquad H(z) \;\deq\; \begin{pmatrix}
-\Sigma^{-1} & X
\\
X^* & -z
\end{pmatrix}\,,
\end{equation*}
as well as the 
$\cal I_M \times \cal I_M$ matrix
\begin{equation} \label{defGM}
G_M(z) \;\deq\; \pb{-\Sigma^{-1}+z^{-1} XX^*}^{-1} \;=\; z \Sigma^{1/2} \pb{\Sigma^{1/2}XX^*\Sigma^{1/2}-z}^{-1}\Sigma^{1/2}
\end{equation}
and the $\cal I_N \times \cal I_N$ matrix
\begin{equation} \label{defGN}
G_N(z) \;\deq\; ( X^*\Sigma X - z)^{-1}\,.
\end{equation}
Throughout the following we frequently omit the argument $z$ from our notation.
\end{definition}

Since $H(z)$ and $G(z)$ are only defined under the assumption \eqref{T_diag}, we shall always tacitly assume \eqref{T_diag} whenever we use them. Note that under the assumption \eqref{T_diag} we have $R_N = G_N$.

\begin{definition}[Minors] \label{def:minors}
For $S \subset \cal I$ we define the minor
$H^{(S)} \deq (H_{st} \col s,t \in \cal I \setminus S)$.
We also write
$G^{(S)} \deq (H^{(S)})^{-1}$.
The matrices $G_N^{(S)}$ and $G_M^{(S)}$ are defined similarly. We abbreviate $(\{s\}) \equiv (s)$ and $(\{s,t\}) \equiv (st)$.
\end{definition}

\bigskip

\begin{lemma}[Resolvent identities]\label{lem:idG}
\begin{enumerate}
\item
We have
\begin{equation}\label{G=G}
G \;=\;
\begin{pmatrix}
\Sigma X G_N X^* \Sigma - \Sigma & \Sigma X G_N
\\
G_N X^* \Sigma & G_N
\end{pmatrix} 
\;=\;
\begin{pmatrix}
G_M & z^{-1} G_M X
\\
z^{-1} X^* G_M & z^{-2} X^* G_M X - z^{-1}
\end{pmatrix}\,.
\end{equation}
\item
For $\mu \in \cal I_N$ we have
\be
\label{Gii}
\frac{1}{G_{\mu \mu}} \;=\; -z - \pb{X^* G^{(\mu)} X}_{\mu\mu}\,,
\ee
and for $\mu \neq \nu \in \cal I_N$
\begin{equation} \label{Gij}
G_{\mu \nu} \;=\; -G_{\mu\mu} \pb{X^* G^{(\mu)}}_{\mu \nu}  \;=\; -G_{\nu \nu} \pb{G^{(\nu)} X}_{\mu \nu} \;=\; G_{\mu \mu}G_{\nu \nu}^{(\mu)} \pb{X^* G^{(\mu \nu)} X}_{\mu \nu}\,.
\end{equation}
\item
Suppose that $\Sigma$ is diagonal.
Then for $i \in \cal I_M$ we have
\begin{equation} \label{Gaa}
\frac{1}{G_{i i}} \;=\; - \frac{1}{\sigma_i} - \pb{X G^{(i)} X^*}_{i i}\,,
\end{equation}
and for $i \neq j \in \cal I_M$
\begin{equation} \label{Gab}
G_{ij} \;=\; -G_{ii} \pb{X G^{(i)}}_{i j} \;=\; -G_{jj} (G^{(j)} X^*)_{ij} \;=\; G_{ii} G_{jj}^{(i)} \pb{X G^{(ij)} X^*}_{ij}\,.
\end{equation}
\item
For $i \in \cal I_M$ and $\mu \in \cal I_N$ we have
\be\label{Gia}
G_{i \mu} \;=\; - G_{\mu \mu} \pb{G^{(\mu)} X}_{i \mu}\,, \qquad G_{\mu i} \;=\; - G_{\mu \mu}  \pb{X^* G^{(\mu)}}_{\mu i}\,.
\ee
In addition, if $\Sigma$ is diagonal, we have
\begin{align} \label{Gia2}
G_{i \mu} &\;=\; -G_{ii} (X G^{(i)})_{i \mu} \;=\; G_{ii} G_{\mu \mu}^{(i)} \pB{-X_{i \mu} + \pb{X G^{(i \mu)} X}_{i \mu}} \,,
\\ \label{Gia3}
G_{\mu i} &\;=\; - G_{ii} (G^{(i)} X)_{\mu i} \;=\; G_{\mu \mu} G_{ii}^{(\mu)} \pB{- X_{\mu i}^* + \pb{X^* G^{(\mu i)} X^*}_{\mu i}}\,.
\end{align}

\item
For $r \in \cal I$ and $s,t \in \cal I \setminus \{k\}$ we have
\be\label{Gstu}
G_{st}^{(r)}=G_{st}-\frac{G_{sr}G_{rt}}{G_{rr}}
\ee
\item
All of the identities from (i)--(v) hold for $G^{(S)}$ instead of $G$ if $S \subset \cal I_N$ or $S \subset \cal I$ and $\Sigma$ is diagonal.
\end{enumerate}
\end{lemma}

\begin{proof}
The identities \eqref{G=G}, \eqref{Gii}, and \eqref{Gaa} follow from Schur's complement formula. The remaining identities follow easily from resolvent identities that have been previously derived in \cite{EYY1, EKYY2}; they are summarized e.g.\ in \cite[Lemma 4.5]{EKYY4}.
\end{proof}

Next, we introduce the spectral decomposition of $G$.
We use the notation
\begin{equation*}
\Sigma^{1/2} X \;=\; \sum_{k=1}^{M \wedge N}\sqrt{\lambda_k} \, \f \xi_k \f \zeta_k^*
\end{equation*}
for the singular value decomposition of $\Sigma^{1/2} X$, where
\begin{equation*}
\lambda_1 \;\geq\; \lambda_2 \;\geq\; \cdots \;\geq\; \lambda_{M \wedge N}  \;\geq\; 0 \;=\; \lambda_{M \wedge N + 1} \;=\; \cdots\;=\; \lambda_{N \vee M}\,,
\end{equation*}
and $\{\f \xi_k\}_{k = 1}^M$ and $\{\f \zeta_k\}_{k = 1}^N$ are orthonormal bases of $\R^{\cal I_M}$ and $\R^{\cal I_N}$ respectively.
Then for $\mu, \nu \in \cal I_N$ we have
\begin{equation} \label{GN_spec}
 G_{\mu \nu} \;=\; \sum_{k=1}^N\frac{\f \zeta_k(\mu) \ol{\f \zeta_k(\nu)}}{\lambda_k-z}\,,
\end{equation}
 and for $i,j \in \cal I_M$
\begin{equation} \label{GM_spec}
G_{ij} \;=\; z \sum_{k=1}^M \frac{  (\Sigma^{1/2}\f \xi_k)(i) \ol{(\Sigma^{1/2} \f \xi_k) (j)}}{\lambda_k-z} \;=\; -\Sigma_{ij} + \sum_{k = 1}^M \frac{ \lambda_k (\Sigma^{1/2}\f \xi_k)(i)\ol{(\Sigma^{1/2}\f \xi_k) (j)}}{\lambda_k-z}\,.
\end{equation}
Moreover, for $i \in \cal I_M$ and $\mu \in \cal I_N$ we have
\begin{equation} \label{Gial}
G_{i \mu} \;=\; \sum_{k = 1}^{M \wedge N} \frac{\sqrt{\lambda_k} (\Sigma^{1/2} \f \xi_k)(i) \ol{\f \zeta_k(\mu)}}{\lambda_k - z}\,, \qquad  G_{\mu i} \;=\; \sum_{k = 1}^{M \wedge N} \frac{\sqrt{\lambda_k} \f \zeta_k(\mu) \ol{(\Sigma^{1/2}\f \xi_k)(i)}}{\lambda_k-z}\,.
\end{equation}
Summarizing, defining
\begin{equation*}
\f u_k \;\deq\; \binom{\ind{k \leq M}\sqrt{\lambda_k} \, \f \xi_k}{\ind{k \leq N} \, \f \zeta_k} \;\in\; \R^{\cal I}\,,
\end{equation*}
we have
\be\label{decompfullG}
G \;=\;
- \ul \Sigma
+ \ul \Sigma^{1/2} \sum_{k = 1}^{N \vee M} \frac{\f u_k \f u_k^*}{\lambda_k-z} \ul \Sigma^{1/2}\,.
\ee

\begin{definition}[Generalized entries]
For $\f v, \f w \in \R^{\cal I}$, $s \in \cal I$, and an $\cal I \times \cal I$ matrix $A$, we abbreviate
\begin{equation*}
A_{\f v \f w} \;\deq\; \scalar{\f v}{A \f w} \,, \qquad 
A_{\f v s} \;\deq\; \scalar{\f v}{A \f e_s}\,, \qquad
A_{s \f v} \;\deq\; \scalar{\f e_s}{A \f v}\,,
\end{equation*}
where $\f e_s$ denotes the standard unit vector in the coordinate direction $s$.
\end{definition}

We sometimes identify vectors $\f v \in \R^{\cal I_M}$ and $\f w \in \R^{\cal I_N}$ with their natural embeddings $\binom{\f v}{0}$ and $\binom{0}{\f w}$ in $\R^{\cal I}$.
The following result is our fundamental tool for estimating entries of $G$.

\begin{lemma}\label{lemma: basic}
Fix $\tau > 0$. Then the following estimates hold for any $z \in \f D$.
We have
\be\label{opbound}
\normb{\ul \Sigma^{-1/2} G \ul \Sigma^{-1/2}} \;\leq\; C \eta^{-1}\,, \qquad  \normb{\ul \Sigma^{-1/2} \partial_z G \, \ul \Sigma^{-1/2}} \;\leq\; C \eta^{-2}\,.
\ee
Furthermore, let $\f v \in \R^{\cal I_M}$ and $\f w \in \R^{\cal I_N}$.
Then we have the bounds
\begin{align}
\label{GMaa4}
\sum_{\mu \in \cal I_N} \abs{G_{\f w \mu}}^2 &\;=\; \frac{\im G_{\f w \f w}}{\eta}\,,
\\
\label{GMaa}
\sum_{i \in \cal I_M} \abs{G_{\f v i}}^2 &\;\leq\;  \frac{C \norm{X^* X}}{\eta} \im G_{\f v \f v} + 2 (\Sigma^2)_{\f v \f v}\,,
\\
\label{GMaa3}
\sum_{i \in \cal I_M} \abs{G_{\f w i}}^2
&\;\leq\; C \norm{X^* X} \sum_{\mu \in \cal I_N} \abs{G_{\f w \mu}}^2\,,
\\
\label{GMaa2}
\sum_{\mu \in \cal I_N} \abs{G_{\f v \mu}}^2 &\;\leq\;  C \norm{X^* X} \sum_{i \in \cal I_M} \abs{G_{\f v i}}^2\,.
\end{align}
Finally, the estimates \eqref{opbound}--\eqref{GMaa2} remain true for $G^{(S)}$ instead of $G$ if $S \subset \cal I_N$ or $S \subset \cal I$ and $\Sigma$ is diagonal.
\end{lemma}

\begin{proof}

The estimates \eqref{opbound} follow from \eqref{decompfullG}, using $\norm{A} \;=\; \sup \hb{\abs{\scalar{\f x}{A \f y}} \col \abs{\f x}, \abs{\f y} \leq 1}$, \eqref{bound_Sigma}, and $\abs{\lambda_k}/{\abs{\lambda_k - z}} \leq C \eta^{-1}$ which follows from the bound $\abs{z} \asymp 1$. Moreover, \eqref{GMaa4} easily follows from \eqref{defGN}, \eqref{G=G}, and the spectral decomposition \eqref{GN_spec}.

In order to prove \eqref{GMaa}, we use \eqref{G=G} to write
\begin{multline*}
\sum_{i \in \cal I_M} \abs{G_{\f v i}}^2 \;=\; \sum_{i \in \cal I_M} \absb{(\Sigma X G X^* \Sigma)_{\f v i} - \Sigma_{\f vi}}^2
\;\leq\; 2 (\Sigma X G X^* \Sigma^2 X G^* X^* \Sigma)_{\f v \f v} + 2 (\Sigma^2)_{\f v \f v}
\\
\;\leq\; C \norm{X^* X} (\Sigma X G_N G_N^* X^* \Sigma)_{\f v \f v} + 2 (\Sigma^2)_{\f v \f v}
\;=\; \frac{C \norm{X^* X}}{\eta} \im (\Sigma X G X^* \Sigma)_{\f v \f v} + 2 (\Sigma^2)_{\f v \f v}
\\
\;=\; \frac{C \norm{X^* X}}{\eta} \im G_{\f v \f v} + 2 (\Sigma^2)_{\f v \f v}\,,
\end{multline*}
which is the claim.

In order to prove \eqref{GMaa3}, we use \eqref{G=G} and \eqref{GN_spec} to get
\begin{equation*}
\sum_{i \in \cal I_M} \abs{G_{\f w i}}^2 \;=\; (G X^* \Sigma^2 X G^*)_{\f w \f w} \;\leq\; C (G X^* X G^*)_{\f w \f w} \;\leq\; C \norm{X^* X} (G_N G^*_N)_{\f w \f w}\,.
\end{equation*}
The estimate \eqref{GMaa2} is proved similarly:
\begin{equation*}
\sum_{\mu \in \cal I_N} \abs{G_{\f v \mu}}^2 \;=\; \abs{z}^{-2} \pb{G XX^*G^*}_{\f v \f v} \;\leq\; C \norm{X^* X} (G_M G_M^*)_{\f v \f v}\,.
\end{equation*}
Finally, the same estimates for $G^{(S)}$ instead of $G$ follow using a trivial modification of the above argument.
\end{proof}

\begin{definition} \label{def:high_prob}
We say that an event $\Xi$ \emph{holds with high probability} if $1 - \ind{\Xi} \prec 0$.
\end{definition}

The following result may be used to estimate the factors $\norm{X^* X}$ in Lemma \ref{lemma: basic} with high probability. It follows from \cite[Theorem 2.10]{BEKYY}.
\begin{lemma} \label{lem:norm_XX}
Under the assumptions \eqref{MsimN}, \eqref{cond on entries of X}, and \eqref{moments of X-1}, there exists a constant $C > 0$ such that $\norm{X^* X}\leq C$ with high probability.
\end{lemma}

Using Lemma \ref{lem:norm_XX}, we observe that we may improve \eqref{opbound} provided we settle for a high-probability instead of a deterministic statement.

\begin{lemma} \label{lem:opnorm}
We have the bounds
\begin{equation*}
\normb{\ul \Sigma^{-1} \, (G + \ul \Sigma) \, \ul \Sigma^{-1}} \;\leq\; C \norm{X^* X} \eta^{-1}\,, \qquad
\normb{\ul \Sigma^{-1} \, \partial_z G \, \ul \Sigma^{-1}} \;\leq\; C \norm{X^* X} \eta^{-2}
\end{equation*}
for all $z \in \f D$.
\end{lemma}
\begin{proof}
The claim is an easy consequence of the first identity of \eqref{G=G} combined with \eqref{defGN}.
\end{proof}

We conclude this section with the following basic properties of $m$, which can be proved as in \cite{BaoPanZhou} and the references therein.
\begin{lemma}[General properties of $m$] \label{lem:gen_prop_m}
Fix $\tau > 0$ and suppose that \eqref{phi_bounded}, \eqref{bound_Sigma}, and \eqref{Sigma_gap} hold. Then there exists a constant $C > 0$ such that
\begin{equation} \label{m_sim_1}
C^{-1} \;\leq\; \abs{m(z)} \;\leq\; C
\end{equation}
and
\begin{equation} \label{Im_m_geq_c}
\im m(z) \;\geq\; C^{-1} \, \eta
\end{equation}
for all $z \in \C_+$ satisfying $\tau \leq \abs{z} \leq \tau^{-1}$.
\end{lemma}
In particular, from the upper bound in \eqref{m_sim_1} we deduce that $\varrho$ has a bounded density on $[\tau,\infty)$.

\section{Entrywise local law for diagonal $\Sigma$} \label{sec:diag_Sigma}

In this section we prove Theorem \ref{thm:diag_gen}, hence performing Step (A) of the proof mentioned in the introduction.
The proof of Theorem \ref{thm:diag_gen} is similar to previous proofs of local entrywise laws, such as \cite{PY, BEKYY}. We follow the basic approach of \cite[Section 4]{BEKYY}, and only give the details where the argument departs significantly from that of \cite{BEKYY}.

The main novel observation of this section is that the equation \eqref{def_m} arises very easily from the random matrix model by a double application of Schur's complement formula. Heuristically, this may be seen using the identities \eqref{Gii} and \eqref{Gaa}. Indeed, suppose that $G_{\mu \mu} \approx m$ for $\mu \in \cal I_N$. We ignore the random fluctuations in \eqref{Gii} to get
\begin{equation} \label{approx_1}
\frac{1}{m} \;\approx\; \frac{1}{G_{\mu \mu}} \;=\; - z - \pb{X^* G^{(\mu)} X}_{\mu\mu} \;\approx\; - z - \frac{1}{N} \sum_{i \in \cal I_M} G_{ii}^{(\mu)} \;\approx\; - z - \frac{1}{N} \sum_{i \in \cal I_M} G_{ii}\,.
\end{equation}
Similarly, ignoring the random fluctuations in \eqref{Gaa}, we get
\begin{equation} \label{approx_2}
\frac{1}{G_{ii}} \;=\; - \frac{1}{\sigma_i} - \pb{X G^{(i)} X^*}_{i i} \;\approx\; -\frac{1}{\sigma_i} - \frac{1}{N} \sum_{\mu \in \cal I_N} G^{(i)}_{\mu \mu} \;\approx\; -\frac{1}{\sigma_i} - \frac{1}{N} \sum_{\mu \in \cal I_N} G_{\mu \mu} \;\approx\; -\frac{1}{\sigma_i} - m\,.
\end{equation}
Plugging \eqref{approx_2} into \eqref{approx_1} yields \eqref{def_m}. In this section we give a rigorous justification of these approximations.

\subsection{Weak entrywise law}
In this subsection we establish the following weaker version of Theorem \ref{thm:diag_gen}.  It is analogous to \cite[Proposition 4.2]{BEKYY}.
\begin{proposition}[Weak entrywise law] \label{prop:weak law}
Suppose that the assumptions of Theorem \ref{thm:diag_gen} hold. Then $\Lambda \prec (N \eta)^{-1/4}$ uniformly in $z \in \f S$.
\end{proposition}

The rest of this subsection is devoted to the proof of Proposition \ref{prop:weak law}. For each $i \in \cal I_M$ we define
\begin{equation} \label{def_m_i}
m_i \;\deq\; \frac{- \sigma_i}{1 + m \sigma_i}\,.
\end{equation}
Recalling \eqref{def_m}, we find that the functions $m$ and $m_i$ satisfy
\begin{equation*}
\frac{1}{m} \;=\; -z - \frac{1}{N} \sum_{i \in \cal I_M} m_i \,, \qquad \frac{1}{m_i} \;=\; - \frac{1}{\sigma_i} - m\,.
\end{equation*}
Note that \eqref{1+xm} implies
\begin{equation} \label{m_i_bound}
\abs{m_i} \;\leq\; C \sigma_i \qquad \text{for } z \in \f S \text{ and } i \in \cal I_M\,.
\end{equation}

Next, we define the random control parameters
\begin{equation*}
\Lambda \;\deq\; \max_{s,t \in \cal I} \, \absb{\pb{\ul \Sigma^{-1} \pb{G(z) - \Pi(z)} \ul \Sigma^{-1}}_{st}}\,,
\qquad
\Lambda_o \;\deq\; \max_{s \neq t \in \cal I} \, \absb{\pb{\ul \Sigma^{-1} G(z) \ul \Sigma^{-1}}_{st}}\,.
\end{equation*}
We extend the definitions of $\sigma_i$ and $m_i$ for $i \in \cal I_M$ by setting $\sigma_\mu \deq 1$ and $m_\mu \deq m$ for $\mu \in \cal I_N$.
We may therefore write
\begin{equation*}
\Lambda \;=\; \max_{s,t \in \cal I} \frac{\abs{G_{st} - \delta_{st} m_s}}{\sigma_s \sigma_t}\,, \qquad
\Lambda_o \;=\; \max_{s \neq t \in \cal I} \frac{\abs{G_{st}}}{\sigma_s \sigma_t}\,.
\end{equation*}
Moreover, we define the averaged control parameters
\begin{equation*}
\Theta \;\deq\; \Theta_M + \Theta_N \,, \qquad \Theta_M \;\deq\; \absbb{\frac{1}{M} \sum_{i \in \cal I_M} (G_{ii} - m_i)} \,, \qquad \Theta_N \;\deq\; \absbb{\frac{1}{N} \sum_{\mu \in \cal I_N} (G_{\mu \mu} - m)} \;=\; \abs{m_N - m}
\,.
\end{equation*}
We have the trivial bound
\begin{equation} \label{Theta_leq_Lambda}
\Theta \;\leq\; C \Lambda\,.
\end{equation}

For $s \in \cal I$ we introduce the conditional expectation
\begin{equation} \label{cond_exp}
\E_s[\,\cdot\,] \;\deq\; \E\qb{ \,\cdot\, \vert H^{(s)}}\,.
\end{equation}
Using \eqref{Gaa} we get for $i \in \cal I_M$
\begin{equation} \label{Gii_exp}
\frac{1}{G_{ii}} \;=\; -\frac{1}{\sigma_i} - \frac{1}{N} \tr G_N^{(i)} - Z_i \,, \qquad Z_i \;\deq\; (1 - \E_i) \pb{X G^{(i)} X^*}_{ii}\,,
\end{equation}
and using \eqref{Gii} we get for $\mu \in \cal I_N$
\begin{equation} \label{Gmumu_exp}
\frac{1}{G_{\mu \mu}} \;=\; - z - \frac{1}{N} \tr G^{(\mu)}_{M} - Z_\mu\,, \qquad Z_\mu \;\deq\; (1 - \E_\mu) \pb{X^* G^{(\mu)} X}_{\mu \mu}\,.
\end{equation}

In analogy to \cite[Section 4]{BEKYY}, we define the $z$-dependent event
$\Xi \deq \hb{\Lambda \leq (\log N)^{-1}}$
and the control parameter
\begin{equation*}
\Psi_\Theta \;\deq\; \sqrt{\frac{\im m + \Theta}{N \eta}}\,.
\end{equation*}

The following estimate is analogous to \cite[Lemma 4.4]{BEKYY}.
\begin{lemma} \label{lem:Lambda_o_Z}
Suppose that the assumptions of Theorem \ref{thm:diag_gen} hold.
Then for $s \in \cal I$ and $z \in \f S$ we have
\begin{equation} \label{Lambda_o_Z_bound}
\ind{\Xi} \pb{\abs{Z_s} + \Lambda_o} \;\prec\; \Psi_\Theta
\end{equation}
as well as
\begin{equation} \label{Lambda_o_Z_bound_initial}
\ind{\eta \geq 1} \pb{\abs{Z_s} + \Lambda_o} \;\prec\; \Psi_\Theta\,.
\end{equation}
\end{lemma}
\begin{proof}
The proof relies on the identities from Lemma \ref{lem:idG} and large deviation estimates, like that of \cite[Lemma 4.4]{BEKYY} and \cite[Theorems 6.8 and 6.9]{PY}. Note first that \eqref{m_i_bound}, combined with \eqref{m_sim_1} and \eqref{def_m_i}, yields
\begin{equation} \label{m_asymp_1}
\abs{m_s} \;\asymp\; \sigma_s \qquad (s \in \cal I)\,.
\end{equation}
Using \eqref{Gstu} and a simple induction argument, it is not hard to conclude that
\begin{equation} \label{Gtt_bound}
\ind{\Xi} \absb{G_{tt}^{(S)}} \;\asymp\; \sigma_t
\end{equation}
for any $S \subset \cal I$ and $t \in \cal I \setminus S$ satisfying $\abs{S} \leq C$.

Let us first estimate $\Lambda_o$ in \eqref{Lambda_o_Z_bound}. We shall in fact prove that
\begin{equation} \label{Gst_step1}
\ind{\Xi} \abs{G_{st}} \;\prec\; \sigma_s \sigma_t \pbb{\frac{\im m + \Theta + \Lambda_o^2}{N \eta}}^{1/2}
\end{equation}
for all $s \neq t \in \cal I$. From \eqref{Gst_step1} it is easy to deduce that $\ind{\Xi} \Lambda_o \prec \Psi_\Theta$.

Let us start with $G_{ij}$ for $i \neq j \in \cal I_M$. Using \eqref{Gab}, \eqref{Gtt_bound}, and a large deviation estimate (see \cite[Lemma 3.1]{BEKYY}), we find
\begin{equation} \label{Gij_estimate}
\ind{\Xi} \abs{G_{ij}} \;\leq\; \ind{\Xi} \absb{G_{ii} G_{jj}^{(i)}} \absBB{\sum_{\mu,\nu \in \cal I_N}  X_{i \mu} G^{(ij)}_{\mu \nu} X^*_{\nu j}} \;\prec\; \ind{\Xi} \sigma_i \sigma_j\pBB{\frac{1}{N^2} \sum_{\mu,\nu \in \cal I_N} \absb{G_{\mu \nu}^{(ij)}}^2}^{1/2}\,.
\end{equation}
The term in parentheses is
\begin{equation*}
\ind{\Xi} \frac{1}{N^2} \sum_{\mu,\nu \in \cal I_N} \absb{G_{\mu \nu}^{(ij)}}^2 \;=\; \ind{\Xi} \frac{1}{N^2 \eta} \sum_{\mu \in \cal I_N} \im G_{\mu \mu}^{(ij)} \;\prec\; \frac{1}{N^2 \eta} \sum_{\mu \in \cal I_N} \im G_{\mu \mu} + \frac{\Lambda_o^2}{N \eta} \;\leq\; \frac{\im m + \Theta_N + \Lambda_o^2}{N \eta}\,,
\end{equation*}
where in the first step we used \eqref{Gstu} and \eqref{Gtt_bound}. This yields \eqref{Gst_step1} for $s,t \in \cal I_M$.

Next, $\ind{\Xi} G_{\mu \nu}$ for $\mu \neq \nu$ is estimated similarly, using \eqref{Gij}, \eqref{GMaa}, \eqref{Im_m_geq_c}, and the bound $\im m_i \leq C \sigma_i^2 \im m$
for all $i \in \cal I_M$, as follows easily from \eqref{def_m_i}. Finally, $\ind{\Xi} G_{i \mu}$ with $i \in \cal I_M$ and $\mu \in \cal I_N$ is estimated similarly, using \eqref{Gia2}, \eqref{GMaa3}, Lemma \ref{lem:norm_XX}, and \eqref{GMaa4}. This concludes the estimate of $\ind{\Xi} \Lambda_o$. 

An analogous argument for $Z_s$   (see e.g.\ \cite[Lemma 5.2]{EKYY4}) \nc completes the proof of \eqref{Lambda_o_Z_bound}.

In order to prove \eqref{Lambda_o_Z_bound_initial}, we proceed similarly. For $\eta \geq 1$, we proceed as above to get $\abs{Z_s} \prec N^{-1/2}$, where we used that $\norm{G^{(s)}} \leq C$ by \eqref{opbound}. Similarly, as in \eqref{Gij_estimate} we get
\begin{equation*}
\abs{G_{ij}} \;\leq\; \absb{G_{ii} G_{jj}^{(i)}} \absBB{\sum_{\mu,\nu \in \cal I_N}  X_{i \mu} G^{(ij)}_{\mu \nu} X^*_{\nu j}} \;\prec\; \absb{G_{ii} G_{jj}^{(i)}} N^{-1/2}\,,
\end{equation*}
where in the last step we used that $\im G_{\mu \mu}^{(ij)} \leq C$, by \eqref{opbound}. Moreover, from \eqref{defGM} and \eqref{G=G} we get immediately that $\abs{G_{ii}} \leq C \sigma_i$, and a similar argument for $G^{(i)}$ implies that $\abs{G_{jj}^{(i)}} \leq C \sigma_j$. This concludes the proof.
\end{proof}

Recall the definition of $m_N$ from \eqref{def_m_N}.
From \eqref{Gii_exp} combined with \eqref{Gstu} and Lemma \ref{lem:Lambda_o_Z} we get, for $i \in \cal I_M$ and $z \in \f S$,
\begin{equation} \label{Gii_2}
\ind{\Xi} G_{ii} \;=\; \ind{\Xi} \frac{-\sigma_i}{1 + m_N \sigma_i + \sigma_i Z_i + O_\prec(\sigma_i \Psi_\Theta^2)}\,.
\end{equation}
Similarly, from \eqref{Gmumu_exp} we get, for $\mu \in \cal I_N$ and $z \in \f S$,
\begin{equation} \label{Gmumu_2}
\ind{\Xi} \frac{1}{G_{\mu \mu}} \;=\; \ind{\Xi} \pBB{-z - \frac{1}{N} \sum_{i \in \cal I_M} G_{ii} - Z_\mu + O_\prec(\Psi_\Theta^2)}\,.
\end{equation}

As in \cite[Lemma 4.7]{BEKYY}, it is easy to derive from \eqref{Gmumu_2} and Lemma \ref{lem:Lambda_o_Z} that
\begin{equation} \label{Gmumu-m}
\ind{\Xi} \abs{G_{\mu \mu} - m_N} \;\prec\; \Psi_\Theta
\end{equation}
for $\mu \in \cal I_N$ and $z \in \f S$. Hence, expanding $G_{\mu \mu} = m_N + (G_{\mu \mu} - m_N)$ and using \eqref{m_sim_1} yields
\begin{equation*}
\ind{\Xi} \frac{1}{N} \sum_{\mu \in \cal I_N} \frac{1}{G_{\mu \mu}} \;=\; \ind{\Xi} \frac{1}{m_N} + O_\prec(\Psi_\Theta^2)\,.
\end{equation*}
Plugging this and \eqref{Gii_2} into \eqref{Gmumu_2} yields
\begin{equation} \label{m_N_1}
\ind{\Xi} \frac{1}{m_N} \;=\; \ind{\Xi}\pBB{-z + \frac{1}{N} \sum_{i \in \cal I_M} \frac{\sigma_i}{1 + m_N \sigma_i + \sigma_i Z_i + O_\prec(\sigma_i \Psi_\Theta^2)} - \frac{1}{N} \sum_{\mu \in \cal I_N} Z_\mu + O_\prec(\Psi_\Theta^2)}\,.
\end{equation}
From \eqref{m_N_1}, Lemma \ref{lem:Lambda_o_Z}, and the estimate $\ind{\Xi} \abs{1 + m_N \sigma_1} \geq c$ (as follows from \eqref{1+xm}), we conclude the following result.

\begin{lemma} \label{lem:calD_est}
Suppose that the assumptions of Theorem \ref{thm:diag_gen} hold.
Define
\begin{equation*}
[Z]_N \;\deq\; \frac{1}{N} \sum_{\mu \in \cal I_N} Z_\mu\,, \qquad [Z]_M \;\deq\; \frac{1}{N} \sum_{i \in \cal I_M} \frac{\sigma_i^2}{(1 + m_N \sigma_i)^2} Z_i\,.
\end{equation*}
Then for $z \in \f S$ we have
\begin{equation} \label{calD_bound}
\ind{\Xi} (f(m_N) - z) \;=\; \ind{\Xi} \pb{[Z]_N  + [Z]_M + O_\prec(\Psi_\Theta^2)}\,,
\end{equation}
where $f$ was defined in \eqref{def_fx}.
\end{lemma}

Next, we give the precise stability condition for the equation \eqref{def_m}. Roughly, it says that if $f(u(z)) - z$ is small and $u(\tilde z) - m(\tilde z)$ is small for $\tilde z \deq z + \ii N^{-5}$, then $u(z) - m(z)$ is small. (Recall that $f(m(z)) - z = 0$.)

\begin{definition}[Stability of \eqref{def_m} on $\f S$] \label{def:stability}
Let $\f S \subset \f D$ be a spectral domain (see Definition \ref{def:local_laws}). We say that \eqref{def_m} is \emph{stable on} $\f S$ if the following holds for some large enough constant $C > 0$.
Suppose that $\delta \col \f S \to (0,\infty)$ satisfies $N^{-2} \leq \delta(z) \leq (\log N)^{-1}$ for $z \in \f S$ and that $\delta$ is Lipschitz continuous with Lipschitz constant $N^2$. Suppose moreover that for each fixed $E$, the function $\eta \mapsto \delta(E + \ii \eta)$ is nonincreasing for $\eta > 0$. Suppose that $u \col \f S \to \C$ is the Stieltjes transform of a probability measure supported in $[0,C]$. Let $z \in \f S$ and suppose that
\begin{equation*}
\absb{f(u(z)) - z} \;\leq\; \delta(z)\,.
\end{equation*}
If $\im z < 1$ suppose also that
\begin{equation} \label{u-m_stability}
\abs{u - m} \;\leq\; \frac{C \delta}{\im m +  \sqrt{\delta}}
\end{equation}
holds at $z + \ii N^{-5}$. Then \eqref{u-m_stability} holds at $z$.
\end{definition}

This condition has previously appeared, in somewhat different guises, in the works \cite{PY, BEKYY, BaoPanZhou}, where it was established under various assumptions on $\pi$. For instance, in \cite{BaoPanZhou}, it was established for $\f S = \f D^e_1$ under the assumption that the edge $k = 1$ is regular (see Definition \ref{def:regular} (i)).
In Appendix \ref{sec:varrho}, we establish it for each of the subdomains $\f D^e_k$, $\f D^b_k$, and $\f D^o$ separately, under the regularity assumptions from Definition \ref{def:regular}.

In accordance with the assumptions of Theorem \ref{thm:diag_gen}, we suppose throughout this section that \eqref{def_m} is stable on $\f S$. Using \eqref{Lambda_o_Z_bound_initial} and \eqref{opbound}, it is easy to obtain the following result, which is analogous to \cite[Lemma 4.6]{BEKYY}.
\begin{lemma} \label{lem:initial}
Suppose that the assumptions of Theorem \ref{thm:diag_gen} hold.
Then we have $\Lambda \prec N^{-1/4}$ uniformly in $z \in \f S$ satisfying $\eta \geq 1$.
\end{lemma}

Exactly as in \cite[Section 4]{BEKYY}, we use a stochastic continuity argument to estimate $\Lambda$, using Lemmas \ref{lem:calD_est} and \ref{lem:initial}. The major input is the stability of \eqref{def_m} on $\f S$ in the sense of Definition \ref{def:stability}, which is analogous to \cite[Lemma 4.5]{BEKYY}. Proposition \ref{prop:weak law} now follows by estimating the right-hand side of \eqref{calD_bound} by $O_\prec(\Psi_\Theta)$, as follows from Lemma \ref{lem:Lambda_o_Z}. In its proof, the error $\ind{\Xi} (G_{\mu \mu} - m)$ is controlled using \eqref{Gmumu-m} by $\Theta + \Psi_\Theta$; similarly, the error $\ind{\Xi} (G_{ii} - m_i)$ is controlled using \eqref{Gii_2} by $\sigma_i^2 \Psi_\Theta$. We omit further details. This concludes the proof of Proposition \ref{prop:weak law}.

\subsection{Fluctuation averaging and proof of Theorem \ref{thm:diag_gen}}

The weak law, Proposition \ref{prop:weak law}, may be upgraded to the strong law, Theorem \ref{thm:diag_gen}, using improved estimates for the averaged quantities $[Z]_M$ and $[Z]_N$. We follow the arguments of \cite[Section 4.2]{BEKYY} to the letter. The key input is the following result, which is very similar to \cite[Lemma 4.9]{BEKYY}, combined with the observation that $Z_s = (1 - \E_s) \frac{1}{G_{ss}}$ for $s \in \cal I$, as follows from \eqref{Gii_exp} and \eqref{Gmumu_exp}.

\begin{lemma}[Fluctuation averaging] \label{lem:fluct_avg}
Suppose that the assumptions of Theorem \ref{thm:diag_gen} hold.
Suppose that $\Upsilon$ is a positive, $N$-dependent, deterministic function on $\f S$ satisfying $N^{-1/2} \leq \Upsilon \leq N^{-c}$ for some constant $c > 0$. Suppose moreover that $\Lambda \prec N^{-c}$ and $\Lambda_o \prec \Upsilon$ on $\f S$. Then on $\f S$ we have
\begin{equation} \label{avg_conclusion}
\frac{1}{N} \sum_{\mu \in \cal I_N} \pb{1 - \E_\mu} \frac{1}{G_{\mu \mu}} \;=\; O_\prec(\Upsilon^2)
\end{equation}
and
\begin{equation} \label{avg_conclusion2}
\frac{1}{M} \sum_{i \in \cal I_M} \frac{\sigma_i^2}{(1 + m_N \sigma_i)^2} \pb{1 - \E_i} \frac{1}{G_{ii}} \;=\; O_\prec(\Upsilon^2) \,.
\end{equation}
\end{lemma}

\begin{proof}
The estimate \eqref{avg_conclusion} is a trivial extension of \cite[Lemma 4.9]{BEKYY} and \cite[Theorem 4.7]{EKYY4}. The estimate \eqref{avg_conclusion2} may be proved using exactly the same method, explained in \cite[Appendix B]{EKYY4}. The only complication in the proof is that the coefficients $\frac{\sigma_i^2}{(1 + m_N \sigma_i)^2}$ are random and depend on $i$. Using \eqref{Gstu}, this is dealt with in the proof of \cite[Appendix B]{EKYY4} by writing, for any $j \in \cal I_N$,
\begin{equation*}
m_N \;=\; \frac{1}{N} \sum_{\mu \in \cal I_N} G^{(j)}_{\mu \mu} + \frac{1}{N} \sum_{\mu \in \cal I_N} \frac{G_{\mu j} G_{j \mu}}{G_{jj}}\,,
\end{equation*}
and continuing in this manner with any further indices $k,l,\dots \in \cal I_N$ that we wish to include as superscripts of $G_{\mu \mu}$. We refer to \cite[Appendix B]{EKYY4} for the full details.
\end{proof}

Using Lemmas \ref{lem:fluct_avg} and \eqref{lem:Lambda_o_Z} combined with \eqref{calD_bound} we get $\ind{\Xi} \abs{f(m_N) - z} \prec \Psi_\Theta^2$. Then we may follow the argument \cite[Section 4.2]{BEKYY} verbatim to get $\Theta \prec (N \eta)^{-1}$ on $\f S$. This concludes the proof of the averaged local law in Theorem \ref{thm:diag_gen}. Moreover, the entrywise local law follows immediately from Proposition \ref{prop:weak law}, Lemma \ref{lem:Lambda_o_Z}, \eqref{Gii_2}, and \eqref{Gmumu-m}. This concludes the proof of Theorem \ref{thm:diag_gen}. 

\section{Anisotropic local law for Gaussian $X$} \label{sec:iso_gauss}

We now begin the proof of Theorem \ref{thm:gen_iso}, which consists of Sections \ref{sec:iso_gauss}--\ref{sec:avg_proof}. In this section we perform the first step of the proof, by  establishing Theorem \ref{thm:gen_iso} for the special case that $X = X^{\txt{Gauss}}$ is Gaussian. This corresponds to Step (B) of the proof mentioned in the introduction.

\begin{proposition} \label{prop:Gaussian}
Theorem \ref{thm:gen_iso} holds if $X = X^{\txt{Gauss}}$ is Gaussian.
\end{proposition}

The rest of this section is devoted to the proof of Proposition \ref{prop:Gaussian}. We shall in fact prove the following result.
\begin{lemma} \label{lem:Gaussian}
Suppose that the assumptions of Theorem \ref{thm:gen_iso} hold and that $X = X^{\txt{Gauss}}$ is Gaussian.
If the entrywise local law holds with parameters $(X, D, \f S)$, then the entrywise local law holds with parameters $(X, \Sigma, \f S)$. (Recall the definition of $D \equiv D(\Sigma)$ from \eqref{def_D_Sigma}.)
\end{lemma}

Before proving Lemma \ref{lem:Gaussian}, we show how it implies Proposition \ref{prop:Gaussian}. In order to prove the anisotropic local law, we have to estimate the left-hand side of \eqref{isotropic law}. We split $\f v = \f v_M + \f v_N$ and $\f w = \f w_M + \f w_N$, where $\f v_M, \f w_M \in \R^{\cal I_M}$ and $\f v_N, \f w_N \in \R^{\cal I_N}$. Plugging this into the left-hand side of \eqref{isotropic law}, we find that it suffices to control,  for arbitrary deterministic orthogonal matrices $O_M \in \r O(M)$ and $O_N \in \r O(N)$,  the entries of the matrix
\begin{equation} \label{rotation of matrix}
\begin{pmatrix}
O_M & 0
\\
0 & O_N
\end{pmatrix}
G^\Sigma
\begin{pmatrix}
O_M^* & 0
\\
0 & O_N^*
\end{pmatrix}
\;\eqdist\; G^{O_M \Sigma O_M^*}\,,
\end{equation}
where we used that $X \eqdist O_M X O_N^*$ since $X$ is Gaussian. Applying Lemma \ref{lem:Gaussian} to the matrices $\wt \Sigma = O_M \Sigma O_M^*$ and $D = D(\Sigma) = D(\wt \Sigma)$, we obtain the anisotropic local law with parameters $(X, \Sigma, \f S)$.  Moreover, the averaged local law follows by writing $\Sigma = U D U^*$ and setting $O_M = U^*$ and $O_N = 1$ in \eqref{rotation of matrix}. This concludes the proof of Proposition \ref{prop:Gaussian}.

\begin{proof}[Proof of Lemma \ref{lem:Gaussian}]
In this proof we abbreviate $G^D \equiv G$. Using \eqref{rotation of matrix} with $O_M = U^*$ and $O_N = 1$, we find that it suffices to prove
\begin{equation*}
\qBB{
\begin{pmatrix}
U & 0
\\
0 & 1
\end{pmatrix}
\begin{pmatrix}
D^{-1} & 0
\\
0 & 1
\end{pmatrix}
\qBB{G -
\begin{pmatrix}
-D (1 + mD)^{-1} & 0
\\
0 & m
\end{pmatrix}
}
\begin{pmatrix}
D^{-1} & 0
\\
0 & 1
\end{pmatrix}
\begin{pmatrix}
U^* & 0
\\
0 & 1
\end{pmatrix}
}_{st} \;=\; O_\prec(\Psi)
\end{equation*}
for all $s,t \in \cal I$.
In components, this reads
\begin{align} \label{UGU_ij}
\absBB{\sum_{k,l \in \cal I_M} U_{ik} \frac{G_{kl} - \delta_{kl} m_k}{\sigma_k \sigma_l} U^*_{lj}} &\;\prec\; \Psi\,,
\\ \label{UG_imu}
\absBB{\sum_{k \in \cal I_M} U_{ik} \frac{G_{k \mu}}{\sigma_k}} + \absBB{\sum_{k \in \cal I_M} \frac{G_{\mu k}}{\sigma_k} U_{ki}} &\;\prec\; \Psi\,,
\\ \label{G_munu_bound}
\absb{G_{\mu \nu} - \delta_{\mu \nu} m_\mu} &\;\prec\; \Psi\,,
\end{align}
for $i,j \in \cal I_M$ and $\mu, \nu \in \cal I_N$.

The estimate \eqref{G_munu_bound} is trivial by assumption. What remains is the proof of \eqref{UGU_ij} and \eqref{UG_imu}. It is based on the polynomialization method developed in \cite[Section 5]{BEKYY}. The argument is very similar to that of \cite{BEKYY}, and we only outline the differences.

Let us begin with \eqref{UGU_ij}. By the assumption $\abs{G_{kk} - m_k} \prec \Psi \sigma_k^2$ and orthogonality of $U$, we have
\begin{equation*}
\sum_{k,l} U_{ik} \frac{G_{kl} - \delta_{kl} m_k}{\sigma_k \sigma_l} U^*_{lj} \;=\; \sum_{k} U_{ik} \frac{G_{kk} - m_k}{\sigma_k^2} U^*_{kj} + \sum_{k \neq l} U_{ik} \frac{G_{kl}}{\sigma_k \sigma_l} U_{lj} \;=\; O_\prec(\Psi) + \cal Z\,,
\end{equation*}
where we defined $\cal Z \deq \sum_{k \neq l} (\sigma_k \sigma_l)^{-1} U_{ik} G_{kl} U_{lj}$. We need to prove that $\abs{\cal Z} \prec \Psi$, which, following \cite[Section 5]{BEKYY}, we do by estimating the moment $\E \abs{\cal Z}^p$ for fixed $p \in 2 \N$. The argument from \cite[Section 5]{BEKYY} may be taken over with minor changes. We use the identities \eqref{Gstu},
\begin{equation} \label{iso_exp}
\sum_{k \neq l \in \cal I_M \setminus T} (\sigma_k \sigma_l)^{-1} U_{ik} G_{kl}^{(T)} U_{lj} \;=\;  \sum_{k \neq l \in \cal I_M \setminus T} (\sigma_k \sigma_l)^{-1} U_{ik} U_{lj} G_{kk}^{(T)} G_{ll}^{(kT)} \pb{X G^{(klT)} X^*}_{kl}
\end{equation}
for $T \subset \cal I_M$ (which follows from \eqref{Gab}), \eqref{Gaa}, and
\begin{equation*}
G_{ii}^{(T)} \;=\; \frac{-\sigma_i}{1 + m \sigma_i + \sigma_i ((X G^{(i T)} X^*)_{ii} - m)} \;=\; \sum_{\ell = 0}^{L - 1} m_i^{\ell + 1} \pb{(X G^{(i T)} X^*)_{ii} - m}^\ell + O_\prec \pb{\sigma_i^{3 L + 1} \Psi^L}\,,
\end{equation*}
as follows from \eqref{Gaa}, \eqref{m_i_bound}, and $(X G^{(i T)} X^*)_{ii} - m = O_\prec(\sigma_i^2 \Psi)$ (which may itself be deduced from \eqref{Gaa}). We omit further details.

Finally, the proof of \eqref{UG_imu} is similar to that of \eqref{UGU_ij}. Writing $\cal Z' \deq \sum_{k \in \cal I_M} \sigma_k^{-1} U_{ik} G_{k \mu}$, we estimate $\E \abs{\cal Z'}^p$ for $p \in 2\N$ using the method of \cite[Section 5]{BEKYY}. Instead of \eqref{iso_exp} we use
\begin{equation*}
\sum_{k \in \cal I_M \setminus T} \sigma_k^{-1} U_{ik} G_{k \mu}^{(T)} \;=\; - \sum_{k \in \cal I_M \setminus T} \sigma_k^{-1} U_{ik} G_{kk}^{(T)} (X G^{(kT)})_{k \mu}\,.
\end{equation*}
The rest of the argument is the same as before.
\end{proof}

\section{Self-consistent comparison I: the main argument} \label{sec:comparison1}

In this section we establish Theorem \ref{thm:gen_iso} (i) under the additional assumption that the third moment of all entries of $X$ is zero.

\begin{proposition} \label{prop:comparison_1}
Suppose that the assumptions of Theorem \ref{thm:gen_iso} hold. Suppose moreover that $X$ satisfies the additional condition
\begin{equation} \label{zero_third_moment}
\E X_{i \mu}^3 \;=\; 0\,.
\end{equation}
If the anisotropic local law holds with parameters $(X^{\txt{Gauss}}, \Sigma, \f S)$, then the anisotropic local law holds with parameters $(X, \Sigma, \f S)$.
\end{proposition}

The rest of this section is devoted to the proof of Proposition \ref{prop:comparison_1}. This is the heart of the proof of the anisotropic local law.

\subsection{Sketch of the proof} \label{sec:sketch_compar}
Before giving the full proof of Proposition \ref{prop:comparison_1}, we outline the key ideas of the self-consistent comparison argument on which it relies. For simplicity, we ignore the factors $\ul \Sigma^{-1}$ on the left-hand side of \eqref{isotropic law} (hence obtaining a weaker bound; see \eqref{bound_Sigma}), so that by polarization it suffices to estimate $G_{\f v \f v} - \Pi_{\f v \f v}$ for a deterministic unit vector $\f v \in \R^{\cal I}$.

We introduce a family of interpolating matrices $(X^\theta)_{\theta \in [0,1]}$ satisfying $X^0 = G^{\txt{Gauss}}$ and $X^1 = X$. In order to prove $\abs{G_{\f v \f v}(X^1) - \Pi_{\f v \f v}} \prec \Psi$, it suffices to prove a high moment bound $\E \absb{G_{\f v \f v}(X^1) - \Pi_{\f v \f v}}^p \leq (N^{C \delta} \Psi)^p$ for any fixed $p \in \N$ and $\delta > 0$, and large enough $N$. Since, by assumption, this moment bound holds for $X^1$ replaced by $X^0$, using Gr\"onwall's inequality it suffices to prove that
\begin{equation} \label{self_const_sketch}
\frac{\dd}{ \dd \theta} \E \absb{G_{\f v \f v}(X^\theta) - \Pi_{\f v \f v}}^p \;\leq\; (N^{C \delta} \Psi)^p + C \, \E \absb{G_{\f v \f v}(X^\theta) - \Pi_{\f v \f v}}^p\,.
\end{equation}
Note that we estimate the derivative of the quantity $\E \absb{G_{\f v \f v}(X^\theta) - \Pi_{\f v \f v}}^p$ in terms of itself. It is therefore  important that the arguments $X^\theta$ are the same on both sides. In particular, a Lindeberg-type replacement of the matrix entries one by one would not work, and a continuous interpolation is necessary.  A common choice of continuous interpolation of random matrices, arising for example from Dyson Brownian motion, is $X^\theta = \sqrt{\theta} X^1 + \sqrt{1 - \theta} X^0$, where $X^0$ and $X^1$ are defined on a common probability space and are independent. With this choice of interpolation, however, the differentiation on the left-hand side of \eqref{self_const_sketch} leads to very complicated expressions that are hard to control. Instead, we interpolate by setting
$X^\theta_{i \mu} \deq \chi^\theta_{i\mu} X^1_{i \mu} + (1 - \chi^\theta_{i \mu}) X^0_{i \mu}$,
where $(\chi^\theta_{i \mu})$ is a family of i.i.d.\ Bernoulli random variables, independent of $X^0$ and $X^1$, satisfying $\P(\chi^\theta_{i \mu} = 1) = \theta$ and $\P(\chi^\theta_{i \mu} = 0) = 1 - \theta$. This may also be interpreted as a linear interpolation between the laws of $X^0_{i \mu}$ and $X^1_{i \mu}$. It gives rise to formulas that are simple enough for our comparison argument to work. In fact, after some calculations we find that it suffices to prove
\begin{equation} \label{sketch_main_est}
N^{-n/2} \sum_{i \in \cal I_M} \sum_{\mu \in \cal I_N} \absBB{\E \pbb{\frac{\partial}{\partial X^\theta_{i \mu}}}^n \absb{G_{\f v \f v}(X^\theta) - \Pi_{\f v \f v}}^p} \;\leq\; (N^{C \delta} \Psi)^p + C \, \E \absb{G_{\f v \f v}(X^\theta) - \Pi_{\f v \f v}}^p
\end{equation}
for all $n = 4, \dots, 4p$.

Computing the derivatives on the left-hand side of \eqref{sketch_main_est} leads to a product of terms of the form
\begin{equation}
\txt{(i)} \quad \abs{G_{\f v \f v} - \Pi_{\f v \f v}}\,, \qquad
\txt{(ii)} \quad G_{\f v i}, G_{\f v \mu}, G_{i \f v}, G_{\mu \f v}\,, \qquad
\txt{(iii)} \quad G_{i \mu}, G_{\mu i}\,,
\end{equation}
and their complex conjugates
(here we omit the argument $X^\theta$).
Terms of type (i) are simply kept as they are; they will be put into the second term on the right-hand side of \eqref{sketch_main_est}. Terms of type (ii) are the key to the gain that allows us to compensate losses from several other terms, such as the terms of type (iii) (see blow). The gain is obtained in combination with the summation over $i$ and $\mu$ on the left-hand side of \eqref{sketch_main_est}, according to estimates of the form
\begin{equation} \label{sketch_est_sum_i}
\frac{1}{N} \sum_{i \in \cal I_M} \abs{G_{\f v i}}^2 \;\prec\; \frac{\im G_{\f v \f v} + \eta}{N \eta}\,.
\end{equation}
Note that the right-hand side of \eqref{sketch_est_sum_i} cannot be estimated by either $N^{C \delta}\Psi$ or $\abs{G_{\f v \f v} - \Pi_{\f v \f v}}$, so that a further ingredient is required to obtain \eqref{sketch_main_est}. This ingredient takes the form of an a priori bound
\begin{equation} \label{sketch_a_priori_1}
\im G_{\f v \f v} \;\prec\; \im m + N^{C \delta} \Psi\,,
\end{equation}
which allows us to estimate the right-hand side of \eqref{sketch_est_sum_i} by $N^{C \delta} \Psi^2$. The a priori bound \eqref{sketch_a_priori_1} will be obtained from a bootstrapping, as explained below.

An estimate similar to \eqref{sketch_est_sum_i} may be obtained for $\frac{1}{N} \sum_{i \in \cal I_M} \abs{G_{\f v i}}$ by a simple application of Cauchy-Schwarz to \eqref{sketch_est_sum_i}. However, for $\frac{1}{N} \sum_{i \in \cal I_M} \abs{G_{\f v i}}^d$ with $d \geq 3$, we have to estimate $d - 2$ factors $\abs{G_{\f v i}}$ pointwise to recover \eqref{sketch_est_sum_i}. The trivial bound $\abs{G_{\f vi}} \leq \eta^{-1}$ is far too rough for small $\eta$. A more reasonable attempt is to write $\abs{G_{\f v i}} \leq \abs{G_{\f v i} - \Pi_{\f v i}} + \abs{\Pi_{\f v i}} \prec \max_{\abs{\f w} = 1} \abs{G_{\f w \f w} - \Pi_{\f w \f w}} + 1$, by polarization and $\norm{\Pi} \leq C$. This, however, leads to bounds of the form $\E \max_{\abs{\f w} = 1} \abs{G_{\f w \f w} - \Pi_{\f w \f w}}^p$, which are not affordable. Note that, as explained in \eqref{intr_sc}, bounds of the form $\max_{\abs{\f w} = 1} \E \abs{G_{\f w \f w} - \Pi_{\f w \f w}}^p$ would be affordable, and in fact we shall need such more general bounds in Section \ref{sec:self-const_II}.

The correct way to deal with the pointwise estimate of $\abs{G_{\f v i}}$ is a second type of a priori bound,
\begin{equation} \label{sketch_a_priori_2}
\abs{G_{\f x \f y}} \;\prec\; N^{2 \delta}\,,
\end{equation}
where $\f x$ and $\f y$ are deterministic unit vectors. As \eqref{sketch_a_priori_1}, the estimate \eqref{sketch_a_priori_2} will follow from the bootstrapping explained below. The need for an a priori bound of the form \eqref{sketch_a_priori_2} on the entries of $G$ is also apparent for terms of type (iii), which are also estimated using \eqref{sketch_a_priori_2}.

The preceding discussion was rather cavalier with the various constants $C$ in the factors $N^{C \delta}$. In fact, some care has to be taken to ensure that the factors of $N^{C \delta}$ arising from the use of the a priori bounds \eqref{sketch_a_priori_1} and \eqref{sketch_a_priori_2} are compensated by a sufficiently large number of factors of the form \eqref{sketch_est_sum_i}. Moreover, in order to be able to iterate this argument from large to small scales, as explained in the next paragraph, we need an explicit bound of the constant $C$ in \eqref{sketch_main_est} in terms of the constant $C$ in \eqref{sketch_a_priori_1}.

Finally, we outline the bootstrapping. The above argument can be carried out for $\eta = 1$ using the trivial a priori bounds where the right-hand sides of \eqref{sketch_a_priori_1} and \eqref{sketch_a_priori_2} are replaced by $1$. The idea of the bootstrapping is to fix a small exponent $\delta > 0$ and to proceed from larger scales $\eta$ to smaller scales in multiplicative increments of $N^{-\delta}$, i.e.\ $\eta = 1, N^{-\delta}, N^{-2 \delta}, \dots, N^{-1}$. Suppose that we have proved the anisotropic local law at the scale $\eta = N^{-\delta l}$. In particular, since $\abs{\Pi_{\f x \f y}} \leq C$, we have proved for $\eta = N^{-\delta l}$ that $\abs{G_{\f x \f y}(X, E + \ii \eta)} \prec 1$ for any $X$ satisfying the assumptions \eqref{cond on entries of X} and \eqref{moments of X-1}. Starting from this bound, it is not hard to derive the a priori bound \eqref{sketch_a_priori_2} at $z = E + \ii N^{-\delta} \eta$. Similarly, by using simple monotonicity properties of the Stieltjes transform, we obtain the other a priori bound \eqref{sketch_a_priori_1} at $z = E + \ii N^{-\delta} \eta$. These estimates hold for $X = X^\theta$ for all $\theta \in [0,1]$, since $X^\theta$ satisfies the assumptions \eqref{cond on entries of X} and \eqref{moments of X-1} uniformly in $\theta$. Note that this bootstrapping on $\eta$ is different from the stochastic continuity argument commonly used in establishing local laws (see e.g.\ \cite[Section 4.1]{BEKYY}), since the multiplicative steps of size $N^{-\delta}$ are much too large for a continuity argument to work; all that they provide are the rather crude a priori bounds, \eqref{sketch_a_priori_1} and \eqref{sketch_a_priori_2}.

Since $\delta > 0$ is fixed, the bootstrapping consists of $O(\delta^{-1})$ steps. The resulting estimates contain an extra factor $N^{C \delta}$. Since $\delta > 0$ can be made arbitrarily small, the claim will follow on all scales.

\subsection{Bootstrapping on the spectral scale}

We now move on to the proof of Proposition \ref{prop:comparison_1}. By polarization, it suffices to prove that for any deterministic unit vector $\f v \in \R^{\cal I}$ we have
\begin{equation} \label{claim_prop_comp1}
\scalarb{\f v}{ \ul \Sigma^{-1} \pb{G(z) - \Pi(z)} \ul \Sigma^{-1} \f v} \;=\; O_\prec(\Psi(z))
\end{equation}
for $z \in \f S$. In fact, we shall prove \eqref{claim_prop_comp1} for all $z \in \wh{\f S}$ in some discrete subset $\wh{\f S} \subset \f S$. We take $\wh{\f S}$ to be an $N^{-10}$-net in $\f S$ (i.e.\ for every $z \in \f S$ there is a $w \in \wh{\f S}$ such that $\abs{z - w} \leq N^{-10}$) satisfying $\abs{\wh{\f S}} \leq N^{20}$. The function $z \mapsto \Sigma^{-1} \pb{G(z) - \Pi(z)} \ul \Sigma^{-1}$ is Lipschitz continuous (with respect to the operator norm) in $\f S$ with Lipschitz constant $(C + \norm{X^* X}) N^2$, as follows from Lemma \ref{lem:opnorm} and the bound \eqref{1+xm}. Using Lemma \ref{lem:norm_XX} it is therefore not hard to see that \eqref{claim_prop_comp1} for all $z \in \f S$ follows provided we can prove \eqref{claim_prop_comp1} for all $z \in \wh{\f S}$.

The proof consists of a bootstrap argument from larger scales to smaller scales in multiplicative increments of $N^{-\delta}$. Here
\begin{equation} \label{cond_delta}
\delta \;\in\; \pbb{0, \frac{\tau}{2 C_0}}\,,
\end{equation}
is fixed, where $C_0 > 0$ is a universal constant that will be chosen large enough in the proof ($C_0 = 25$ will work).
For any $\eta \geq N^{-1}$ we define $\eta_0 \leq \eta_1 \leq \dots \leq \eta_L$, where
\begin{equation} \label{def_L_eta}
L \;\equiv\; L(\eta) \;\deq\; \max \hb{l \in \N \col \eta N^{\delta (l -1)} < 1}\,,
\end{equation}
through
\begin{equation} \label{def_eta_k}
\eta_l \;\deq\; \eta N^{\delta l} \qquad (l = 0, \dots, L - 1), \qquad \eta_L \;\deq\; 1\,.
\end{equation}
Note that $L \leq \delta^{-1} + 1$.

We shall always work with a net $\wh{\f S}$ satisfying the following condition.
\begin{definition}
Let $\wh{\f S}$ be an $N^{-10}$-net of $\f S$ satisfying $\abs{\wh{\f S}} \leq N^{20}$ and the condition
\begin{equation*}
E + \ii \eta \in \wh{\f S} \qquad \Longrightarrow \qquad E + \ii \eta_l \in \wh{\f S} \quad \txt{for} \quad l = 1, \dots, L(\eta)\,.
\end{equation*}
\end{definition}

The bootstrapping is formulated in terms of two scale-dependent properties, \textbf{($\f A_m$)} and \textbf{($\f C_m$)}, formulated on the subsets
\begin{equation*}
\wh{\f S}_m \;\deq\; \hb{z \in \wh{\f S} \col \im z \geq N^{-\delta m}}\,.
\end{equation*}
We denote by $\bb S \deq \h{\f v \in \R^{\cal I} \col \abs{\f v} = 1}$ the unit sphere of $\R^{\cal I}$, and abbreviate
\begin{equation} \label{def_bar_v}
\bar {\f v} \;\deq\; \ul \Sigma^{-1} \f v
\end{equation}
for $\f v \in \bb S$.

\begin{itemize}
\item[\textbf{($\f A_m$)}]
For all $z \in \wh{\f S}_m$, all deterministic $\f v \in \bb S$, and all $X$ satisfying \eqref{cond on entries of X} and \eqref{moments of X-1}, we have
\begin{equation} \label{bound_assumption}
\im G_{\bar {\f v} \bar {\f v}}(z) \;\prec\; \im m(z) + N^{C_0 \delta} \Psi(z)\,.
\end{equation}
\item[\textbf{($\f C_m$)}]
For all $z \in \wh{\f S}_m$, all deterministic $\f v \in \bb S$, and all $X$ satisfying \eqref{cond on entries of X} and \eqref{moments of X-1}, we have
\begin{equation} \label{bound_goal}
\absb{G_{\bar {\f v} \bar {\f v}}(z) - \Pi_{\bar {\f v} \bar {\f v}}(z)} \;\prec\; N^{C_0 \delta} \Psi(z)\,.
\end{equation}
\end{itemize}

The bootstrapping is started by the following result.
\begin{lemma} \label{lem:start_induction}
Property \textbf{($\f A_0$)} holds.
\end{lemma}
\begin{proof}
Using the definition of $\Pi$, the assumption \eqref{1+xm}, and Lemmas \ref{lem:opnorm} and \ref{lem:norm_XX}, we find for $z \in \wh {\f S}_0$
\begin{equation*}
\im G_{\bar {\f v} \bar {\f v}}(z) \;\leq\; \im \Pi_{\bar {\f v} \bar {\f v}}(z) + \normb{\ul \Sigma^{-1} \pb{G(z) - \Pi(z)} \ul \Sigma^{-1}} \;\prec\; 1\,.
\end{equation*}
The claim now follows from Lemma \ref{lem:gen_prop_m}.
\end{proof}

The key step is the following result.

\begin{lemma} \label{lem:induction}
Under the assumptions of Proposition \ref{prop:comparison_1}, for any $1 \leq m \leq \delta^{-1}$, property \textbf{($\f A_{m - 1}$)} implies property \textbf{($\f C_m$)}.
\end{lemma}

Using the definition of $\Pi$ and the assumption \eqref{1+xm}, it is easy to see that property \textbf{($\f C_m$)} implies property \textbf{($\f A_m$)}. We therefore conclude from Lemmas \ref{lem:start_induction} and \ref{lem:induction} that \eqref{bound_goal} holds for all $z \in \wh{\f S}$. Since $\delta$ can be chosen arbitrarily small under the condition \eqref{cond_delta}, we conclude that \eqref{claim_prop_comp1} follows for all $z \in \wh{\f S}$, and the proof of Proposition \ref{prop:comparison_1} is complete.

What remains is the proof of Lemma \ref{lem:induction}.
We shall estimate the random variable
\begin{equation} \label{def_Fp}
F_{\f v}(X, z) \;\deq\; \absb{G_{\bar{\f v} \bar{\f v}}(z) - \Pi_{\bar{\f v} \bar{\f v}}(z)}\,,
\end{equation}
where $\f v \in \bb S$ is deterministic.
By Markov's inequality and Definition \ref{def:stocdom}, we conclude that \eqref{bound_goal} follows provided we can prove the following result.
\begin{lemma} \label{lem:moment_bound_1}
Fix $p \in 2 \N$ and $m \leq \delta^{-1}$. Suppose that \eqref{zero_third_moment} and \textbf{($\f A_{m - 1}$)} hold. Then we have
\begin{equation*}
\E F_{\f v}^p(X,z) \;\leq\; (N^{C_0 \delta}\Psi(z))^p
\end{equation*}
for all $z \in \wh{\f S}_m$ and all deterministic $\f v \in \bb S$.
\end{lemma}

\subsection{Interpolation} \label{sec:interpolation}

We use the interpolation outlined in Section \ref{sec:sketch_compar}.
\begin{definition}[Interpolating matrices]
Introduce the notations $X^0 \deq G^{\txt{Gauss}}$ and $X^1 \deq X$.
For $u \in\{0,1\}$, $i \in \cal I_M$, and $\mu \in \cal I_N$, denote by $\rho^u_{i \mu}$ the law of $X^u_{i \mu}$. For $\theta \in [0,1]$ we define the law
\begin{equation*}
\rho^\theta_{i \mu} \;\deq\; \theta \rho^1_{i \mu} + (1 - \theta) \rho^0_{i \mu}\,.
\end{equation*}
Let $(X^0, X^\theta, X^1)$ be a triple of independent $\cal I_M \times \cal I_N$ random matrices, where for $u \in \{0,\theta, 1\}$ the matrix $X^u = (X^u_{i \mu})$ has law
\begin{equation*}
\prod_{i \in \cal I_M} \prod_{\mu \in \cal I_N} \rho^u_{i \mu} (\dd X_{i \mu}^u)\,.
\end{equation*}
For $i \in \cal I_M$, $\mu \in \cal I_N$, and $\lambda \in \R$ we define the matrix
\begin{equation*}
\pb{X^{\theta,\lambda}_{(i\mu)}}_{j \nu} \;\deq\;
\begin{cases}
\lambda & \txt{if } (j,\nu) = (i,\mu)
\\
X^\theta_{j \nu} & \txt{if } (j,\nu) \neq (i,\mu)\,.
\end{cases}
\end{equation*}
We also introduce the matrices
\begin{equation*}
G^\theta(z) \;\deq\; G^\Sigma(X^\theta, z)
\,, \qquad
G^{\theta, \lambda}_{(i \mu)}(z) \;\deq\; G^\Sigma\pb{X^{\theta,\lambda}_{(i \mu)}, z}\,,
\end{equation*}
(recall the notation \eqref{def_G}).
\end{definition}

Throughout the following we shall need to deduce bounds of the form $\E \xi \prec \Gamma$ from $\xi \prec \Gamma$. This is the content of the following lemma, whose proof is a simple application of Cauchy-Schwarz.
\begin{lemma} \label{lem:E_prec}
Let $\Gamma$ be deterministic and satisfy $\Gamma \geq N^{-C}$ for some constant $C > 0$. Let $\xi$ be a random variable satisfying $\xi \prec \Gamma$ and $\E \xi^2 \leq N^C$. Then $\E \xi \prec \Gamma$.
\end{lemma}

In the following applications of Lemma \ref{lem:E_prec}, we shall often not mention the assumption $\E \xi^2 \leq N^C$; it may be always easily verified using rough bounds from Lemma \ref{lem:opnorm}.

We shall prove Lemma \ref{lem:moment_bound_1} by interpolation between the ensembles $X^0$ and $X^1$. The bound for the Gaussian case $X^0$ is given by the following result.
\begin{lemma} \label{lem:comparison_0}
Lemma \ref{lem:moment_bound_1} holds if $X$ is replaced with $X^0$.
\end{lemma}
\begin{proof}
By assumption of Proposition \ref{prop:comparison_1}, we have $F_{\f v}^p(X^0,z) \prec \Psi^p$. In order to conclude the proof using Lemma \ref{lem:E_prec}, we need a rough bound of the form $\E (F_{\f v}^p(X^0,z))^2 \leq N^{C_p}$ for some constant $C_p$ depending on $p$. This easily follows from the first identity in \eqref{G=G} combined with the bound
\begin{equation*}
\normb{\Sigma^{-1} \pb{-\Sigma + \Sigma (1 + m \Sigma)^{-1}} \Sigma^{-1}} \;\leq\; C\,.
\end{equation*}
We omit further details.
\end{proof}

The basic interpolation formula is given by the following lemma, which follows from the fundamental theorem of calculus.

\begin{lemma} \label{lem:interpol}
For $F \col \R^{\cal I_M \times \cal I_N} \to \C$ we have
\begin{equation} \label{interpol}
\E F(X^1) - \E F(X^0) \;=\; \int_0^1 \dd \theta \, \sum_{i \in \cal I_M} \sum_{\mu \in \cal I_N} \qB{ \E F \pB{X^{\theta,X^{1}_{i \mu}}_{(i \mu)}} - \E F \pB{X^{\theta,X^{0}_{i \mu}}_{(i \mu)}}}
\end{equation}
provided all the expectations exist.
\end{lemma}

We shall apply Lemma \ref{lem:interpol} with $F(X) = F_{\f v}^p(X,z)$, where $F_{\f v}(X,z)$ was defined in \eqref{def_Fp}. The main work is to derive the following self-consistent estimate for the right-hand side of \eqref{interpol}. We emphasize that our estimates are always uniform in quantities, such as $\f v$, $\Sigma$, $\theta$, and $z$, that are not explicitly fixed.

\begin{lemma} \label{lem:main_comparison}
Fix $p \in 2 \N$ and $m \leq \delta^{-1}$. Suppose that \eqref{zero_third_moment} and \textbf{($\f A_{m - 1}$)} hold.
Then we have
\begin{equation} \label{main_comparison}
\sum_{i \in \cal I_M} \sum_{\mu \in \cal I_N} \qB{ \E F_{\f v}^p \pB{X^{\theta,X^1_{i \mu}}_{(i \mu)}, z} - \E F_{\f v}^p \pB{X^{\theta,X^0_{i \mu}}_{(i \mu)}, z}} \;=\; O \pB{(N^{C_0 \delta}\Psi)^p + \sup_{\f w \in \bb S} \E F_{\f w}^p(X^\theta, z)}
\end{equation}
for all $\theta \in [0,1]$, all $z \in \wh{\f S}_m$, and all deterministic $\f v \in \bb S$.
\end{lemma}

Combining Lemmas \ref{lem:comparison_0}, \ref{lem:interpol}, and \ref{lem:main_comparison} with a Gr\"onwall argument, we conclude the proof of Lemma \ref{lem:moment_bound_1}, and hence of Proposition \ref{prop:comparison_1}. Note that for this Gr\"onwall argument to work, it is essential that the error term $\sup_{\f w \in \bb S} \E F_{\f w}^p(X^\theta, z)$ on the right-hand side of \eqref{main_comparison} be multiplied by a factor that is bounded (as is implied by the notation $O(\cdot)$). Even a factor $\log N$ multiplying $\sup_{\f w \in \bb S} \E F_{\f w}^p(X^\theta, z)$ would render \eqref{main_comparison} useless.

We remark that, in the proof Proposition \ref{prop:comparison_1} developed in the rest of this section, we in fact obtain a stronger version of \eqref{main_comparison} where the quantity $\sup_{\f w \in \bb S} \E F_{\f w}^p(X^\theta, z)$ on the right-hand side of \eqref{main_comparison} is replaced with $\E F_{\f v}^p(X^\theta, z)$. However, we keep the more general form \eqref{main_comparison} because it will be needed for the proof of Proposition \ref{prop:comparison_2} in Section \ref{sec:self-const_II}.

In order to prove Lemma \ref{lem:main_comparison}, we compare the ensembles $X^{\theta,X^0_{i \mu}}_{(i \mu)}$ and $X^{\theta,X^1_{i \mu}}_{(i \mu)}$ via $X^{\theta,0}_{(i \mu)}$. Clearly, it suffices to prove the following result.
\begin{lemma} \label{lem:main_comparison2}
Fix $p \in 2 \N$ and $m \leq \delta^{-1}$. Suppose that \eqref{zero_third_moment} and \textbf{($\f A_{m - 1}$)} hold.
Then there exists some function $A(\cdot, z)$ such that for $u \in \{0,1\}$ we have
\begin{equation} \label{main_comparison2}
\sum_{i \in \cal I_M} \sum_{\mu \in \cal I_N} \qB{ \E F_{\f v}^p \pB{X^{\theta,X^u_{i \mu}}_{(i \mu)}, z} - \E A \pB{X^{\theta,0}_{(i \mu)}, z}} \;=\; O \pB{(N^{C_0 \delta}\Psi)^p + \sup_{\f w \in \bb S} \E F_{\f w}^p(X^\theta, z)}
\end{equation}
for all $\theta \in [0,1]$, all $z \in \wh{\f S}_m$, and all deterministic $\f v \in \bb S$.
\end{lemma}

In the remainder of this section, we prove Lemma \ref{lem:main_comparison2} for $u = 1$. In order to make use of the assumption \textbf{($\f A_{m - 1}$)}, which involves spectral parameters in $\wh {\f S}_{m - 1}$, to obtain estimates for spectral parameters in $\wh {\f S}_m$, we use the following rough bound.

\begin{lemma} \label{lem:rough_bound}
For any $z  = E + \ii \eta\in \f S$ and $\f x, \f y \in \R^{\cal I}$ we have
\begin{equation*}
\absb{G_{\f x \f y}(z) - \Pi_{\f x \f y}(z)} \;\prec\; N^{2 \delta} \sum_{l = 1}^{L(\eta)} \pB{\im G_{\f x \f x}(E + \ii \eta_l) + \im G_{\f y \f y}( E + \ii \eta_l)} + \abs{\ul \Sigma \f x} \abs{\ul \Sigma \f y}\,,
\end{equation*}
where we recall the definitions of $L(\eta)$ and $\eta_l$ from \eqref{def_L_eta} and \eqref{def_eta_k}.
\end{lemma}

\begin{proof}
From \eqref{def_Pi} and \eqref{decompfullG} we get
\begin{equation*}
G - \Pi \;=\; \sum_{k} \frac{\f u_k \f u_k^*}{\lambda_k - z} -
\begin{pmatrix}
m \Sigma^2 (1 + m \Sigma)^{-1} & 0
\\
0 & I_N
\end{pmatrix}\,.
\end{equation*}
Using \eqref{1+xm} we therefore get
\begin{equation*}
\absb{G_{\f x \f y}(z) - \Pi_{\f x \f y}(z)} \;\leq\; \sum_{k} \frac{\scalar{\f x}{\f u_k}^2}{\abs{\lambda_k - z}} + \sum_{k} \frac{\scalar{\f y}{\f u_k}^2}{\abs{\lambda_k - z}}
+ C \abs{\ul \Sigma \f x} \abs{\ul \Sigma \f y}\,.
\end{equation*}
It suffices to estimate the first term. Setting $\eta_{-1} \deq 0$ and $\eta_{L + 1} \deq \infty$, we define the subsets of indices
\begin{equation*}
U_l \;\deq\; \hb{k \col \eta_{l - 1} \leq \abs{\lambda_k - E} < \eta_l} \qquad (l = 0,1, \dots, L+1)\,.
\end{equation*}
We split the summation
\begin{equation*}
\sum_{k} \frac{\scalar{\f x}{\f u_k}^2}{\abs{\lambda_k - z}} \;=\; \sum_{l = 0}^{L + 1} \sum_{k \in U_l} \frac{\scalar{\f x}{\f u_k}^2}{\abs{\lambda_k - z}}
\end{equation*}
and treat each $l$ separately. For $l = 1, \dots, L$ we find
\begin{multline*}
\sum_{k \in U_l} \frac{\scalar{\f x}{\f u_k}^2}{\abs{\lambda_k - z}} \;\leq\; \sum_{k \in U_l} \frac{\scalar{\f x}{\f u_k}^2 \eta_l}{(\lambda_k - E)^2} \;\leq\; 2 \sum_{k \in U_l} \frac{\scalar{\f x}{\f u_k}^2 \eta_l}{(\lambda_k - E)^2 + \eta_{l - 1}^2} 
\\
\leq\; \frac{2 \eta_l}{\eta_{l - 1}} \im G_{\f x \f x}(E + \ii \eta_{l - 1}) \;\leq\; 2 N^\delta \im G_{\f x \f x}(E + \ii \eta_{l - 1})\,,
\end{multline*}
where the third step easily follows from \eqref{decompfullG}. Using that the map $y \mapsto y \im G_{\f x \f x}(E + \ii y)$ is nondecreasing, we find for $l = 1, \dots, L$ that
\begin{equation*}
\sum_{k \in U_l} \frac{\scalar{\f x}{\f u_k}^2}{\abs{\lambda_k - z}} \;\leq\;  2 N^{2 \delta} \im G_{\f x \f x}(E + \ii \eta_{l \vee 1})\,,
\end{equation*}
as desired.

Next, we estimate
\begin{equation*}
\sum_{k \in U_0} \frac{\scalar{\f x}{\f u_k}^2}{\abs{\lambda_k - z}} \;\leq\; \sum_{k \in U_0} \frac{\scalar{\f x}{\f u_k}^2 \, 2  \eta}{(\lambda_k - E)^2 + \eta^2} \;=\; 2 \im G_{\f x \f x} (E + \ii \eta) \;\leq\; 2 N^\delta \im G_{\f x \f x} (E + \ii \eta_1)\,.
\end{equation*}
Finally, we estimate
\begin{equation*}
\sum_{k \in U_{L+1}} \frac{\scalar{\f x}{\f u_k}^2}{\abs{\lambda_k - z}} \;\leq\; 2\sum_{k \in U_{L+1}} \frac{\scalar{\f x}{\f u_k}^2 \abs{\lambda_k - E} \eta_L}{(\lambda_k - E)^2 + \eta_L^2} \;\prec\; \sum_{k \in U_{L+1}} \frac{\scalar{\f x}{\f u_k}^2 \eta_L}{(\lambda_k - E)^2 + \eta_L^2} \;\leq\; \im G_{\f x \f x}(E + \ii \eta_L)\,,
\end{equation*}
where in the second step we used that $\lambda_k \prec 1$, as follows from \eqref{bound_Sigma} and Lemma \ref{lem:norm_XX}. This concludes the proof.
\end{proof}

\begin{lemma} \label{cor:rough_bound}
Suppose that \textbf{($\f A_{m - 1}$)} holds. Then we have
\begin{equation} \label{a_priori_1}
\ul \Sigma^{-1} \pb{G(z) - \Pi(z)} \ul \Sigma^{-1} \;=\; O_\prec(N^{2 \delta})
\end{equation}
for all $z \in \wh{\f S}_m$.
Moreover, for any deterministic $\f v \in \bb S$ we have
\begin{equation} \label{a_priori_2}
\im G_{\bar {\f v} \bar {\f v}}(z) \;\prec\; N^{2 \delta} \pb{\im m(z) + N^{C_0 \delta} \Psi(z)}
\end{equation}
for all $z \in \wh{\f S}_m$.
\end{lemma}

\begin{proof}
Let $E + \ii \eta \in \wh {\f S}_m$. Then we have $E + \ii \eta_l \in \wh{\f S}_{m - 1}$ for $l = 1, \dots, L(\eta)$. Therefore \eqref{bound_assumption} yields, for any deterministic $\f v \in \bb S$,
\begin{equation*}
\im G_{\bar {\f v} \bar {\f v}}(E + \ii \eta_l) \;\prec\; \abs{\ul \Sigma \bar {\f v}}^2 + \im \scalarb{\bar {\f v}}{\Pi(E + \ii \eta_l) \bar {\f v}} \;\leq\; C \abs{\ul \Sigma \bar {\f v}}^2 \;=\; C\,,
\end{equation*}
where in the second step we used \eqref{1+xm} and the definition of $\Pi$ from \eqref{def_Pi}. The estimate \eqref{a_priori_1} now follows easily using Lemma \ref{lem:rough_bound}.

In order to prove \eqref{a_priori_2}, we remark that if $s(z)$ is the Stieltjes transform of any probability measure, the map $\eta \mapsto \eta \im s(E + \ii \eta)$ is nondecreasing and the map $\eta \mapsto \eta^{-1} \im s(E + \ii \eta)$ is nonincreasing. Abbreviating $z = E + \ii \eta \in \wh{\f S}_m$ and $z_1 = E + \ii \eta_1 \in \wh{\f S}_{m-1}$, we therefore find
\begin{equation*}
\im G_{\bar {\f v} \bar {\f v}}(z) \;\leq\; N^\delta \im G_{\bar {\f v} \bar {\f v}}(z_1) \;\prec\; N^\delta \pb{\im m(z_1) + N^{C_0 \delta} \Psi(z_1)} \;\leq\;  N^{2 \delta} \pb{\im m(z) + N^{C_0 \delta} \Psi(z)}\,,
\end{equation*}
where in the second step we used the assumption \textbf{($\f A_{m - 1}$)}.
\end{proof}

\subsection{Expansion} \label{sec:expansion}
We now develop the main expansion which underlies the proof of Lemma \ref{lem:main_comparison2}. Throughout the rest of this section we suppose that \textbf{($\f A_{m - 1}$)} holds. The rest of the proof is performed at a single $z \in \wh{\f S}_m$, and from now on we therefore consistently omit the argument $z$ from our notation.

Let $i \in \cal I_M$ and $\mu \in \cal I_N$. Define the $\cal I \times \cal I$ matrix $\Delta^\lambda_{(i \mu)}$ through
\begin{equation*}
\pb{\Delta^\lambda_{(i \mu)}}_{st} \;\deq\; \lambda \delta_{i s} \delta_{\mu t} + \lambda \delta_{it} \delta_{\mu s}\,.
\end{equation*}
Thus we get the resolvent expansion, for $\lambda, \lambda' \in \R$ and any $K \in \N$,
\begin{equation} \label{res_exp}
G^{\theta,\lambda'}_{(i \mu)} \;=\; G^{\theta,\lambda}_{(i \mu)} + \sum_{k = 1}^K G^{\theta,\lambda}_{(i \mu)} \pb{\Delta_{(i \mu)}^{\lambda - \lambda'} G^{\theta,\lambda}_{(i \mu)}}^k + G^{\theta,\lambda'}_{(i \mu)} \pb{\Delta_{(i \mu)}^{\lambda - \lambda'} G^{\theta, \lambda}_{(i \mu)}}^{K+1}\,.
\end{equation}
The following result provides a priori bounds for the entries of $G^{\theta,\lambda}_{(i \mu)}$.

\begin{lemma} \label{lem:rough bound}
Suppose that $y$ is a random variable satisfying $\abs{y} \prec N^{-1/2}$. Then
\begin{equation} \label{rough_bound2}
\ul \Sigma^{-1} \pb{G^{\theta, y}_{(i \mu)} - \Pi} \ul \Sigma^{-1} \;=\; O_\prec(N^{2 \delta})
\end{equation}
for all $i \in \cal I_M$ and $\mu \in \cal I_N$.
\end{lemma}
\begin{proof}
We use \eqref{res_exp} with $K \deq 10$, $\lambda' \deq y$, and $\lambda \deq X_{i \mu}^\theta$, so that $G_{(i \mu)}^{\theta, \lambda} = G^\theta$.
By \eqref{1+xm}, we have
\begin{equation} \label{Pi_bounded}
\norm{\ul \Sigma^{-1} \Pi} + \norm{\Pi \ul \Sigma^{-1}} \;\leq\; C\,.
\end{equation}
By assumption \textbf{($\f A_{m - 1}$)}, the conclusions of Lemma \ref{cor:rough_bound} hold for the matrix ensemble $X^\theta$ (since it satisfies \eqref{cond on entries of X} and \eqref{moments of X-1}). Hence, \eqref{rough_bound2} holds for $y = \lambda$, and in particular $\ul \Sigma^{-1} G_{(i \mu)}^{\theta, \lambda} = O_\prec(N^{2 \delta})$ and $G_{(i \mu)}^{\theta, \lambda} \ul \Sigma^{-1} = O_\prec(N^{2 \delta})$. Plugging these estimates into \eqref{res_exp} and using that $\abs{\lambda - \lambda'} \prec N^{-1/2}$, it is easy to estimate the contribution to \eqref{rough_bound2} of all terms of \eqref{res_exp} except the rest term. In order to handle the rest term, we use the rough bounds $\ul \Sigma^{-1} G_{(i \mu)}^{\theta, \lambda'} = O_\prec(N)$ and $G_{(i \mu)}^{\theta, \lambda'} \ul \Sigma^{-1} = O_\prec(N)$, which may be deduced from Lemma \ref{lem:opnorm} and a simple modification of Lemma \ref{lem:norm_XX}.
\end{proof}

To simplify notation, we introduce the function
\begin{equation} \label{def_f}
f_{(i \mu)}(\lambda) \;\deq\; F_{\f v}^p\pb{X_{(i \mu)}^{\theta, \lambda}}\,,
\end{equation}
where we omit the dependence on $\theta$, $p$, $\f v$, and $z$ from our notation. (Recall that $p$ is fixed and all estimates are uniform in $z \in \wh {\f S}_m$,$\f v \in \bb S$, and $\theta \in [0,1]$.) We denote by $f^{(n)}_{(i \mu)}$ the $n$-th derivative of $f_{(i \mu)}$.

The following result is easy to deduce from \eqref{res_exp} and Lemma \ref{lem:rough bound}.

\begin{lemma} \label{lem:taylor_exp}
Suppose that $y$ is a random variable satisfying $\abs{y} \prec N^{-1/2}$. Then for any fixed $n \in \N$ we have
\begin{equation}
\absb{f_{(i \mu)}^{(n)}(y)} \;\prec\; N^{2 \delta (p + n)}\,.
\end{equation}
By Taylor expansion, we therefore have
\begin{equation*}
f_{(i \mu)}(y) \;=\; \sum_{n = 1}^{4p} \frac{y^n}{n!} f_{(i \mu)}^{(n)}(0) + O_\prec(\Psi^p)\,.
\end{equation*}
\end{lemma}

From Lemma \ref{lem:taylor_exp} and Lemma \ref{lem:E_prec}, we get
\begin{align}
\qB{ \E F_{\f v}^p \pB{X^{\theta,X^1_{i \mu}}_{(i \mu)}} - \E F_{\f v}^p \pB{X^{\theta,0}_{(i \mu)}}} &\;=\;
\E \qb{f_{(i \mu)}\pb{X_{i \mu}^1} - f_{(i \mu)}(0)}
\notag \\ \label{moment_matching}
&\;=\; \E f_{(i \mu)}(0) + \frac{1}{2 N} \E f^{(2)}_{(i \mu)}(0) + \sum_{n = 4}^{4p} \frac{1}{n!} \E f_{(i \mu)}^{(n)}(0) \E (X_{i \mu}^1)^n + O_\prec(\Psi^p)\,,
\end{align}
where we used that $X_{i \mu}^1$ has vanishing first and third moments, by \eqref{zero_third_moment}, and its variance is equal to $1/N$. Recalling our goal \eqref{main_comparison2}, we therefore find that we only have to prove
\begin{equation} \label{sc_step 1}
N^{-n/2} \sum_{i \in \cal I_M} \sum_{\mu \in \cal I_N} \absb{\E f_{(i \mu)}^{(n)}(0)} \;=\; O \pB{(N^{C_0 \delta}\Psi)^p+ \sup_{\f w \in \bb S}\E F^p_{\f w}(X^\theta)}
\end{equation}
for $n = 4, \dots, 4p$. Here we used the bounds \eqref{moments of X-1}.

In order to obtain a self-consistent estimate in terms of the matrix $X^\theta$ on the right-hand side of \eqref{sc_step 1}, we need to replace the matrix $X^{\theta,0}_{(i \mu)}$ in $f_{(i \mu)}(0) = F_{\f v}^p\pb{X_{(i \mu)}^{\theta, 0}}$ (and its derivatives) with $X^\theta = X^{\theta, X^\theta_{i \mu}}_{(i \mu)}$. We shall do this by an other application of \eqref{res_exp}.

\begin{lemma} \label{lem:0_to_theta}
Suppose that
\begin{equation} \label{sc_step 2}
N^{-n/2} \sum_{i \in \cal I_M} \sum_{\mu \in \cal I_N} \absb{\E f_{(i \mu)}^{(n)}(X_{i \mu}^\theta)} \;=\; O \pB{(N^{C_0 \delta}\Psi)^p + \sup_{\f w \in \bb S}\E F^p_{\f w}(X^\theta)}
\end{equation}
holds for $n = 4, \dots, 4p$. Then \eqref{sc_step 1} holds for $n = 4, \dots, 4 p$.
\end{lemma}

\begin{proof}
To simplify notation, we abbreviate $f_{(i \mu)} \equiv f$ and $X^\theta_{i \mu} \equiv \xi$. The proof consists of a repeated application of the identity
\begin{equation} \label{Taylo_id1}
\E f^{(l)}(0) \;=\; \E f^{(l)}(\xi) - \sum_{k = 1}^{4p - l} \E f^{(l + k)}(0) \frac{\E \xi^k}{k!} + O_\prec(N^{l/2} \Psi^p)
\end{equation}
for $l \leq 4p$,
which follows from Lemmas \ref{lem:E_prec} and \ref{lem:taylor_exp}.
Fix $n = 4, \dots, 4p$. Using \eqref{Taylo_id1} we get
\begin{align*}
\E f^{(n)}(0) &\;=\; \E f^{(n)}(\xi) -  \sum_{k_1 \geq 1} \ind{n + k_1 \leq 4p} \E f^{(n + k_1)}(0) \frac{\E \xi^{k_1}}{k_1!} + O_\prec(N^{n/2} \Psi^p)
\\
&\;=\; \E f^{(n)}(\xi) -  \sum_{k_1 \geq 1} \ind{n + k_1 \leq 4p} \E f^{(n + k_1)}(\xi) \frac{\E \xi^{k_1}}{k_1!}
\\
&\qquad
+ \sum_{k_1,k_2 \geq 1} \ind{n + k_1 + k_2 \leq 4p} \E f^{(n + k_1 + k_2)}(0) \frac{\E \xi^{k_1}}{k_1!} \frac{\E \xi^{k_2}}{k_2!}
 + O_\prec(N^{n/2} \Psi^p)
 \\
 &\;=\; \cdots \;=\; \sum_{q = 0}^{4p - n} (-1)^q \sum_{k_1, \dots, k_q \geq 1} \indbb{n + \sum_j k_j \leq 4p} \E f^{(n + \sum_j k_j)}(\xi) \prod_{j=1}^q \frac{\E \xi^{k_j}}{k_j!}  + O_\prec(N^{n/2} \Psi^p)\,.
\end{align*}
The claim now follows easily using \eqref{moments of X-1}.
\end{proof}

What therefore remains is to prove \eqref{sc_step 2}. Since it only involves the matrix ensemble $X^\theta$, for the remainder of the proof we abbreviate $X^\theta \equiv X$. Recalling the notation \eqref{def_f}, we find from Lemma \ref{lem:0_to_theta} that it suffices to prove the following result.

\begin{lemma} \label{lem:sc_step 3}
for any $n = 4, \dots, 4p$ we have
\begin{equation} \label{sc_step 3}
N^{-n/2} \sum_{i \in \cal I_M} \sum_{\mu \in \cal I_N} \absbb{\E \pbb{\frac{\partial}{\partial X_{i \mu}}}^n F^p_{\f v}(X)} \;=\; O \pB{(N^{C_0 \delta}\Psi)^p+\sup_{\f w \in \bb S}\E F^p_{\f w}(X)}\,.
\end{equation}
\end{lemma}

\begin{remark} \label{rem:Bernoulli}
We conclude this subsection with a general digression on the algebra of the Bernoulli interpolation method underlying the proof of Proposition \ref{prop:comparison_1}, and in particular establish the formula \eqref{K_n_def} from the introduction. Let $(X_\alpha^0)_{\alpha}$ and $(X_\alpha^1)_{\alpha}$ be arbitrary independent finite families of independent random variables. Define $X^\theta = (X^\theta_\alpha)_{\alpha}$ as in \eqref{bern_interpol}. Then we get
\begin{equation*}
\frac{\dd}{\dd \theta} \E F(X^\theta) \;=\; \sum_{\alpha} \qB{\E F\pB{X^{\theta,X_\alpha^1}_{(\alpha)}} - \E F\pB{X^{\theta,X_\alpha^0}_{(\alpha)}}}\,,
\end{equation*}
where $X^{\theta,x}_{(\alpha)}$ denotes the family obtained from $X^\theta$ by replacing $X_\alpha^\theta$ with $x$. Fix $\alpha$ and abbreviate $f(x) \deq F\pb{X^{\theta,x}_{(\alpha)}}$, $\zeta \deq X_\alpha^1$, $\zeta' \deq X_\alpha^0$, and $\xi \deq X_\alpha^\theta$. We may assume that $\zeta$,  $\zeta'$, and $\xi$ are independent. We want to compute the difference
\begin{equation*}
\E \q{f(\zeta) - f(\zeta')} \;=\; \E \q{f(\zeta) - f(0)}  - \E \q{f(\zeta') - f(0)}\,.
\end{equation*}
Since we are only interested in the algebra of the interpolation, we assume for simplicity that all random variables have finite exponential moments and that $f$ is analytic. (Otherwise, as in the computations above, the expansions have to be truncated.)
Repeating the steps of the proof of Lemma \ref{lem:0_to_theta}, we find
\begin{equation*}
\E (f(\zeta) - f(0))
\;=\; \sum_{q \geq 0} (-1)^q \sum_{n, k_1, \dots, k_q \geq 1} \E f^{(n + k_1 + \cdots + k_q)}(\xi) \, \frac{\E \zeta^{n}}{n!} \prod_{j = 1}^q \frac{\E \xi^{k_j}}{k_j!}
\;=\; \sum_{m \geq 1} K_m(\zeta,\xi) \, \E f^{(m)}(z)\,,
\end{equation*}
where
\begin{equation*}
K_m(\zeta,\xi) \;\deq\; \sum_{q \geq 0} (-1)^q \sum_{n, k_1, \dots, k_q \geq 1} \ind{n + k_1 + \cdots + k_q = m} \frac{\E \zeta^n}{n!}\prod_{j = 1}^q \frac{\E \xi^{k_j}}{k_j!} \;=\; \frac{1}{m!} \pbb{\frac{\dd}{\dd t}}^m \bigg\vert_{t = 0} \frac{\E \me^{t \zeta} - 1}{\E \me^{t \xi}}\,.
\end{equation*}
Together with the same computation for the difference $\E (f(\zeta') - f(0))$, this yields the identity \eqref{K_n_def} from the introduction, with $K_{m,\alpha}^\theta \deq K_m(X_\alpha^1,X_\alpha^\theta) - K_m(X_\alpha^0,X_\alpha^\theta)$.
\end{remark}

\subsection{Introduction of words and conclusion of the proof} \label{sec:words}
In order to prove \eqref{sc_step 3}, we shall have to exploit the detailed structure of the derivatives in on the left-hand side of \eqref{sc_step 3}. The following definition introduces the basic algebraic objects that we shall use.

\begin{definition}[Words] \label{def:words}
We consider words $w \in \cal W$ of even length in the two letters $\{\f i, \f \mu\}$. We denote by $2 n(w)$ the length of the word $w$, where $n(w) = 0,1,2, \dots$. We always use bold symbols to denote the letters of words. We use the notation
\begin{equation} \label{word_letters}
w \;=\; \f t_1 \f s_2 \f t_2 \f s_3 \cdots \f t_n \f s_{n+1}
\end{equation}
for a word of length $2n = 2n(w)$. For $n = 0,1,2, \dots$ we introduce the subset $\cal W_n \deq \{w \in \cal W \col n(w) = n\}$ of words of length $2n$. We require that each word $w \in \cal W_n$ satisfy $\f t_l \f s_{l+1} \in \{\f i \f \mu, \f \mu \f i\}$ for all $1 \leq l \leq n$.

Next, we assign to each letter $*$ its value $[*] \equiv [*]_{i,\mu} \in \cal I$ through
\begin{equation*}
[\f i] \;\deq\; i\,, \qquad [\f \mu] \;\deq\; \mu\,.
\end{equation*}
(Our choice of the names of the two letters is suggestive of their value. Note, however, that it is important to distinguish the abstract letter from its value, which is an index in $\cal I$ and may be used as a summation index.)

Finally, to each word $w \in \cal W$ we assign a random variable $A_{\f v, i,\mu}(w)$ as follows. If $n(w) = 0$ (i.e.\ $w$ is the empty word) we define
\begin{equation*}
A_{\f v,i,\mu}(w) \;\deq\; G_{\bar {\f v} \bar {\f v}} - \Pi_{\bar {\f v} \bar {\f v}}\,.
\end{equation*}
If $n(w) \geq 1$ with $w$ as in \eqref{word_letters} we define
\begin{equation} \label{def_w_1}
A_{\f v,i,\mu}(w) \;\deq\; G_{\bar{\f v} [\f t_1]} G_{[\f s_2] [\f t_2]} \cdots G_{[\f s_n] [\f t_n]} G_{[\f s_{n+1}] \bar {\f v}}\,.
\end{equation}
\end{definition}

Definition \ref{def:words} is constructed so that
\begin{equation*}
\pbb{\frac{\partial}{\partial X_{i \mu}}}^n \pb{G_{\bar {\f v} \bar {\f v}} - \Pi_{\bar {\f v} \bar {\f v}}} \;=\; (-1)^n \sum_{w \in \cal W_n} A_{\f v,i,\mu}(w)
\end{equation*}
for any $n = 0,1,2,\dots$. This may be easily deduced from \eqref{res_exp}. Using Leibnitz's rule \nc we conclude that
\begin{multline*}
\pbb{\frac{\partial}{\partial X_{i \mu}}}^n F^p_{\f v}(X) \;=\; (-1)^n \sum_{n_1, \dots, n_{p/2} \geq 0} \sum_{\tilde n_1, \dots, \tilde n_{p/2} \geq 0} \indbb{\sum_r (n_r + \tilde n_r) = n}\frac{n!}{\prod_r n_r !\tilde n_r!}
\\
\times \prod_{r} \pBB{\sum_{w_r \in \cal W_{n_r}} \sum_{\tilde w_r \in \cal W_{\tilde n_r}} A_{\f v,i, \mu}(w_r) \ol{A_{\f v,i, \mu}(\tilde w_r)}}\,.
\end{multline*}
To prove \eqref{sc_step 3}, it therefore suffices to prove that
\begin{equation} \label{sc_step 4}
N^{-n/2} \sum_{i \in \cal I_M} \sum_{\mu \in \cal I_N} \absbb{\E \prod_{r = 1}^{p/2} \pb{A_{\f v,i, \mu}(w_r) \ol{A_{\f v,i, \mu}(\tilde w_r)}}} \;=\; O \pB{(N^{C_0 \delta}\Psi)^p+\sup_{\f w \in \bb S} \E F^p_{\f w}(X)}
\end{equation}
for $4 \leq n \leq 4p$ and words $w_r, \tilde w_r \in \cal W$ satisfying $\sum_{r} \pb{n(w_r) + n(\tilde w_r)} = n$. To avoid irrelevant notational complications arising from the complex conjugates, we in fact prove that
\begin{equation} \label{sc_step 5}
N^{-n/2} \sum_{i \in \cal I_M} \sum_{\mu \in \cal I_N} \absbb{\E \prod_{r = 1}^{p} A_{\f v,i, \mu}(w_r)} \;=\; O \pB{(N^{C_0 \delta}\Psi)^p+\sup_{\f w \in \bb S} \E F^p_{\f w}(X)}
\end{equation}
for $4 \leq n \leq 4p$ and words $w_r \in \cal W$ satisfying $\sum_{r} n(w_r) = n$. (The proof of \eqref{sc_step 4} is the same with slightly heaver notation.) Treating words $w_r$ with $n(w_r) = 0$ separately, we find that it suffices to prove
\begin{equation} \label{sc_step 6}
N^{-n/2} \sum_{i \in \cal I_M} \sum_{\mu \in \cal I_N} \E \absBB{A_{\f v,i, \mu}(w_0)^{p - q} \prod_{r = 1}^{q} A_{\f v,i, \mu}(w_r)} \;=\; O \pB{(N^{C_0 \delta}\Psi)^p+\sup_{\f w \in \bb S} \E F^p_{\f w}(X)}
\end{equation}
for $4 \leq n \leq 4p$, $1 \leq q \leq p$, and words $w_r \in \cal W$ satisfying $\sum_{r} n(w_r) = n$, $n(w_0) = 0$, and $n(w_r) \geq 1$ for $r \geq 1$. Note that we also have the bound $q \leq n$.

In order to estimate \eqref{sc_step 6}, we introduce the quantity
\begin{equation*}
R_s \;\deq\; \abs{G_{\bar {\f v} s}} + \abs{G_{s \bar {\f v}}}\,.
\end{equation*}

\begin{lemma} \label{lem:w_est}
For $w \in \cal W$ we have the rough bound
\begin{equation} \label{w_est}
\abs{A_{\f v,i, \mu}(w)} \;\prec\; N^{2 \delta (n(w) + 1)}\,.
\end{equation}
Moreover, for $n(w) \geq 1$ we have
\begin{equation} \label{w_est_2}
\abs{A_{\f v,i, \mu}(w)} \;\prec\; (R_i^2 + R_\mu^2) N^{2 \delta (n(w) - 1)}\,.
\end{equation}
Finally, for $n(w) = 1$ we have the sharper bound
\begin{equation} \label{w_est_1}
\abs{A_{\f v,i, \mu}(w)} \;\prec\; R_i R_\mu\,.
\end{equation}
\end{lemma}
\begin{proof}
The estimates \eqref{w_est} and \eqref{w_est_2} follow easily from Lemma \ref{cor:rough_bound} and the definition \eqref{def_w_1}. The estimate \eqref{w_est_1} follows from the constraint $\f t_1 \neq \f s_2$ in Definition \ref{def:words}.
\end{proof}

In addition to the high-probability bounds from Lemma \ref{lem:w_est}, we have the rough estimate
\begin{equation} \label{A_rough}
\E \abs{A_{\f v,i,\mu}(w)}^2 \;\leq\; N^C
\end{equation}
for some $C > 0$ and all $w \in \cal W$ satisfying $n(w) \leq 4p$; this follows easily from Definition \ref{def:words} and Lemma \ref{lem:opnorm}.

By pigeonholing  on the words appearing on the left-hand side of \eqref{sc_step 6}, \nc if $n \leq 2q - 2$ then there exist at least two words $w_r$ satisfying $n(w_r) = 1$. Using Lemma \ref{lem:w_est} we therefore get
\begin{multline} \label{sc_step 7}
\absbb{A_{\f v,i, \mu}(w_0)^{p - q} \prod_{r = 1}^{q} A_{\f v,i, \mu}(w_r)}
\\
\prec\; N^{2 \delta (n + q)} F_{\f v}^{p - q}(X) \pB{\ind{n \geq 2q - 1}  (R_i^2 + R_\mu^2) + \ind{n \leq 2q - 2} R_i^2 R_\mu^2}\,.
\end{multline}
From Lemma \ref{lemma: basic} combined with Lemma \ref{lem:norm_XX} we get
\begin{equation} \label{bound_sum_im}
\frac{1}{N}\sum_{i \in \cal I_M} R_i^2 + \frac{1}{N} \sum_{\mu \in \cal I_N }R_\mu^2 \;\prec\; \frac{\im G_{\bar {\f v} \bar {\f v}} + \eta}{N \eta} \;\prec\; N^{2 \delta} \frac{\im m + N^{C_0 \delta} \Psi}{N \eta} \;\leq\; N^{(C_0 + 2) \delta} \Psi^2\,,
\end{equation}
where in the second step we used \eqref{a_priori_2}, and in the last step the definition of $\Psi$.
Inserting \eqref{bound_sum_im} into \eqref{sc_step 7}, we get using Lemma \ref{lem:E_prec} and \eqref{A_rough} that the left-hand side of \eqref{sc_step 6} is bounded by
\begin{equation*}
N^{-n/2 + 2} N^{3 \delta (n + q)} \E F_{\f v}^{p - q}(X) \pB{\ind{n \geq 2q - 1} \pb{N^{C_0 \delta / 2}\Psi}^2 + \ind{n \leq 2q - 2} \pb{N^{C_0 \delta / 2}\Psi}^4}\,.
\end{equation*}
Using that $\Psi \geq c N^{-1/2}$, we find that the left-hand side of \eqref{sc_step 6} is bounded by
\begin{multline*}
N^{3 \delta (n + q)} \E F_{\f v}^{p - q}(X) \pB{\ind{n \geq 2q - 1} \pb{N^{C_0 \delta / 2}\Psi}^{n - 2} + \ind{n \leq 2q - 2} \pb{N^{C_0 \delta / 2}\Psi}^n}
\\
\leq\; \E F_{\f v}^{p - q}(X) \pB{\ind{n \geq 2q - 1} \pb{N^{(C_0 /2 + 12) \delta}\Psi}^{n - 2} + \ind{n \leq 2q - 2} \pb{N^{(C_0/2 + 12) \delta}\Psi}^n}\,,
\end{multline*}
where we used that $q \leq n$ and $n \geq 4$. Choose $C_0 \geq 25$. Then, by assumption \eqref{cond_delta} on $\delta$ we have $N^{(C_0/2 + 12) \delta}\Psi \leq 1$. Moreover, if $n \geq 4$ and $n \geq 2q - 1$ then $n \geq q + 2$. We conclude that the left-hand side of \eqref{sc_step 6} is bounded by
\begin{equation} \label{sc_step 8}
\E F_{\f v}^{p - q}(X) \pb{N^{C_0 \delta}\Psi}^{q}\,.
\end{equation}
Now \eqref{sc_step 6} follows using Young's inequality. This concludes the proof of \eqref{sc_step 3}, and hence of Lemma \ref{lem:main_comparison2}. The proof of Proposition \ref{prop:comparison_1} is therefore complete.

\section{Self-consistent comparison II: general $X$} \label{sec:self-const_II}

In this section and the next we prove the following result, which is Proposition \ref{prop:comparison_1} without the condition \eqref{zero_third_moment}.

\begin{proposition} \label{prop:comparison_2}
Suppose that the assumptions of Theorem \ref{thm:gen_iso} hold. If the anisotropic local law holds with parameters $(X^{\txt{Gauss}}, \Sigma, \f S)$, then the anisotropic local law holds with parameters $(X, \Sigma, \f S)$.
\end{proposition}

\subsection{Overview of the proof of Proposition \ref{prop:comparison_2}} \label{sec:sketch_proof_7}
The proof of Proposition \ref{prop:comparison_2} builds on that of Proposition \ref{prop:comparison_1}. Throughout this section, we take over the notations of Section \ref{sec:comparison1} without further comment. In particular, as in \eqref{cond_delta}, $C_0$ is some large enough constant (which may depend only on $\tau$) and $\delta > 0$ is fixed and small, satisfying \eqref{cond_delta}.

As in Section \ref{sec:expansion}, we suppose that \textbf{($\f A_{m - 1}$)} holds, and we fix $z \in \wh{\f S}_m$ which we consistently omit from our notation. From the assumption \textbf{($\f A_{m - 1}$)} and Lemma \ref{cor:rough_bound} we get the bounds \eqref{a_priori_1} and \eqref{a_priori_2}, which we summarize here for easy reference:
\begin{equation} \label{sec7_assumpt}
\ul \Sigma^{-1} \pb{G - \Pi} \ul \Sigma^{-1} \;=\; O_\prec(N^{2 \delta}) \,, \qquad
\im G_{\bar {\f v} \bar {\f v}} \;\prec\; N^{2 \delta} \pb{\im m + N^{C_0 \delta} \Psi}
\end{equation}
for all deterministic $\f v \in \bb S$.

The assumption \eqref{zero_third_moment} was used in the proof of Proposition \ref{prop:comparison_1} only in \eqref{moment_matching}, where it ensured that the summation over $n$ starts from $4$ instead of $3$. Without the assumption \eqref{zero_third_moment}, we in addition have to estimate the term $n = 3$ in \eqref{moment_matching}. It therefore suffices to prove the following result.

\begin{lemma} \label{lem:71}
Let $z \in \wh {\f S}_m$, and suppose that \eqref{sec7_assumpt} holds at $z$. There exists a constant $C_0$ such that, for any $\delta$ satisfying \eqref{cond_delta}, the estimate \eqref{sc_step 1} holds for $n = 3$.
\end{lemma}

We shall in fact prove a slightly stronger bound than \eqref{sc_step 1}, where the term $\sup_{\f w \in \bb S} \E F_{\f w}^p(X^\theta, z)$ on the right-hand side of \eqref{sc_step 1} is replaced with $\sup \hb{\E F_{\f w}^p(X^\theta, z)\col \f w \in \{\f v\} \cup \{\f e_\mu \col \mu \in \cal I_N\}}$. However, \eqref{sc_step 1} is simpler to state and strong enough for our purposes.

As in Lemma \ref{lem:0_to_theta}, one may easily replace the matrix $X^{\theta,0}_{i \mu}$ in the definition of $f^{(3)}_{i \mu}(0)$ with $X^\theta$. As in Lemma \ref{lem:sc_step 3}, from now on we abbreviate $X^\theta \equiv X$. Thus, we find that in order to prove Lemma \ref{lem:71} it suffices to prove the following result, which complements Lemma \ref{lem:sc_step 3}.

\begin{lemma} \label{lem:73}
Let $z \in \wh {\f S}_m$, and suppose that \eqref{sec7_assumpt} holds at $z$. There exists a constant $C_0$ such that, for any $\delta$ satisfying \eqref{cond_delta}, the estimate \eqref{sc_step 3} holds for $n = 3$.
\end{lemma}

Next, recall the definition of words from Definition \ref{def:words}. As in \eqref{sc_step 4}--\eqref{sc_step 6}, Lemma \ref{lem:73} is proved provided we can show the following result, which is analogous to \eqref{sc_step 6}.

\begin{lemma} \label{lem:76}
Let $z \in \wh {\f S}_m$, and suppose that \eqref{sec7_assumpt} holds at $z$.
Let $1 \leq q \leq 3$ and choose words $w_1, \dots, w_p \in \cal W$ satisfying $n(w_r) \geq 1$ for $r \leq q$, $n(w_r) = 0$ for  $r \geq q + 1$, and $\sum_r n(w_r) = 3$.
Then \eqref{sc_step 5} holds with $n = 3$.
\end{lemma}

We now explain some key ingredients of the proof of Lemma \ref{lem:76}. The first main difficulty is to extract an additional factor $N^{-1/2}$ from the left-hand side of \eqref{sc_step 5} with $n = 3$, so as to obtain in total a factor $N^{-2}$ that will cancel the number of terms in the summation. The second main difficulty is to extract a factor $\Psi^q$ from the expectation, so as to apply estimates of the form $\E \abs{A_{\f v,i, \mu}(w_0)^{p - q}} \Psi^q \leq (N^{C \delta} \Psi)^p + \E F^p_{\f v}(X)$. In order to extract the factor $N^{-1/2}$, we analyse the dependence of each factor $A_{\f v,i, \mu}(w_r)$ on the $\mu$-th column of $X$, i.e.\ the variables $X_\mu \deq (X_{i \mu})_{i \in \cal I_M}$. We express each term, up to negligible error terms, as a polynomial in $X_\mu$ whose coefficients are independent of $X_\mu$ (i.e.\ $X^{(\mu)}$-measurable). Hence we may take the conditional expectation $\E(\,\cdot\, \vert X^{(\mu)})$ of the product of the variables $X_\mu$, which results in a partition of all entries of $X_\mu$. The extra factor of $N^{-1/2}$ then follows using a parity argument similar to the one first used in \cite{BEKYY}: the degree of the polynomial is odd, so that a perfect matching, which would give a contribution of order $N^0$, is not possible.

Beyond the factor of $N^{-1/2}$, the need to obtain the bound from \eqref{sc_step 5} that is strong enough to close the self-consistent estimate presents significant difficulties, which we outline briefly. These difficulties are roughly of two types. (a) The off-diagonal entries of $G^{(\mu)}$ are in general not small; only entries of $G^{(\mu)} - \Pi$ are small. (b) A priori, using the estimate \eqref{a_priori_1}, the entries of $G^{(\mu)}$ are not bounded by $O_\prec(1)$ but by $O_\prec(N^{2 \delta})$. This would yield a bound on the coefficients of the polynomial in $X_\mu$ that is a high power of $N^{2 \delta}$, which may become too large to conclude the proof.

We deal with difficulty (a) not by estimating individual entries $G_{\f v t}^{(\mu)}$ but by exploiting the extra summations generated by the polynomial expansion, expressing our error bounds in terms of as many summed entries of the form
\begin{equation} \label{sum_j_G}
\frac{1}{N} \sum_{j \in \cal I_M} \abs{G_{\f v j}^{(\mu)}}^d
\end{equation}
as possible, where $d \in \N$. For instance, similarly to \eqref{sketch_est_sum_i}, if $d = 2$ then \eqref{sum_j_G} may be estimated by $\frac{\im G_{\f v \f v} + \eta}{N \eta}$, which is much smaller than the naive bound $N^{4 \delta}$. (See \eqref{bound_sum_hat} below for a precise statement.) The factor $\im G_{\f v \f v}$ may be estimated using the a priori bound \eqref{a_priori_2}. The proof relies on a careful balance of the integers $d$ from \eqref{sum_j_G} versus the number of entries of $G$ that cannot be estimated in this way. We also note that, while $d = 2$ gives a stronger bound than $d = 1$ in \eqref{sum_j_G}, for $d \geq 3$ the bound \eqref{sum_j_G} cannot be improved over the corresponding bound for $d = 2$. In this sense, powers of $d$ larger than $2$ lead to wasted factors of $\Psi$, and we have to make sure that the final combined power of $\Psi$ and $F_{\f v}$ is large enough despite the possibility of indices $d$ larger than $2$.

Next, we deal with difficulty (b) by showing that polynomials whose coefficients consist of many entries of $G^{(\mu)}$, each estimated by $O_\prec(N^{2 \delta})$, also give rise to many factors of the form \eqref{sum_j_G} with large enough $d$, which compensate the powers of $O_\prec(N^{2 \delta})$. Hence, controlling the number of factors of the form \eqref{sum_j_G} and the corresponding indices $d$ is crucial to obtain a sufficiently high power of $\Psi + F_{\f v}$ and a sufficiently small prefactor. The details of this counting are explained in Section \ref{sec:deg_counting}.

We conclude this sketch by noting that, in order to get rid of the omnipresent upper indices $(\mu)$ generated by the expansion in $X_\mu$, we use identities from Lemma \ref{lem:idG}, which generate error terms of the form $F_{\f e_\mu}(X)$ in addition to $F_{\f v}(X)$ (see \eqref{whA_rough4} below). Since $\mu$ may take any value in $\cal I_N$, we need the general self-consistent error term $\sup_{\f w \in \bb S} \E F_{\f w}^p(X)$ on the right-hand side of \eqref{sc_step 5} instead of the simple $\E F_{\f v}^p(X)$.

\subsection{Tagged words}
We now move on to the actual proof of Lemma \ref{lem:76}. The index $\mu \in \cal I_N$ will play a distinguished role.
In a preliminary step, we show that, up to an affordable error, the random variables $A_{\f v,i,\mu}(w)$ in \eqref{sc_step 5} may be replaced with
\begin{equation*}
\wh A_{\f v,i,\mu}(w) \;\deq\;
\begin{cases}
G_{\bar {\f v} \bar {\f v}} - \Pi_{\bar {\f v} \bar {\f v}} - \bar {\f v}(\mu)^2 \pb{G_{\mu \mu} - \Pi_{\mu \mu}} & \text{if } n(w) = 0
\\
A_{\f v,i,\mu}(w) & \text{if } n(w) \geq 1\,.
\end{cases}
\end{equation*}
The motivation behind the definition of $\wh A_{\f v,i,\mu}(w)$ stems from expansion formula \eqref{main_identity} below, which will play a key role in the proof of Proposition \ref{prop:comparison_2}.

\begin{lemma} \label{lem:A wt A}
Lemma \ref{lem:76} holds provided it holds with $A$ replaced by $\wh A$ in \eqref{sc_step 5}.
\end{lemma}
\begin{proof}
We write, dropping the subscripts from $A$ and using $w_0$ to denote the empty word,
\begin{align*}
\absBB{\prod_{r = 1}^{p} A(w_r) - \prod_{r = 1}^{p} \wh A(w_r)} &\;=\; \absBB{\pb{A(w_0)^{p - q} - \wh A(w_0)^{p - q}} \prod_{r = 1}^q A(w_r)}
\\
&\;\leq\; p \pb{\abs{A(w_0)} + \abs{\wh A(w_0)}}^{p - q - 1} \abs{G_{\mu \mu} - \Pi_{\mu \mu}} \bar {\f v}(\mu)^2 \prod_{r = 1}^q \abs{A(w_r)}
\\
&\;\prec\; \pb{F_{\f v}(X) + F_{\f e_\mu}(X)}^{p - q - 1} F_{\f e_\mu}(X) \bar {\f v}(\mu)^2 \prod_{r = 1}^q \abs{A(w_r)}\,.
\end{align*}
 
Next, from Definition \ref{def:words} we conclude, analogously to the proof of Lemma \ref{lem:w_est}, that if $n(w_r) \geq 1$ for $1 \leq r \leq q$ and $\sum_{r = 1}^q n(w_r) = 3$ \nc then
\begin{equation*}
\sum_{i \in \cal I_M} \prod_{r = 1}^q \abs{A(w_r)} \;\prec\;  \sum_{i \in \cal I_M} N^{C \delta} R_i^{q - 1} \;\prec\; N N^{C \delta} (N^{C_0/2} \Psi)^{q - 1}\,,
\end{equation*}
where in the last step we used \eqref{bound_sum_im}. We therefore conclude using Lemma \ref{lem:E_prec} that
\begin{align*}
&\mspace{-30mu}N^{-3/2} \sum_{i \in \cal I_M} \sum_{\mu \in \cal I_N} \absBB{\E \qBB{\prod_{r = 1}^{p} A(w_r) - \prod_{r = 1}^{p} \wh A(w_r)}}
\\
&\;\leq\;
N^{C \delta} N^{-1/2} \sum_{\mu \in \cal I_N} \bar {\f v}(\mu)^2 (N^{C_0/2} \Psi)^{q - 1} \E \pb{F_{\f v}(X) + F_{\f e_\mu}(X)}^{p - q - 1} F_{\f e_\mu}(X)
\\
&\;\leq\; \pb{N^{(C + C_0/2) \delta} \Psi}^p + \sup_{\f w \in \bb S} \E F_{\f w}^p(X)\,,
\end{align*}
where in the last step we used Young's inequality and the estimate $N^{-1/2} \leq \Psi$. Here we also used that $\sum_{\mu \in \cal I_N} \bar {\f v}(\mu)^2 \leq 1$, as follows from the definition of $\bar {\f v}$ from \eqref{def_bar_v} and \eqref{def_ul_Sigma}. This concludes the proof.
\end{proof}

Next, we introduce a family of \emph{tagged words}, which refine the words $w$ from Definition \ref{def:words}. They are designed to encode formulas which refine the definition of $A_{\f v,i,\mu}(w)$ from Definition \ref{def:words}. In order to state them, we shall need a decomposition of entries $G_{\f x \f y}$ into three pieces, denoted by $[G_{\f x \f y}]^0$, $[G_{\f x \f y}]^1$, and $[G_{\f x \f y}]^2$. The identity behind this decomposition is
\begin{equation}  \label{main_identity}
G_{\f x \f y} \;=\; \f x(\mu) \f y(\mu) G_{\mu \mu} + G_{\f x\f y}^{(\mu)} + G_{\mu \mu} (G^{(\mu)} X)_{\f x \mu} (X^* G^{(\mu)})_{\mu \f y} - \f x(\mu) G_{\mu \mu} (X^* G^{(\mu)})_{\mu \f y} - \f y(\mu) G_{\mu \mu} (G^{(\mu)} X)_{\f x \mu}\,,
\end{equation}
which follows from Lemma \ref{lem:idG}.
(Recall that an expression of the form $G_{\f x \f y}^{(\mu)}$ is by definition equal to $\f x^* G^{(\mu)} \f y$, where the matrix multiplication is performed according to Definition \ref{def:matrix_m}.)

\begin{definition} \label{def:G_brack}
Let $\f x, \f y \in \{\bar {\f v}, \f e_\mu, \f e_i\}$ with $\f x$ and $\f y$ not both equal to $\bar {\f v}$. Then the quantities $[G_{\f x \f y}]^0$, $[G_{\f x \f y}]^1$, and $[G_{\f x \f y}]^2$ are uniquely defined by requiring that they do not contain factors of $G_{\mu \mu}$ or $\bar {\f v}(\mu)$, and that \eqref{main_identity} reads
\begin{equation} \label{G_split}
G_{\f x \f y} \;=\; [G_{\f x\f y}]^0 + G_{\mu \mu} [G_{\f x \f y}]^1  + \bar {\f v}(\mu) G_{\mu \mu} [G_{\f x \f y}]^2\,.
\end{equation}
In particular, $[G_{\f x\f y}]^0$ is $X^{(\mu)}$-measurable.
\end{definition}
Explicitly, this decomposition reads
\begin{align*}
G_{\bar {\f v} \mu} &\;=\; \bar{\f v}(\mu) G_{\mu \mu} - G_{\mu \mu} (G^{(\mu)} X)_{\bar{\f v} \mu}
\;=\; \bar {\f v}(\mu) G_{\mu \mu} [G_{\bar {\f v} \mu}]^2 + G_{\mu \mu} [G_{\bar {\f v} \mu}]^1\,,
\\
G_{\bar {\f v} i} &\;=\; G_{\bar {\f v} i}^{(\mu)} + G_{\mu \mu} (G^{(\mu)} X)_{\bar {\f v} \mu} (X^* G^{(\mu)})_{\mu i} - \bar {\f v}(\mu) G_{\mu \mu} (X^* G^{(\mu)})_{\mu i} \;=\; [G_{\bar {\f v} i}]^0 + G_{\mu \mu} [G_{\bar {\f v} i}]^1  + \bar {\f v}(\mu) G_{\mu \mu} [G_{\bar {\f v} i}]^2\,,
\\
G_{i \mu} &\;=\; - G_{\mu \mu} (G^{(\mu)} X)_{i \mu} \;=\; G_{\mu \mu} [G_{i \mu}]^1\,,
\\
G_{i i} &\;=\; G_{ii}^{(\mu)} + G_{\mu \mu} (G^{(\mu)} X)_{i \mu} (X^* G^{(\mu)})_{\mu i} \;=\; [G_{ii}]^0 + G_{\mu \mu} [G_{ii}]^1\,.
\end{align*} 
Throughout the following we abbreviate $\Z_k \deq \qq{0,k-1}$. \nc

\begin{definition}[Tagged words] \label{def:4-words}
Let $w \in \cal W$ be a word from Definition \ref{def:words}. A \emph{tagged word} is a pair $(w, \sigma)$, where $\sigma = (\sigma(l))_{l = 1}^{n(w) + 1} \in \Z_3^{n(w) + 1}$. We assign to each tagged word $(w,\sigma)$ with $w \in \cal W_n$ of the form \eqref{word_letters} a random variable $A_{\f v,i,\mu}(w,\sigma)$ according to the following construction.
\begin{enumerate}
\item
If $n = 0$ (i.e.\ $w$ is the empty word) we define
\begin{align*}
A_{\f v,i,\mu}(w,0) &\;\deq\; G_{\bar {\f v} \bar {\f v}}^{(\mu)} - \Pi_{\bar {\f v} \bar {\f v}}^{(\mu)}\,,
\\
A_{\f v,i,\mu}(w,1) &\;\deq\; (G^{(\mu)} X)_{\bar {\f v} \mu} (X^* G^{(\mu)})_{\mu \bar {\f v}}\,,
\\
A_{\f v,i,\mu}(w,2) &\;\deq\; - (X^* G^{(\mu)})_{\mu \bar {\f v}} - (G^{(\mu)} X)_{\bar {\f v} \mu}\,,
\end{align*}
where we introduced the matrix $\Pi^{(\mu)} \deq \pb{\Pi_{st} \col s,t \in \cal I \setminus \{\mu\}}$.

\item
If $n \geq 1$ we define
\begin{equation} \label{Angeq1def}
A_{\f v,i,\mu}(w,\sigma) \;\deq\; \qb{{  G}_{\bar{\f v} [\f t_1]}}^{\sigma(1)} \qb{G_{[\f s_2] [\f t_2]}}^{\sigma(2)} \cdots \qb{G_{[\f s_n] [\f t_n]}}^{\sigma(n)} \qb{{ G}_{[\f s_{n+1}] \bar{\f v}}}^{\sigma(n+1)}\,.
\end{equation}
\item
We define
\begin{equation*}
\abs{\sigma}_{\f \mu} \;\deq\; \sum_l \ind{\sigma(l) \geq 1}\,, \qquad 
\abs{\sigma}_{\f v} \;\deq\; \sum_{l} \ind{\sigma(l) = 2}\,.
\end{equation*}
(Note that the choice of indices $\f \mu$ and $\f v$ is merely suggestive of the meaning of $\abs{\, \cdot \,}_{\f \mu}$ and $\abs{\, \cdot\,}_{\f v}$; these quantities clearly do not depend on $\mu$ and $\f v$.)
\item
By construction, $A_{\f v,i,\mu}(w,\sigma)$ is a homogeneous polynomial in the variables $\{X_{k \mu}\}_{k \in \cal I_M}$, whose coefficients are $X^{(\mu)}$-measurable. We use $\deg \pb{A_{\f v,i,\mu}(w,\sigma)} \in \N$ to denote its (total) degree.
\end{enumerate}
\end{definition}

From \eqref{G_split}, \eqref{main_identity}, and Definition \ref{def:4-words} we find that for $w \in \cal W_n$ we have
\begin{equation} \label{wh_A_identity}
\wh A_{\f v,i,\mu}(w) \;=\; \sum_{\sigma \in \Z_3^{n+1}} A_{\f v,i,\mu}(w,\sigma) \, (G_{\mu\mu})^{\abs{\sigma}_{\f \mu}} \, \bar{\f v}(\mu)^{\abs{\sigma}_{\f v}}\,.
\end{equation}

Recalling Lemma \ref{lem:A wt A}, we conclude that, in order to prove Lemma \ref{lem:76}, it suffices to prove the following result.

\begin{lemma} \label{lem:82}
Suppose that \eqref{sec7_assumpt} holds. Let $1 \leq q \leq 3$ and choose words $w_1, \dots, w_p \in \cal W$ satisfying $n(w_r) \geq 1$ for $r \leq q$, $n(w_r) = 0$ for  $r \geq q + 1$, and $\sum_r n(w_r) = 3$. For $r = 1, \dots, p$ let $\sigma_r \in \Z_3^{n(w_r) + 1}$.
Then we have for all $\f v \in \bb S$
\begin{equation} \label{sec8 2}
N^{-3/2} \sum_{i \in \cal I_M} \sum_{\mu \in \cal I_N} \absBB{\E \pBB{(G_{\mu\mu})^{d_{\f \mu}} \, \bar{\f v}(\mu)^{d_{\f v}} \prod_{r = 1}^{p} A_{\f v,i, \mu}(w_r, \sigma_r)}} \;=\; O \pB{(N^{C_0 \delta}\Psi)^p+\sup_{\f w \in \bb S}\E F^p_{\f w}(X)}\,,
\end{equation}
were we abbreviated
\begin{equation} \label{def_d_star}
d_* \;\equiv\; d_*(\sigma_1, \dots, \sigma_p) \;\deq\; \sum_{r = 1}^p \abs{\sigma_r}_*
\end{equation}
for $* = \f \mu, \f v$.
\end{lemma}

\subsection{Preliminary estimates and the entries $G_{\mu \mu}$}
We shall require the following rough bounds, which refine those of Lemma \ref{lem:w_est}.

\begin{lemma}[Rough bounds on $A_{\f v,i,\mu}(w,\sigma)$] \label{lem:whA_rough2}
Suppose that \eqref{sec7_assumpt} holds. Let $(w,\sigma)$ be a tagged word. Then
\begin{equation} \label{whA_rough3}
\abs{A_{\f v,i,\mu}(w,\sigma)} \;\prec\; N^{2 \delta (n(w) + 1)}\,.
\end{equation}
Moreover, if $n(w) = 0$ then
\begin{equation} \label{whA_rough4}
\abs{A_{\f v,i,\mu}(w,\sigma)} \;\prec\; N^{(C_0/2 + 1) \delta} \Psi + F_{\f v}(X) + F_{\f e_\mu}(X)\,.
\end{equation}
\end{lemma}

\begin{proof}
Using a large deviation estimate (see \cite[Lemma 3.1]{BEKYY}) combined with Lemmas \ref{lemma: basic} and \ref{lem:norm_XX} we get
\begin{equation} \label{BGX_est}
\absb{(G^{(\mu)} X)_{\bar {\f w} \mu}}^2 + \absb{(X^* G^{(\mu)})_{\mu \bar {\f w}}}^2 \;\prec\; \frac{1}{N} \sum_{j \in \cal I_M} \absb{G^{(\mu)}_{\bar {\f w} j}}^2 + \frac{1}{N} \sum_{j \in \cal I_M} \absb{G^{(\mu)}_{j \bar {\f w}}}^2 \;\prec\; \frac{\im G^{(\mu)}_{\bar {\f w} \bar {\f w}} + \eta}{N \eta}
\end{equation}
for any deterministic $\f w$ satisfying $\abs{\f w} \leq C$.
Using \eqref{main_identity} with $\f x = \f y = \bar {\f w}$,
\eqref{BGX_est}, \eqref{sec7_assumpt}, $\Psi \geq cN^{-1/2}$, and the trivial estimate $\abs{\bar {\f w}(\mu)} \leq 1$, we find
\begin{equation*}
\im G^{(\mu)}_{\bar {\f w} \bar {\f w}} \;=\; \im G_{\bar {\f w} \bar {\f w}} + O_\prec \pBB{\im G_{\mu \mu} + N^{2 \delta} \sqrt{\frac{\im G^{(\mu)}_{\bar {\f w} \bar {\f w}} + \eta}{N \eta}}}\,,
\end{equation*}
We conclude that
\begin{equation*}
\im G^{(\mu)}_{\bar {\f w} \bar {\f w}} \;\prec\; \im G_{\bar {\f w} \bar {\f w}} + \im G_{\mu \mu} + N^{3 \delta} \Psi \;\prec\; N^{2 \delta} \im m + N^{(C_0 + 2) \delta} \Psi\,,
\end{equation*}
where in the second step we used \eqref{sec7_assumpt}. Therefore,
\begin{equation} \label{im_BGB}
\frac{\im G^{(\mu)}_{\bar {\f w} \bar {\f w}}}{N \eta} \;\prec\; N^{(C_0 + 2) \delta} \Psi^2\,.
\end{equation}
Now \eqref{whA_rough4} with $\sigma \geq 1$ follows easily from \eqref{BGX_est} and \eqref{im_BGB} with $\f w = \f v$. Moreover, \eqref{whA_rough4} with $\sigma = 0$ follows using \eqref{main_identity} and \eqref{sec7_assumpt}.

Next, in order to prove \eqref{whA_rough3}, we deduce Definition \ref{def:G_brack} and \eqref{sec7_assumpt}, as well as from \eqref{BGX_est} and \eqref{im_BGB} with $\f w = \f v, \f e_\mu, \Sigma \f e_i$, that
\begin{equation} \label{G_sigma_bound}
\absb{[G_{\bar {\f v} i}]^\sigma} + \absb{[G_{i \bar {\f v}}]^\sigma} + \absb{[G_{\bar {\f v} \mu}]^\sigma} + \absb{[G_{\mu \bar {\f v}}]^\sigma} + \absb{[G_{i \mu}]^\sigma} + \absb{[G_{\mu i}]^\sigma} \;\prec\; N^{2 \delta}\,,
\end{equation}
for all $\sigma = 0, 1,2$. This concludes the proof.
\end{proof}

In order to analyse the left-hand side of \eqref{sec8 2}, we need to exhibit the precise $X$-dependence of the factor $(G_{\mu \mu})^{d_{\f \mu}}$. To that end, we write, using \eqref{Gii}, $G_{\mu \mu} = \p{-z - Y_\mu - Z_\mu}^{-1}$,
where we defined
\begin{equation*}
Z_\mu \;\deq\; (X^* G^{(\mu)} X)_{\mu \mu} - Y_\mu\,, \qquad
Y_\mu \;\deq\; \E \qB{(X^* G^{(\mu)} X)_{\mu \mu} \big\vert X^{(\mu)}} \;=\; \frac{1}{N} \sum_{j \in \cal I_M} G^{(\mu)}_{jj}\,.
\end{equation*}
Using the large deviation estimates from \cite[Lemma 3.1]{BEKYY}, we find $\abs{Z_\mu} \prec N^{-\tau/2 + 2 \delta}$. Since $\abs{G_{\mu \mu}} \prec N^{2 \delta}$ by \eqref{sec7_assumpt}, we therefore deduce using $-\tau/2 + 2 \delta < - 2 \delta$ \nc that
\begin{equation} \label{z-Y_bound}
\abs{-z - Y_\mu}^{-1} \prec N^{2 \delta}\,.
\end{equation}
Hence there exists a constant $K \equiv K(\tau)$ such that
\begin{equation*}
G_{\mu \mu} \;=\; \sum_{k = 0}^K (-z - Y_\mu)^{-k - 1} Z_\mu^k + O_\prec(N^{-10})\,.
\end{equation*}
Plugging in the definition of $Z_\mu$, we find
\begin{equation} \label{Gmumu_exp2}
G_{\mu \mu} \;=\; \sum_{k = 0}^K \cal Y_{\mu,k} (X^* G^{(\mu)} X)_{\mu \mu}^k + O_\prec(N^{-10})\,,
\end{equation}
where the coefficients $\cal Y_{\mu,k}$ are $X^{(\mu)}$-measurable and satisfy the bound $\abs{\cal Y_{\mu, k}} \prec N^{C \delta}$.
Here we used \eqref{z-Y_bound} and the estimate $\abs{G_{jj}^{(\mu)}} \prec N^{2 \delta}$, which follows from \eqref{sec7_assumpt}, \eqref{main_identity}, \eqref{BGX_est}, and \eqref{im_BGB} with $\bar {\f w} = \f x = \f y = \f e_j$.
The precise form of $\cal Y_{\mu,k}$ is unimportant. In order to apply Lemma \ref{lem:E_prec} in the following, we shall also need the rough bound
\begin{equation} \label{cal_Y_rough}
\E \abs{\cal Y_{\mu,k}}^q \;\leq\; N^{C_{p,q}}
\end{equation}
for all $q \in \N$,
which may be easily deduced from $\E (\abs{Y_\mu} + \abs{-z - Y_\mu}^{-1})^{\ell} \leq N^{C_p}$ for any $\ell \leq 16 K$. This latter estimate follows easily from the deterministic estimate $\abs{Y_\mu} + \abs{-z - Y_\mu}^{-1} \leq CN$, where we used that $\im (-z - Y_\mu) \leq \im (-z) \leq -N^{-1}$.

Now we replace the factors $G_{\mu \mu}$ on the left-hand side of \eqref{sec8 2} with the leading term in \eqref{Gmumu_exp2}. From \eqref{Gmumu_exp2}, \eqref{a_priori_1}, \eqref{cal_Y_rough}, and Lemma \ref{lem:E_prec} we get
\begin{multline*}
\E \prod_{r = 1}^{p} \abs{A_{\f v,i,\mu}(w_r, \sigma_r)} \absBB{(G_{\mu \mu})^{d_{\f \mu}} - \pBB{\sum_{k = 0}^K \cal Y_{\mu,k} (X^* G^{(\mu)} X)_{\mu \mu}^k}^{d_{\f \mu}}} \;\prec\; N^{-10} N^{2 \delta d_{\f \mu}} \E \prod_{r = 1}^{p} \abs{A_{\f v,i,\mu}(w_r, \sigma_r)}
\\
\prec\; N^{-7} \pB{\pb{N^{(C_0/2 + 4) \delta}\Psi}^p+\sup_{\f w \in \bb S} \E F^p_{\f w}(X)}\,,
\end{multline*}
where in the last step we used $d_{\f \mu} \leq p$, Lemma \ref{lem:whA_rough2} and that at most three $r \in \qq{1,p}$ satisfy $n(w_r) \geq 1$. We conclude that, to prove Lemma \ref{lem:82}, it suffices to prove under the assumptions of Lemma \ref{lem:82} that
\begin{multline} \label{sec 8 25}
N^{-3/2} \sum_{i \in \cal I_M} \sum_{\mu \in \cal I_N} \bar{\f v}(\mu)^{d_{\f v}} \absa{\E \qa{\prod_{r = 1}^{p} A_{\f v,i,\mu}(w_r, \sigma_r) \pBB{\sum_{k = 0}^K \cal Y_{\mu,k} (X^* G^{(\mu)} X)_{\mu \mu}^k}^{d_{\f \mu}}}}
\\
=\; O \pB{(N^{C_0 \delta}\Psi)^p+\sup_{\f w \in \bb S} \E F^p_{\f w}(X)}\,.
\end{multline}

Next, we expand
\begin{equation*}
\pBB{\sum_{k = 0}^K \cal Y_{\mu,k} (X^* G^{(\mu)} X)_{\mu \mu}^k}^{d_{\f \mu}} \;=\; \sum_{k = 0}^{K d_{\f \mu}} \cal Z_{\mu,k} (X^* G^{(\mu)} X)_{\mu \mu}^k\,,
\end{equation*}
where the coefficients $\cal Z_{\mu,k}$ are $X^{(\mu)}$-measurable and satisfy the bound
\begin{equation} \label{calZ_bound}
\abs{\cal Z_{\mu,k}} \;\prec\; N^{C d_{\f \mu} \delta}\,.
\end{equation}
Moreover, for any tagged word $(w,\sigma)$ we split
\begin{equation} \label{whA01}
A_{\f v,i,\mu}(w,\sigma) \;=\; A_{\f v,i,\mu}^0(w,\sigma) A_{\f v,i,\mu}^+(w,\sigma)\,,
\end{equation}
where $A_{\f v,i,\mu}^0(w,\sigma)$ is equal to $A_{\f v,i,\mu}^0(w,0)$ if $n(w) = 0$ and contains all factors $[G_{\f x \f y}]^0$ from \eqref{Angeq1def} if $n(w) \geq 1$. Hence, $A_{\f v,i,\mu}^0(w,\sigma)$ is $X^{(\mu)}$-measurable and $A_{\f v,i,\mu}^+(w,\sigma)$ is a product of terms of the form
\begin{equation} \label{A1_form}
(G^{(\mu)} X)_{\f x \mu} \quad \txt{or}  \quad (X^* G^{(\mu)})_{\mu \f x} \quad \txt{where} \quad \f x \in \{\bar {\f v}, \f e_i\}\,.
\end{equation}

From \eqref{sec 8 25}, Lemma \ref{lem:E_prec} and \eqref{calZ_bound}, we conclude that, to prove Lemma \ref{lem:82}, it suffices to prove under the assumptions of Lemma \ref{lem:82} and for any nonnegative $k \leq C d_{\f \mu}$ that
\begin{multline} \label{sec8 3}
N^{-3/2} N^{C d_{\f \mu} \delta} \sum_{i \in \cal I_M} \sum_{\mu \in \cal I_N} \bar{\f v}(\mu)^{d_{\f v}} \E \pa{\absBB{\prod_{r = 1}^{p} A_{\f v,i,\mu}^0(w_r, \sigma_r)} \absBB {\E_\mu \pBB{\prod_{r = 1}^{p} A_{\f v,i,\mu}^+(w_r, \sigma_r)} (X^* G^{(\mu)} X)_{\mu \mu}^k}}
\\
=\; O \pB{(N^{C_0 \delta}\Psi)^p+\sup_{\f w \in \bb S} \E F^p_{\f w}(X)}\,,
\end{multline}
where we recall the definition of the conditional expectation $\E_\mu$ from \eqref{cond_exp}.
(Here the applicability of Lemma \ref{lem:E_prec} may be checked using the estimates $\norm{G_M} \leq CN$ and $\E \abs{\cal Z_{\mu,k}}^8 \leq N^{C_p}$. The letter estimate follows from \eqref{cal_Y_rough}.)

\subsection{Degree counting}

For the following we choose $q$, $w_1, \dots, w_p$, and $\sigma_1, \dots, \sigma_r$ as in Lemma \ref{lem:82}. We introduce the abbreviation
\begin{equation} \label{def_d_X}
d_{\f X} \;\equiv\; d_{\f X}(w_1, \sigma_1, \dots, w_p, \sigma_p) \;\deq\; \sum_{r = 1}^p \deg (A_{\f v,i,\mu}(w_r,\sigma_r)) \;=\; \sum_{r = 1}^p \deg (A_{\f v,i,\mu}^+(w_r,\sigma_r))\,.
\end{equation}
(Recall the definition of $\deg (\,\cdot\,)$ from Definition \ref{def:4-words} (iv).)
By definition, the polynomial given by $\prod_{r = 1}^p A_{\f v,i,\mu}^+(w_r, \sigma_r)$ is a product of $d_{\f X}$ factors from the list
\begin{equation} \label{A1_form2}
(G^{(\mu)} X)_{i \mu}\,, \qquad (X^* G^{(\mu)})_{\mu i} \,, \qquad \p{G^{(\mu)} X}_{\bar{\f v} \mu}\,, \qquad \pb{X^* G^{(\mu)}}_{\mu \bar {\f v}}\,.
\end{equation}

\begin{lemma} \label{lem:d odd}
If $d_{\f v} = 0$ then $d_{\f X}$ is odd.
\end{lemma}
\begin{proof}
This may be checked directly using Definitions \ref{def:G_brack} and \ref{def:4-words}. The condition $d_{\f v} = 0$ means that we only take factors $[\,\cdot\,]^0$ and $[\,\cdot\,]^1$ in \eqref{Angeq1def}. The key observation is that the parity of $\deg (A_{\f v,i,\mu}(w_r,\sigma_r))$ for $n \geq 1$ is the same as that of the number of indices $\mu$ appearing on the right-hand side of \eqref{Angeq1def}. The claim then follows from the observation that in $\prod_{r = 1}^p A_{\f v,i,\mu}(w_r)$, which is a product of entries of $G$, the index $\mu$ appears exactly three times.
\end{proof}

We may now estimate the conditional expectation $\E_\mu(\, \cdot\,)$ in \eqref{sec8 3}.
\begin{lemma} \label{lem:87}
Under the assumptions of Lemma \ref{lem:82} and for nonnegative $k \leq C d_{\f \mu}$ we have
\begin{equation} \label{lem_87_est}
\absBB{\E_\mu \pBB{\prod_{r = 1}^{p} A_{\f v,i,\mu}^+(w_r, \sigma_r)} (X^* G^{(\mu)} X)_{\mu \mu}^k} \;\prec\; N^{-\ind{d_{\f v} = 0}/2} \pb{N^{(C_0/2 + C) \delta} \Psi}^{d_{\f X} - \ind{d_{\f X} \geq 3} \ind{d_{\f v} = 0}}\,.
\end{equation}
\end{lemma}
\begin{proof}

To streamline notation, we abbreviate $d \deq d_{\f X}$.
Recalling the form of $A^+(w_r,\sigma_r)$ from \eqref{A1_form}, we find that
\begin{equation} \label{prod_pA1}
\pBB{\prod_{r = 1}^{p} A^+_{\f v, i,\mu}(w_r, \sigma_r)} (X^* G^{(\mu)} X)_{\mu \mu}^k \;=\; \sum_{j_1, \dots, j_{d + 2k} \in \cal I_N} \cal G_{j_1 \dots j_d} \wt {\cal G}_{j_{d+1} \dots j_{d + 2k}}  \prod_{l = 1}^{d + 2 k} X_{j_l \mu}\,,
\end{equation}
where $\cal G_{j_1 \dots j_d}$ is the product of $d$ terms in $\hb{G^{(\mu)}_{j_l \f x}, G^{(\mu)}_{\f x j_l} \col l \in \qq{1,d}, \f x \in \{\bar{\f v},\f e_i\}}$ and $\wt {\cal G}_{j_{d+1} \dots j_{d + 2k}}$ is the product of $k$ terms in $\hb{G^{(\mu)}_{j_{l} j_{l'}} \col l,l' \in \qq{d+1, d_k}}$. In particular, $\wt {\cal G}_{j_{d+1} \dots j_{d + 2k}}$ is $H^{(\mu)}$-measurable and satisfies the bound $\absb{\wt {\cal G}_{j_{d+1} \dots j_{d + 2k}}} \prec N^{2 \delta k}$.
Since $\E_\mu X_{j \mu} = 0$, the contribution of the indices $j_1, \dots, d_{d + 2k}$ after taking the conditional expectation $\E_\mu(\, \cdot\,)$ is nonzero only if each index appears at least twice. We classify the summation on the right-hand side of \eqref{prod_pA1} according to the coincidences of the indices, which results in a sum over partitions of the set $\{1, \dots, d + 2k\}$ whose blocks have size at least two. We denote by $L$ the number of blocks (i.e.\ independent summation indices). For a block $b \subset \{1, \dots, d + 2k\}$ we define $d_b \deq \abs{b \cap \qq{1,d}}$ and $k_b \deq \abs{b \cap \qq{d+1, d+2k}}$. The number $d_b$ denotes the number of coinciding summation indices in the block $b$ that originate from the first factor on the left-hand side of \eqref{prod_pA1}, and the number $d_b$ the number of coinciding summation indices in the block $b$ that originate from the second factor. Indexing the blocks as $\qq{1,L}$ (and hence writing $d_l$ and $k_l$ with $l \in \qq{1,L}$ instead of $d_b$ and $k_b$) and renaming the independent summation indices $j_1, \dots, j_L$, we therefore get
\begin{multline} \label{727}
\absBB{\E_\mu \pBB{\prod_{r = 1}^{p} A^+_{\f v, i,\mu}(w_r, \sigma_r)} (X^* G^{(\mu)} X)_{\mu \mu}^k}
\\
\leq\; C_p N^{2 \delta k} \max_{L} \max_{\{d_l\}} \max_{\{k_l\}} \sum_{j_1, \dots, j_L} \prod_{l = 1}^L \pBB{ \E \abs{X_{j_l \mu}}^{d_l + k_l} \pB{\absb{G^{(\mu)}_{\bar{\f v} j_l}} + \absb{G^{(\mu)}_{j_l \bar{\f v}}} + \absb{G^{(\mu)}_{j_l i}} + \absb{G^{(\mu)}_{i j_l}}}^{d_l}}\,,
\end{multline}
where the maxima are taken over $L \in \N$ and $d_l,k_l \geq 0$ satisfying
\begin{equation} \label{max_cond}
\sum_{l = 1}^L d_l \;=\; d \,, \qquad \sum_{l = 1}^L k_l \;=\; 2k \,, \qquad d_l + k_l \;\geq\; 2\,.
\end{equation}
Here the constant $C_p$ accounts for the immaterial constants depending on $p$ arising from the combinatorics of all partitions. For $L$, $\{d_l\}$, and $\{k_l\}$ as above, we may estimate
\begin{multline*}
C_p N^{2 \delta k} \sum_{j_1, \dots, j_L} \prod_{l = 1}^L \pBB{ \E \abs{X_{j_l \mu}}^{d_l + k_l} \pB{\absb{G^{(\mu)}_{\bar{\f v} j_l}} + \absb{G^{(\mu)}_{j_l \bar{\f v}}} + \absb{G^{(\mu)}_{j_l i}} + \absb{G^{(\mu)}_{i j_l}}}^{d_l}}
\\
\leq\; C_p N^{2 \delta k} N^{-d/2 - k} \prod_{l = 1}^L \pBB{ \sum_j \pB{\absb{G^{(\mu)}_{\bar{\f v} j}} + \absb{G^{(\mu)}_{j \bar{\f v}}} + \absb{G^{(\mu)}_{j i}} + \absb{G^{(\mu)}_{i j}}}^{d_l}}\,,
\end{multline*}
where we used \eqref{moments of X-1}.
Now we use the estimate
\begin{equation} \label{bound_sum_hat}
\sum_{j \in \cal I_M} \absb{G^{(\mu)}_{\bar{\f w} j}}^d \;\prec\; N \pb{N^{(C_0/2 + 1) \delta} \Psi}^{d \wedge 2} N^{2 \delta [d - 2]_+}\,,
\end{equation}
which follows from \eqref{BGX_est}, \eqref{im_BGB}, and \eqref{sec7_assumpt}. Hence we get
\begin{multline} \label{730}
\absBB{\E_\mu \pBB{\prod_{r = 1}^{p} A^+(w_r, \sigma_r)} (X^* G^{(\mu)} X)_{\mu \mu}^k}
\\
\prec\;
\max_{L} \max_{\{d_l\}} \max_{\{k_l\}} N^{-d/2 - k + L} N^{2 \delta k} N^{2 \delta \sum_l [d_l - 2]_+} \pb{N^{(C_0/2 + 1) \delta} \Psi}^{\sum_l (2 \wedge d_l)}\,,
\end{multline}
where the maxima are subject to the same conditions as above.

Next, from $\sum_l (d_l + k_l) = 2k + d$ and $d_l + k_l \geq 2$ we deduce that
\begin{equation*}
L \;=\; k + \frac{d}{2} - \frac{1}{2} \sum_l [d_l + k_l - 2]_+\,.
\end{equation*}
Together with $\sum_l [d_l - 2]_+ \leq \sum_l d_l = d$, this gives
\begin{multline}
N^{-d/2 - k + L} N^{2 \delta k} N^{2 \delta \sum_l [d_l - 2]_+} \pb{N^{(C_0/2 + 1) \delta} \Psi}^{\sum_l (2 \wedge d_l)}
\\
\leq\; N^{2 \delta (d+k)} \pb{N^{(C_0/2 + 1) \delta} \Psi}^{\sum_l (2 \wedge d_l)} N^{-\frac{1}{2} \sum_l [d_l + k_l - 2]_+}\,.
\end{multline}

Suppose first that $d_{\f v} = 0$.
From Lemma \ref{lem:d odd} and \eqref{max_cond} we find that $\sum_l [d_l + k_l - 2]_+ \geq 1$. Using $\Psi \geq N^{-1/2}$, we therefore get
\begin{align*}
&\mspace{-30mu} N^{-d/2 - k + L} N^{2 \delta k} N^{2 \delta \sum_l [d_l - 2]_+} \pb{N^{(C_0/2 + 1) \delta} \Psi}^{\sum_l (2 \wedge d_l)}
\\
&\;\leq\;
N^{2 \delta (d+k)} N^{-1/2} \pb{N^{(C_0/2 + 1) \delta} \Psi}^{\sum_l (2 \wedge d_l) + \sum_l [d_l + k_l - 2]_+ - 1}
\\
&\;\prec\;
N^{2 \delta (d+k)} N^{-1/2} \pb{N^{(C_0/2 + 1) \delta} \Psi}^{d - \ind{d \geq 3}}
\end{align*}
where in the last step we used that $[d_l + k_l  - 2]_+ \geq [d_l - 2]_+$, and $\sum_l (2 \wedge d_l) + \sum_l [d_l - 2]_+ = d$. For the case $d = 1$, we also used that $\sum_l (2 \wedge d_l) + \sum_l [d_l + k_l - 2]_+ - 1 \geq 1$.

Next, we claim that
\begin{equation} \label{sigma_leq_d 2}
d_{\f \mu} \;\leq\; d + 6\,.
\end{equation}
This follows from $\abs{\sigma}_{\f \mu} \leq \deg \pb{A_{\f v,i,\mu}(w,\sigma)} + 2 \, \ind{n(w) \geq 1}$, which may be checked using Definition \ref{def:4-words}.
Since $k \leq C d_{\f \mu}$, we deduce that $k \leq C d$, and therefore conclude the proof of \eqref{lem_87_est} for the case $d_{\f v} = 0$.

Finally, if $d_{\f v} \geq 1$ then we only have the trivial lower bound $\sum_l [d_l + k_l  - 2]_+ \geq 0$, from which \eqref{lem_87_est} may be obtained by a minor modification of the above argument.
\end{proof}

We conclude that the left-hand side of \eqref{sec8 3} is bounded by
\begin{multline} \label{sec8 4}
N^{-3/2 -\ind{d_{\f v} = 0}/2} \sum_{\mu \in \cal I_N} \abs{\bar{\f v}(\mu)}^{d_{\f v}} \sum_{i \in \cal I_M}  \E \absBB{\prod_{r = 1}^{p} A_{\f v,i,\mu}^0(w_r, \sigma_r)} \pb{N^{(C_0/2 + C) \delta} \Psi}^{d_{\f X} - \ind{d_{\f X} \geq 3} \ind{d_{\f v} = 0}}
\\
\leq\;
\max_{\mu \in \cal I_N} N^{-1} \sum_{i \in \cal I_M} \E \absBB{\prod_{r = 1}^{p} A_{\f v,i,\mu}^0(w_r, \sigma_r)} \pb{N^{(C_0/2 + C) \delta} \Psi}^{d_{\f X} - \ind{d_{\f X} \geq 3} \ind{d_{\f v} = 0} + \ind{d_{\f v} \geq 2}}\,,
\end{multline}
where we used \eqref{sigma_leq_d 2} with $d = d_{\f X}$ and $\Psi \geq N^{-1/2}$.

Next, we define $f_{\f i}$ to be the number of factors of the form $[G_{\bar{\f v} i}]^0$ or $[G_{i \bar {\f v}}]^0$ in the polynomial $\prod_{r = 1}^q A_{\f v,i,\mu}(w_r, \sigma_r)$ (recall Definition \ref{def:4-words} (ii)). (Here we use the letter $f$ instead of $d$ to emphasize that the definition of $f_{\f i}$ is not analogous to that of $d_{\f \mu}$.) Then we get from \eqref{whA_rough3}, \eqref{BGX_est}, and \eqref{im_BGB} that
\begin{equation} \label{818}
N^{-1} \sum_{i \in \cal I_M} \absBB{\prod_{r = 1}^{q} A_{\f v,i,\mu}^0(w_r, \sigma_r)} \;\prec\; N^{C \delta} N^{-1}\sum_{i \in \cal I_M} \pb{\abs{G^{(\mu)}_{i \bar{\f v}}}^{f_{\f i} \wedge 2} + \abs{G_{\bar {\f v} i}^{(\mu)}}^{f_{\f i} \wedge 2}} \;\prec\;   N^{C \delta} \nc \pb{N^{(C_0/2 + 1) \delta} \Psi}^{f_{\f i} \wedge 2}\,.
\end{equation} 
Moreover, by \eqref{whA_rough4} we have
\begin{equation} \label{819}
\absBB{\prod_{r = q+1}^{p} A_{\f v,i,\mu}^0(w_r, \sigma_r)} \;\prec\; \pB{N^{(C_0/2 + 1) \delta} \Psi + F_{\f v}(X) + F_{\f e_\mu}(X)}^{p - q - \sum_{r = q+1}^p \abs{\sigma_r}_{\f \mu}}
\end{equation}
Plugging \eqref{818} and \eqref{819} into \eqref{sec8 4}, and recalling Lemma \ref{lem:E_prec}, yields
\begin{multline} \label{sec8 5}
\txt{LHS of \eqref{sec8 3}} \;\leq\; \pb{N^{(C_0/2 + C) \delta} \Psi}^{d_{\f X} - \ind{d_{\f X} \geq 3} \ind{d_{\f v} = 0} + \ind{d_{\f v} \geq 2} + f_{\f i} \wedge 2}
\\
\times  \sup_{\f w \in \bb S} \E \pB{N^{(C_0/2 + 1) \delta} \Psi + F_{\f w}(X)}^{p - q - \sum_{r = q+1}^p \abs{\sigma_r}_{\f \mu}}\,,
\end{multline}
  where we absorbed the factor $N^{C \delta}$ on the right-hand side of \eqref{818} into the parentheses on the first line using that $d_{\f X} + f_{\f i} > 0$, as follows from Defnition \ref{def:4-words} and the assumptions of Lemma \ref{lem:82}. \nc

Recalling that $N^{(C_0/2 + C) \delta} \Psi + N^{-3 \delta} F_{\f w}(X) \prec N^{-\delta}$ for small enough $\delta$ by \eqref{sec7_assumpt}, we find using Young's inequality that in order to prove \eqref{sec8 3}, it suffices to prove
\begin{equation} \label{main_comb_est}
d_{\f X} - \ind{d_{\f X} \geq 3} \ind{d_{\f v} = 0} + \ind{d_{\f v} \geq 2} + f_{\f i} \wedge 2 - q - \sum_{r = q+1}^p \abs{\sigma_r}_{\f \mu} \;\geq\; 0\,.
\end{equation}

\subsection{Power counting: proof of \eqref{main_comb_est}} \label{sec:deg_counting}
The proof of \eqref{main_comb_est} requires precise information about the structure of the factors of $A_{\f v,i,\mu}(w_r, \sigma_r)$. To that end, we introduce the quantities
\begin{equation*}
d_{\f X}^< \;\deq\; \sum_{r = 1}^q \deg (A_{\f v,i,\mu}(w_r,\sigma_r)) \,, \qquad d_{\f X}^> \;\deq\; \sum_{r = q+1}^p \deg (A_{\f v,i,\mu}(w_r,\sigma_r))
\end{equation*}
and
\begin{equation*}
d_*^< \;\deq\; \sum_{r = 1}^q \abs{\sigma_r}_* \,, \qquad
d_*^> \;\deq\; \sum_{r = q + 1}^p \abs{\sigma_r}_*
\end{equation*}
for $* = \f \mu, \f v$. Hence, $d_* = d_*^< + d_*^>$ for $* = \f \mu, \f v, \f X$. (Recall the definitions \eqref{def_d_star} and \eqref{def_d_X}.) Moreover,
\begin{equation}
2 d_{\f \mu}^> \;=\; d_{\f X}^> + d_{\f v}^>\,, \qquad d_{\f \mu}^> \;\geq\; d_{\f v}^>\,,
\end{equation}
as may be easily checked from Definition \ref{def:4-words}. (Recall that $n(w_r) = 0$ for $r \geq q + 1$.) We conclude that \eqref{main_comb_est} holds provided we can show that
\begin{equation} \label{main_comb_est_2}
d_{\f X}^< + f_{\f i} \wedge 2 - q
- \ind{d_{\f X} \geq 3} \ind{d_{\f v} = 0} + \ind{d_{\f v} \geq 2} \;\geq\; 0\,.
\end{equation}
The following arguments rely on the algebraic structure of $A_{\f v,i,\mu}(w_r)$ and $A_{\f v,i,\mu}(w_r, \sigma_r)$ from Definitions \ref{def:words} and \ref{def:4-words} respectively, to which we refer tacitly throughout the rest of the proof.

We claim that
\begin{equation} \label{mX_mi}
d_{\f X}^< + f_{\f i} \;\geq\; 1\,.
\end{equation}
To see this, we note that each factor $[G_{\bar {\f v} i}]^1$ contributes one to $d_{\f X}^<$ and each factor $[G_{\bar {\f v} i}]^0$ contributes one to $f_{\f i}$; the same holds for $[G_{i \bar {\f v}}]^*$ (where $* = 0,1$). Since there are exactly three indices $i$ in the factors of $\prod_{r = 1}^q A_{\f v,i,\mu}(w_r)$, we find that $\prod_{r = 1}^q A_{\f v,i,\mu}(w_r)$ must contain at least one of the factors $G_{\bar {\f v} i}$, $G_{i\bar {\f v}}$, $G_{i \mu}$, or $G_{\mu i}$. We therefore conclude that, no matter the choice of $\sigma_1, \dots, \sigma_r$, we always have \eqref{mX_mi}.

Furthermore, if $q = 3$ then there are exactly three factors $G_{\bar {\f v} i}$ or $G_{i \bar {\f v}}$ in $\prod_{r = 1}^q A_{\f v,i,\mu}(w_r)$. We deduce that \eqref{mX_mi} may be improved to
\begin{equation*}
d_{\f X}^< + f_{\f i} \wedge 2 \;\geq\; 1 + \ind{q = 3}\,.
\end{equation*}
We conclude that \eqref{main_comb_est_2} holds provided that (i) $d_{\f v} \geq 2$ or (ii) $d_{\f v} = 1$ and $q = 1$.

Next, as noted above, each factor $G_{\bar {\f v} i}$ or $G_{i \bar {\f v}}$ of $\prod_{r = 1}^q A_{\f v,i,\mu}(w_r)$ contributes one to $d_{\f X}^< + f_{\f i}$. Moreover, each factor $G_{\bar {\f v} \mu}$ or $G_{\mu \bar {\f v}}$ contributes one to $d_{\f X}^< + d_{\f v}^<$. Since the number of factors $G_{\bar {\f v} i}$, $G_{i \bar {\f v}}$, $G_{\bar {\f v} \mu}$ or $G_{\mu \bar {\f v}}$ in $\prod_{r = 1}^q A_{\f v,i,\mu}(w_r)$ is exactly $2 q$, we find
\begin{equation*}
d_{\f X}^< + f_{\f i} \wedge 2 + \ind{f_{\f i} = 3} + d_{\f v}^< \;\geq\; 2q\,,
\end{equation*}
where we used that $f_{\f i} \leq 3$. We conclude that \eqref{main_comb_est_2} holds provided that
\begin{equation} \label{main_comb_est_3}
q - d_{\f v}^< - \ind{f_{\f i} = 3}
- \ind{d_{\f X} \geq 3} \ind{d_{\f v} = 0} \;\geq\; 0\,,
\end{equation}
which is easy to check in the case $d_{\f v} \leq 1$ and $q \geq 2$.

All that remain therefore is the case $d_{\f v} = 0$ and $q = 1$. In that case we have $f_{\f i} \leq 2$ and $d_{\f v}^< = 0$, so that \eqref{main_comb_est_3} holds also in this case. This concludes the proof of \eqref{main_comb_est}.

The proof of \eqref{sec8 3}, and hence of Lemma \ref{lem:82}, is therefore complete. This concludes the proof of Lemma \ref{lem:76}.

We have hence proved Proposition \ref{prop:comparison_2}. Recalling Proposition \ref{prop:Gaussian}, we conclude the proof of Theorem \ref{thm:gen_iso} (i).

\section{Averaged local law} \label{sec:avg_proof}

In this section we complete the proof of Theorem \ref{thm:gen_iso} by proving its part (ii). Bearing applications to eigenvalue rigidity in Section \ref{sec:rigidity} below in mind, we in fact prove a somewhat stronger result. Recall the definition of $m_N$ from \eqref{def_m_N}. We generalize Definition \ref{def:local_laws} (iii) by saying that the \emph{averaged local law holds with parameters $(X, \Sigma, \f S, \Phi)$} if
$
\abs{m_N(z) - m(z)} \prec \Phi(z)$
uniformly in $z \in \f S$. Here $\Phi \col \f S \to (0,\infty)$ is a deterministic control parameter satisfying
\begin{equation} \label{Phi_cond}
c \Psi^2 \;\leq\; \Phi \;\leq\; N^{-\tau/2}\,.
\end{equation}
for some constant $c > 0$.
In particular, we recover Definition \ref{def:local_laws} (iii) by setting $\Phi = (N \eta)^{-1}$.
In this section we prove the following result, which also completes the proof of Theorem \ref{thm:gen_iso}.

\begin{proposition} \label{prop:gen_avg}
Suppose that the assumptions of Theorem \ref{thm:gen_iso} hold. Let $\Phi$ satisfy \eqref{Phi_cond}. Suppose that the entrywise local law holds with parameters $(X^{\txt{Gauss}}, D, \f S)$ and that the averaged local law holds with parameters $(X^{\txt{Gauss}}, D, \f S, \Phi)$. Then the averaged local law holds with parameters $(X, \Sigma, \f S, \Phi)$.
\end{proposition}

\begin{proof}
The proof is similar to that of Propositions \ref{prop:comparison_1} and \ref{prop:comparison_2}, and we only explain the differences.  Note that now there is no bootstrapping, since the necessary a priori bounds are obtained from the anisotropic local law. In analogy to \eqref{def_Fp}, we define
\begin{equation*}
\wt F(X,z) \;\deq\; \abs{m_N(z) - m(z)} \;=\; \absbb{\frac{1}{N} \sum_{\nu \in \cal I_N} G_{\nu \nu}(z) - m(z)}\,.
\end{equation*}
Following the argument leading up to \eqref{sc_step 3} and Lemma \ref{lem:73} to the letter, we find that it suffices to prove that 
\begin{equation*}
N^{-n/2} \sum_{i \in \cal I_M} \sum_{\mu \in \cal I_N} \absBB{\E \pbb{\frac{\partial}{\partial X_{i \mu}}}^n \wt F^p(X)} \;=\; O \pB{\pb{N^{\delta} \Psi^2}^p+\E \wt F^p(X)}
\end{equation*}
for any $n = 3, \dots, 4p$ and $z \in \f S$. Here $\delta > 0$ is an arbitrary positive constant. We use the words $w$ from Definition \ref{def:words}.
Analogously to \eqref{sc_step 5} and Lemma \ref{lem:76}, it suffices to prove, for all $n = 3, \dots, 4p$, that
\begin{equation} \label{avg_pf_claim}
N^{-n/2} \sum_{i \in \cal I_M} \sum_{\mu \in \cal I_N} \absBB{\E \prod_{r = 1}^{p} \pBB{\frac{1}{N} \sum_{\nu \in \cal I_N} A_{\f e_\nu,i, \mu}(w_r)}}  \;=\; O \pB{\pb{N^{\delta}\Psi^2}^p+\E \wt F^p(X)} \qquad \txt{for} \quad \sum_r n(w_r) = n\,.
\end{equation}

Next, from the part (i) of Theorem \ref{thm:gen_iso} we get $\abs{G_{st} - \Pi_{st}} \prec \Psi$ for $s,t \in \cal I$ and $z \in \f S$. From \eqref{def_w_1} we deduce that
\begin{equation} \label{wt_A_estimate}
\absBB{\frac{1}{N} \sum_{\nu \in \cal I_N} A_{\f e_\nu,i, \mu}(w)} \;\prec\; \Psi^2 \qquad \txt{for} \quad n(w) \geq 1\,.
\end{equation}
For $n \geq 4$, the claim \eqref{avg_pf_claim} easily follows from \eqref{wt_A_estimate}  and an application of Young's inequality to the factors $r$ satisfying $n(w_r) = 0$ and $n(w_r) \geq 1$; we omit the details. \nc What remains, therefore, is to verify \eqref{avg_pf_claim} for $n = 3$.

Let $n = 3$. As in Lemma \ref{lem:A wt A}, it is easy to check that it suffices to prove \eqref{avg_pf_claim} with $A_{\f e_\nu,i, \mu}(w_r)$ replaced by $\wh A_{\f e_\nu,i, \mu}(w_r)$. This means that we replace the sum $\sum_{\nu \in \cal I_N}$ with $\sum_{\nu \in \cal I_N \setminus \{\mu\}}$. From \eqref{wh_A_identity} and Definition \ref{def:4-words}, we conclude that it suffices to prove that
\begin{equation*}
N^{-3/2} \sum_{i \in \cal I_M} \sum_{\mu \in \cal I_N} \absBB{\frac{1}{N^p} \sum_{\nu_1, \dots, \nu_p \in \cal I_M \setminus \{\mu\}} \E \prod_{r = 1}^{p} A_{\f e_{\nu_r},i, \mu}(w_r, \sigma_r) (G_{\mu \mu})^{d_{\f \mu}}}  \;=\; O \pB{\pb{N^{\delta}\Psi^2}^p+\E \wt F^p(X)}\,,
\end{equation*}
where $\sum_r n(w_r) = 3$, $n(w_r) \geq 1$ for $r \in \qq{1,q}$ and $n(w_r) = 0$ for $r \geq q+1$, and $\sigma_r \in \Z_2^{n(w_r) + 1}$. Here we recall the index $d_{\f \mu}$ defined in Lemma \ref{lem:82}. Note that, since $\nu_r \neq \mu$, in \eqref{G_split} we have $\bar {\f v}(\mu) = \delta_{\mu \nu_r} = 0$, so that we only need to consider $\sigma_r \in \Z_2^{n(w_r) + 1}$ instead of $\sigma_r \in \Z_3^{n(w_r) + 1}$. In other words, we always have $d_{\f v} = 0$. As in \eqref{sec8 3}, we find that it suffices to prove, under the same assumptions and for any $k \leq C d_{\f \mu}$,
\begin{multline} \label{avg_step 1}
N^{-3/2} N^{d_{\f \mu} \delta} \sum_{i \in \cal I_M} \sum_{\mu \in \cal I_N} \E \absBB{\frac{1}{N^p} \sum_{\nu_1, \dots, \nu_p \in \cal I_M \setminus \{\mu\}} \pBB{\prod_{r = 1}^{p} A^0(r)} \E_\mu \pBB{\prod_{r = 1}^{p} A^{+}(r)} (X^* G^{(\mu)} X)_{\mu \mu}^k}
\\
=\; O \pB{\pb{N^{\delta}\Psi^2}^p+\E \wt F^p(X)}\,,
\end{multline}
where we abbreviated $A^{*}(r) \equiv A^{*}_{\f e_{\nu_r},i,\mu}(w_r, \sigma_r)$ for $* = 0,+$.

Next, note that each factor $A^+(r)$ is a product of factors $(G^{(\mu)} X)_{s \mu}$ and $(X^* G^{(\mu)})_{\mu s}$ where $s \in \{i,\nu_r\}$. We denote by $d_{\f X, \f i, r}$ the number of factors of $A^+(r)$ for which $s = i$, and abbreviate $d_{\f X, \f i} \deq \sum_{r = 1}^p d_{\f X, \f i, r}$. It is not hard to check that $d_{\f X, \f i} \leq 3$. Recall the definition of $d_{\f X}$ from \eqref{def_d_X}. We claim that
\begin{equation} \label{A+_avg}
\absBB{\E_\mu \pBB{\prod_{r = 1}^{p} A^{+}(r)} (X^* G^{(\mu)} X)_{\mu \mu}^k} \;\prec\; N^{-1/2} \Psi^{d_{\f X} - \ind{d_{\f X, \f i} = 3}}\,,
\end{equation}
which may be regarded as an improvement of Lemma \ref{lem:87}.
The proof of \eqref{A+_avg} is similar to that of Lemma \ref{lem:87}, and we merely give a sketch. We have to estimate an expression of the form
\begin{equation*}
\E_\mu \pBB{\prod_{l = 1}^{d_{\f X}} (G^{(\mu)} X)_{s_l \mu}^{b_i} (X^* G^{(\mu)})_{\mu s_l}^{1 - b_i}} (X^* G^{(\mu)} X)_{\mu \mu}^k\,,
\end{equation*}
where $s_l \in \{i,\nu_1, \dots, \nu_p\}$ and $b_i \in \{0,1\}$. (This is simply the general form of a polynomial in $X$ of the correct degree.) The proof of \eqref{A+_avg} relies on the crucial observations that, by Theorem \ref{thm:gen_iso} (i),
\begin{equation*}
\absb{G^{(\mu)}_{\nu j}} + \absb{G^{(\mu)}_{j \nu}} \;\prec\; \Psi \qquad\txt{and} \qquad  \frac{1}{N} \sum_{j \in \cal I_M} \pb{\absb{G^{(\mu)}_{i j}}^a + \absb{G^{(\mu)}_{ji}}^a} \;\prec\; \Psi^{a \wedge 2}
\end{equation*}
for $a \in \N$. Using that $d_{\f X}$ is odd by Lemma \ref{lem:d odd} (recall that $d_{\f v} = 0$), it is not hard to conclude \eqref{A+_avg}. (Note that, thanks to the a priori information provided by Theorem \ref{thm:gen_iso} (i), the estimate \eqref{A+_avg}, including the exponent on the right-hand side, is sharper than Lemma \ref{lem:87}. This improvement will prove crucial for the conclusion of the argument.)

Next, let $r \geq q + 1$ satisfy $\sigma_r = 0$, so that the left-hand side of \eqref{A+_avg} does not depend on $\nu_r$. There are $p - q - \sum_{s = q+1}^p \abs{\sigma_s}$ such indices $r$, and for each such $r$ we find
\begin{equation} \label{A+_avg_2}
\frac{1}{N} \sum_{\nu_r \in \cal I_N \setminus \{\mu\}} A^0(r) \;=\;  \wt F(X) + O_\prec(\Psi^2)\,.
\end{equation}
Here we wrote $A^0(r) = G_{\nu_r \nu_r} - \Pi_{\nu_r \nu_r} - A_{\f e_{\nu_r}, i,\mu}(w_r, 1)$ and estimated the term $A_{\f e_{\nu_r}, i,\mu}(w_r, 1)$ using \eqref{BGX_est} and \eqref{im_BGB}.
For the remaining $\sum_{s = q+1}^p \abs{\sigma_s}$ indices $r \geq q + 1$ we have $A^0(r) = 1$.
Moreover, for $r \leq q$ it is easy to verify directly using Definition \ref{def:4-words} that
\begin{equation} \label{A+_avg_3}
\absb{A^0(r)} \;\prec\; \Psi^{(2 - d_{\f X, r})_+}\,,
\end{equation}
where we abbreviated $d_{\f X, r} \deq \deg(A^+(r))$. From \eqref{A+_avg}, \eqref{A+_avg_2}, and \eqref{A+_avg_3} we find that the left-hand side of \eqref{avg_step 1} is bounded by
\begin{equation*}
N^{d_{\f \mu} \delta} \, \Psi^{d_{\f X} - \ind{d_{\f X, \f i} = 3}} \, \Psi^{\sum_{r = 1}^q (2 - d_{\f X, r})_+} \, \E (\wt F(X) + \Psi^2)^{p - q - \sum_{r = q+1}^p \abs{\sigma_r}}\,.
\end{equation*}
As in the argument following Lemma \ref{lem:87}, we conclude that the claim holds provided that
\begin{equation} \label{cond_avg_pf}
d_{\f X} - \ind{d_{\f X, \f i} = 3} + \sum_{r = 1}^q (2 - d_{\f X, r})_+ - 2q - 2 \sum_{r = q+1}^p \abs{\sigma_r} \;\geq\; 0\,.
\end{equation}

In order to establish \eqref{cond_avg_pf}, we make the following observations about $d_{\f X, \f i, r}$, which may be checked case by case directly from Definition \ref{def:4-words}. First, if $n(w_r) = 0$ then $d_{\f X, \f i, r} = 0$. Second, if $n(w_r) \in \qq{1,3}$ we have the implication
\begin{equation*}
d_{\f X, r} \leq 2 \quad \Longrightarrow \quad d_{\f X, \f i, r} \leq \ind{n(w_r) \geq 2}\,.
\end{equation*}
We conclude that if $d_{\f X, \f i} = 3$ then there exists an $r \leq q$ such that $d_{\f X, r} \geq 3$. Hence, to establish \eqref{cond_avg_pf} it is enough to establish
\begin{equation*}
d_{\f X} + \sum_{r = 1}^q (2 - d_{\f X, r}) - 2q - 2 \sum_{r = q+1}^p \abs{\sigma_r} \;\geq\; 0\,,
\end{equation*}
which is trivial by the identity
\begin{equation*}
d_{\f X} \;=\; \sum_{r = 1}^p d_{\f X,r} \;=\; \sum_{r = 1}^q d_{\f X,r} + 2 \sum_{r = q + 1}^p \abs{\sigma_r}\,.
\end{equation*}
This concludes the proof.
\end{proof}

This concludes the proof of Theorem \ref{thm:gen_iso}. We conclude this section by drawing consequences from Theorems \ref{thm:gen_iso} and \ref{thm:diag_gen}.

\begin{proof}[Proof of Theorems \ref{thm:edge} {\normalfont(i)}, \ref{thm:bulk} {\normalfont(i)}, and \ref{thm:outside} {\normalfont(i)}]
From Lemma \ref{lem:stab_edge1} we find that \eqref{m_sq_root} and \eqref{mz_reg} hold on $\f D^e_k$. From \eqref{mz_reg}, after a possible decrease of $\tau$, we find that \eqref{1+xm} holds. Moreover, from Lemma \ref{lem:stab_edge2} we find that \eqref{def_m} is strongly stable on $\f D^e_k$ in the sense of Definition \ref{def:strong_stab}. Using \eqref{m_sq_root}, we deduce that \eqref{def_m} is stable on $\f D^e_k$ in the sense of Definition \ref{def:stability}. We may therefore invoke Theorems \ref{thm:gen_iso} and \ref{thm:diag_gen} to conclude Theorem \ref{thm:edge} (i).

Similarly, Theorems \ref{thm:bulk} (i) and \ref{thm:outside} (i) follow from Lemmas \ref{lem:stab_bulk} and \ref{lem:stab_outside} respectively, combined with Theorems \ref{thm:gen_iso} and \ref{thm:diag_gen}
\end{proof}

\begin{proof}[Proof of Theorem \ref{thm:LL_G}]
Let $\tau' > 0$ be chosen so that the conclusions of Theorem \ref{thm:edge} (i) hold for $k = 1, \dots, 2p$. Since $\f D = \bigcup_{k = 1}^{2p} \f D^e_k \cup \bigcup_{k = 1}^p \f D^b_k \cup \f D^o$, we obtain Theorem \ref{thm:LL_G} from Theorems \ref{thm:edge} (i), \ref{thm:bulk} (i), and \ref{thm:outside} (i).
\end{proof}

\begin{proof}[Proof of Corollary \ref{cor:GM_GN} and Theorems \ref{thm:edge} {\normalfont(ii)}, \ref{thm:bulk} {\normalfont(ii)}, and \ref{thm:outside} {\normalfont(ii)}  assuming \eqref{T_diag}]
These follow immediately from Theorem \ref{thm:LL_G} and Theorems \ref{thm:edge} (i), \ref{thm:bulk} (i), and \ref{thm:outside} (i) respectively, combined with \eqref{G_M_ident}, \eqref{G_N_ident}, and \eqref{def_Pi}.
\end{proof}

How to remove the assumption \eqref{T_diag} in the above proof is explained in Section \ref{sec:gen_T}.

\section{Eigenvalue rigidity and edge universality} \label{sec:rigidity}

In the first part of this section we establish eigenvalue rigidity (Theorems \ref{thm:rigidity}, \ref{thm:edge} (iii), and \ref{thm:bulk} (iii)).  As a consequence of Theorem \ref{thm:LL_G} we prove Theorem \ref{thm:G_outside}. In the second part of this section we establish the edge universality from Theorem \ref{thm:edge_univ}. We assume \eqref{T_diag} throughout this section; how to remove it is explained in Section \ref{sec:gen_T} below.

\subsection{Eigenvalue rigidity assuming \eqref{T_diag}}
For simplicity, we first prove Theorem \ref{thm:rigidity}, i.e.\ we assume that all edges and bulk components are regular. The proof consists of three steps. First, we prove that with high probability there are no eigenvalues at a distance greater than $N^{-2/3 + \epsilon}$ from the support of $\varrho$. Second, we prove that a neighbourhood of the $k$-th bulk component of $\varrho$ contains with high probability exactly $N_k$ eigenvalues (recall the definition \eqref{def_N_k}). Third, we use the averaged local law from Theorem \ref{thm:LL_G} together with the first two steps to complete the proof.

We begin with the first step. We define the \emph{distance to the spectral edge} through 
\begin{equation} \label{def_kappa}
\kappa \;\equiv\; \kappa(E) \;\deq\; \dist (E, \partial \supp \varrho)\,.
\end{equation}

\begin{lemma} \label{lem:gap}
Fix $\epsilon \in (0,2/3)$. Under \eqref{T_diag} and the assumptions of Theorem \ref{thm:rigidity} we have
\begin{equation*}
\spec(Q) \cap \hb{E \geq \tau \col \dist(E, \supp \varrho) \geq N^{-2/3 + \epsilon}} \;=\; \emptyset
\end{equation*}
with high probability (recall Definition \ref{def:high_prob}).
\end{lemma}

\begin{proof}
The argument is similar to that of the previous works \cite{EYY3, PY}, and we only explain how to adapt it.
From \cite[Theorem 2.10]{BEKYY}, we find that $\norm{X^* X} \leq C_0$ with high probability for some large enough constant $C_0 > 0$ depending on $\tau$. From \eqref{bound_Sigma} we therefore deduce that $\lambda_1 \leq C$ with high probability for some constant $C > 0$ depending on $\tau$. It therefore suffices to prove that, after a possible decrease of $\tau$,
\begin{equation} \label{no eigenvalues outside}
\spec(Q) \cap \hb{E \in [\tau, \tau^{-1}] \col \dist(E, \supp \varrho) \geq N^{-2/3 + \epsilon}} \;=\; \emptyset
\end{equation}
with high probability. For the remainder of the proof we use a spectral parameter $z = E + \ii \eta$ where $E \in [\tau, \tau^{-1}]$ satisfies $\dist(E, \supp \varrho) \geq N^{-2/3 + \epsilon}$, and $\eta \deq N^{-1/2 - \epsilon/4} \kappa^{1/4}$. We omit $z$ from our notation.

In order to prove \eqref{no eigenvalues outside}, it suffices to prove that $\im m_N \prec N^{-\epsilon/2} (N \eta)^{-1}$. By \eqref{m_sq_root} and Lemmas \ref{lem:stab_edge1}, \ref{lem:stab_bulk}, and \ref{lem:stab_outside}, we have
\begin{equation} \label{im_m_asymp}
\im m \;\asymp\; \eta (\kappa + \eta)^{-1/2}\,.
\end{equation}
We conclude that it suffices to prove
\begin{equation} \label{improved_avg_G}
\abs{m_N - m} \;\prec\; (\kappa + \eta)^{-1/2} \, \Psi^2\,.
\end{equation}
It is easy to check that the control parameter on the right-hand side of \eqref{improved_avg_G} satisfies \eqref{Phi_cond}, so that by Proposition \ref{prop:gen_avg} and Theorem \ref{thm:LL_G} it suffices to prove \eqref{improved_avg_G} for diagonal $\Sigma$. The proof is identical to that of Section \ref{sec:diag_Sigma}, except that we use the strong stability of \eqref{def_m} from Definition \ref{def:strong_stab}, which is established in Lemmas \ref{lem:stab_edge2}, \ref{lem:stab_bulk}, and \ref{lem:stab_outside}. This concludes the proof of \eqref{improved_avg_G}, and hence of \eqref{no eigenvalues outside}.

With applications to the deformed Wigner matrices in Section \ref{sec:Wig} in mind, we give another proof of \eqref{no eigenvalues outside}, which is based on \cite[Section 6]{BS2}. Using a partial fraction decomposition, one easily finds that there exist universal constants $C_{n,k}$ such that
\be\label{tur1}
\frac1N\sum_i \prod_{k=1}^n \frac{1}{(\lambda_i-E)^2+k\eta^2} \;=\; \sum_{k=1}^n C_{n,k} \eta^{-2n+1} \im  m_N(z_k)\,, \qquad z_k \;\deq\; E+i\eta_k\,, \qquad \eta_k \;\deq\; \sqrt{k}\eta \,.
\ee
Similarly, we have (with the same $C_{n,k}$ and $z_k$)
\be\label{tur2}\int \prod_{k=1}^n \frac{1}{(x-E)^2+k\eta^2} \, \varrho(\dd x) \;=\; \sum_{k=1}^n C_{n,k} \eta^{-2n+1} \im  m (z_k)\,.
\ee
Let $E \geq \gamma_+ + N^{-2/3 + \epsilon}$ and set $\eta \deq N^{-\epsilon} \kappa$. It is not hard to deduce from \eqref{improved_avg_G} that, for any fixed $n \in \N$, we have
\begin{equation}
\absb{\im  m_N (z_k)-\im  m (z_k)} \;\prec\;  \frac{N^{-\epsilon/2}}{N \eta}
\end{equation}
for all $k = 1, \dots, n$. Setting $n \deq \ceil{2/\epsilon}$, we get from \eqref{tur1} and \eqref{tur2} that
\begin{align*}
 \frac1N\sum_i \prod_{k=1}^n \frac{1}{(\lambda_i-E)^2+k\eta^2} &\;=\; \int  \prod_{k=1}^n \frac{1}{(x-E)^2+k\eta^2} \, \varrho(\dd x)
 + O_\prec(N^{-1-\epsilon/2}\eta^{-2n})
 \\
 &\;=\; O_\prec(N^{-1-\epsilon/2}\eta^{-2n})\,, 
\end{align*}
where in the last step we used that $\abs{x - E} \geq \kappa$ for $x \in \supp \varrho$. This immediately implies that with high probability there is no eigenvalue in $[E-\eta, E+\eta]$. 
\end{proof}

The second step represents most of the work. It is a counting argument, based on a continuous deformation of the matrix $Q$ to another matrix for which the claim is obvious. Since the eigenvalues depend continuously on the deformation parameter and each intermediate matrix satisfies a gap condition from Lemma \ref{lem:gap}, we shall be able to conclude that the number of eigenvalues in a neighbourhood of the $i$-th component does not change under the deformation. We shall in fact need two deformations: one which deforms the original matrix $Q$ to a Gaussian one, $Q^{\txt{Gauss}}$, with the same expectation $\Sigma$ as $Q$, and another which deforms the Gaussian matrix $Q^{\txt{Gauss}}$ to another Gaussian matrix where some eigenvalues of $\Sigma$ have been increased.

For $k = 1, \dots, p-1$, we introduce the number of eigenvalues to the right of the $k$-th gap,
\begin{equation*}
\Upsilon_k \;\deq\; \sum_{i = 1}^{M \wedge N} \indbb{\lambda_i \geq \frac{a_{2k} + a_{2k+1}}{2}}\,.
\end{equation*}
\begin{proposition} \label{prop:counting}
Under \eqref{T_diag} and the assumptions of Theorem \ref{thm:rigidity} we have
$\Upsilon_k = \sum_{l \leq k} N_l$ with high probability.
\end{proposition}
As explained above, the first step in the proof of Proposition \ref{prop:counting} is a deformation of the general matrix $X$ to a Gaussian one.

\begin{lemma} \label{lem:interp_Gauss}
Let $X$ be general and $X^{\txt{Gauss}}$ a Gaussian matrix of the same dimensions.
Suppose that \eqref{T_diag} and the assumptions of Theorem \ref{thm:rigidity} hold. If $\Upsilon_k = \sum_{l \leq k} N_l$ with high probability under the law of $X^{\txt{Gauss}}$, then $\Upsilon_k = \sum_{l \leq k} N_l$ with high probability under the law of $X$.
\end{lemma}

\begin{proof}
Let $X_1 \deq X$ and $X_0 \deq X^{\txt{Gauss}}$ be independent. For $t \in [0,1]$ define $X(t) \deq \sqrt{t} X_1 + \sqrt{1 - t} X_0$ and denote by $\lambda_i(t)$ the eigenvalues of $X(t) \Sigma X(t)^*$. We also write $\Upsilon_k(t)$ for $\Upsilon_k$ defined in terms of $\lambda_i(t)$. Note that Lemmas \ref{lem:norm_XX} and \ref{lem:gap} hold for $X(t)$ uniformly in $t \in [0,1]$. Recalling \eqref{bound_Sigma}, we deduce that there exists a constant $C > 0$ such that
\begin{equation} \label{cont_lambda_t}
\abs{\lambda_i(t) - \lambda_i(s)} \;\leq\; C \sqrt{\abs{t - s}}
\end{equation}
with high probability, uniformly in $t,s \in [0,1]$ and $i$. The claim now follows by considering $t$ in the lattice $0, 1/K, 2/K, \dots, 1$, where $K$ is chosen large enough that $C K^{-1/2} \leq (a_{2k} - a_{2k+1})/10$ and $C$ is the constant from \eqref{cont_lambda_t}. Indeed, suppose $\Upsilon_k(t) = \sum_{l \leq k} N_l$ with high probability. Then we use \eqref{cont_lambda_t} and the gap from Lemma \ref{lem:gap} to deduce that $\Upsilon_k(t + K^{-1}) = \sum_{l \leq k} N_l$ with high probability. The claim for $t = 1 = K/K$ therefore follows by induction.
\end{proof}

Using Lemma \ref{lem:interp_Gauss}, in order to prove Proposition \ref{prop:counting} it suffices to prove the following result.

\begin{lemma} \label{lem:Gaussian counting}
Suppose that \eqref{T_diag} and the assumptions of Theorem \ref{thm:rigidity} hold, that $X$ is Gaussian, and that $\Sigma$ is diagonal. Then $\Upsilon_k = \sum_{l \leq k} N_l$ with high probability.
\end{lemma}

\begin{proof}
Abbreviate $d_k \deq \sum_{l \leq k} N_l$.
We write the diagonal matrix $\Sigma = \diag(\sigma_1, \sigma_2, \dots, \sigma_M)$ in the block form $\Sigma = \diag(\Sigma_1, \Sigma_2)$, where $\Sigma_1$ contains the $d_k$ top eigenvalues of $\Sigma$. We introduce the deformed covariance matrix $\Sigma(t) \deq \diag (t \Sigma_1, \Sigma_2)$. In particular, $\Sigma(1) = \Sigma$. The idea is to increase $t$ until the claim for $\Sigma$ replaced by $\Sigma(t)$ may be deduced from simple linear algebra. Then we use a continuity argument to compare $\Sigma(t)$ to $\Sigma$.

We add the argument $t$ to quantities to indicate that they are defined in terms of $\Sigma(t)$ instead of $\Sigma$. Using \eqref{Ni_rep} (see also Figure \ref{fig:plot_f}), it is not hard to find that the gap $a_{2k}(t) - a_{2k+1}(t)$ is increasing in $t$, and that
\begin{equation}
a_{2k}(t) \;\geq\; ct \,, \qquad a_{2k+1}(t) \;\leq\; C
\end{equation}
for some constants $c,C > 0$. Using that $\norm{\Sigma(t) - \Sigma(s)} \leq C \abs{t - s}$, we deduce from Lemma \ref{lem:norm_XX} that $\norm{Q(t) - Q(s)} \leq C \abs{t - s}$ with high probability. Now let $T$ be a fixed large time to be chosen later. By a continuity argument using Lemma \ref{lem:gap} we therefore find that there exists a constant $K \equiv K(T)$ such that for any $t \in [1,T]$ we have
\begin{equation} \label{deformation_arg}
\absb{\spec(Q(t)) \cap \pb{a_{2k}(t) - N^{-2/3 + \epsilon}, \infty}} \;=\; \absb{\spec(Q(t + K^{-1})) \cap \pb{a_{2k}(t + K^{-1}) - N^{-2/3 + \epsilon}, \infty}}
\end{equation}
with high probability.

It therefore remains to show that there exists a large enough $T$, depending only on $\tau$, such that for $t = T$ the left-hand side of \eqref{deformation_arg} is equal to $d_k$. Writing $X X^* = \begin{pmatrix}
E_{11} & E_{12}
\\
E_{21} & E_{22}
\end{pmatrix}$ as a block matrix, we find
\begin{equation} \label{interp_block}
Q(t) \;=\; \Sigma(t)^{1/2} X X^* \Sigma(t)^{1/2} \;=\; \begin{pmatrix}
t \Sigma_1^{1/2} E_{11} \Sigma_1^{1/2} & \sqrt{t} \Sigma_1^{1/2} E_{12} \Sigma_2^{1/2}
\\
\sqrt{t} \Sigma_2^{1/2} E_{21} \Sigma_1^{1/2} & \Sigma_2^{1/2} E_{22} \Sigma_2^{1/2}
\end{pmatrix}\,.
\end{equation}
From \eqref{Ni_sum} we deduce that $d_k \leq N - N_p \leq (1 - c) N$, by the assumption of Theorem \ref{thm:rigidity}. From \cite[Theorem 2.10]{BEKYY}, we therefore deduce that $c \leq E_{11} \leq C$ with high probability for some positive constants $c,C$. Moreover, from Lemma \ref{lem:norm_XX} we deduce that $\norm{E_{12}} + \norm{E_{21}}+ \norm{E_{22}} \leq C$ with high probability. We conclude that for large enough $t$ (depending on the constants $c$ and $C$ above) the matrix \eqref{interp_block} has with high probability exactly $d_k$ eigenvalues of order $\asymp t$, and all other eigenvalues are of order $O(\sqrt{t})$. Since $\spec(Q(t)) \cap \qb{a_{2k+1}(t) + N^{-2/3 + \tau}, a_{2k} - N^{-2/3 + \tau}} = \emptyset$ with high probability by Lemma \ref{lem:gap}, we conclude that for large enough $t$ the left-hand side of \eqref{deformation_arg} is equal to $d_k$ with high probability. This concludes the proof.
\end{proof}

Proposition \ref{prop:counting} follows immediately from Lemmas \ref{lem:interp_Gauss} and \ref{lem:Gaussian counting}. This concludes the second step outlined above.

Finally, the third step -- the conclusion of the proof of Theorem \ref{thm:rigidity} -- follows from Theorem \ref{thm:LL_G}, Lemma \ref{lem:gap}, and Proposition \ref{prop:counting} by repeating the analysis of \cite{EYY3, PY} with merely cosmetic changes, as explained in the proof of Theorem \ref{thm:rigidity}. This concludes the proof of Theorem \ref{thm:rigidity} under the assumption \eqref{T_diag}. Moreover, as explained in Section \ref{sec:gen_T}, the assumption \eqref{T_diag} may be easily removed.

Next, we note that the proof Theorem \ref{thm:G_outside} under the assumption \eqref{T_diag} is an easy consequence of Theorems \ref{thm:LL_G} and \ref{thm:rigidity}, following the proof of \cite[Theorem 3.12]{BEKYY}.  The key tool is the spectral decomposition \eqref{decompfullG}. We omit further details.

Finally, the proof of Theorems \ref{thm:edge} (iii) and \ref{thm:bulk} (iii) under the assumption \eqref{T_diag} is exactly the same as that of Theorem \ref{thm:rigidity}. Indeed, one can check that in the proof of Theorem \ref{thm:rigidity} only the weaker assumptions of Theorems \ref{thm:edge} (iii) and \ref{thm:bulk} (iii) are needed. We omit the details.

\subsection{Edge universality assuming \eqref{T_diag}}

Finally, we prove the edge universality from Theorem \ref{thm:edge_univ} under the assumption \eqref{T_diag}.

\begin{proof}[Proof of Theorem \ref{thm:edge_univ} assuming \eqref{T_diag}]
First, we claim that the joint asymptotic distribution of $\f q_{k,l}$ does not depend on the distribution of the entries of $X$, provided they satisfy \eqref{cond on entries of X} and \eqref{moments of X-1}. This is a routine application of the Green function comparison method near the edge, developed in \cite[Section 6]{EYY3}. The argument of \cite[Section 6]{EYY3} may be easily adapted to our case, using the linearizing block matrix $H(z)$ and its inverse $G(z)$ to write $m_N = \frac{1}{N} \sum_{\mu \in \cal I_N} G_{\mu \mu}$. The key technical inputs are the anisotropic local law from Theorem \ref{thm:LL_G}, the eigenvalue rigidity from Theorem \ref{thm:edge} (iii), and the resolvent identities from Lemma \ref{lem:idG}. We omit further details.

We may therefore without loss of generality assume that $X$ is Gaussian. By orthogonal / unitary invariance of the law of $X$, we may furthermore assume that $\Sigma$ is diagonal. This concludes the proof.
\end{proof}

\section{General matrices: removing \eqref{T_diag} and extension to $\dot Q$}  \label{sec:generalizations}

In this section we explain how our results, proved under the assumption \eqref{T_diag} and for the matrix $Q$, may be generalized to hold without the assumption \eqref{T_diag} and for the matrix $\dot Q$ from \eqref{def_Q} as well.

\subsection{How to remove \eqref{T_diag}} \label{sec:gen_T}
For simplicity, throughout the proofs up to now we made the assumption \eqref{T_diag}. As advertised, this assumption is not necessary. In this section we explain how to dispense with it. The argument relies on simple approximation and linear algebra. Roughly, if $\Sigma$ has a zero eigenvalue, we consider $\Sigma + \epsilon$ instead and let $\epsilon \downarrow 0$; if $T$ is not square, we augment it to a square matrix by padding it out with zeros. While this extension is simple, we emphasize that it relies crucially on the fact we do not assume that $\Sigma$ has a lower bound (the assumption \eqref{T_diag} only requires the qualitative bound $\Sigma > 0$).

We distinguish the cases $\wh M \geq M$ and $\wh M < M$. Suppose first that $\wh M \geq M$. We extend $T$ to an $\wh M \times \wh M$ matrix by setting
$\wh T \deq \binom{0}{T}
$.
Define the $\wh M \times \wh M$ matrices
\begin{equation*}
\wh \Sigma \;\deq\; \wh T \wh T^* \;=\; \begin{pmatrix}
0 & 0 \\ 0 & \Sigma
\end{pmatrix}\,,
\qquad \wh S \;\deq\; \wh T^* \wh T \;=\; T^* T\,.
\end{equation*}
By polar decomposition, we have
$\wh T = \wh U \wh S^{1/2}$,
where $\wh U$ is orthogonal. Therefore
$\wh \Sigma = \wh U \wh S \wh U^*$.
Moreover, from \eqref{def_m} we get $m \equiv m_{\Sigma, N} = m_{\wh \Sigma, N} = m_{\wh S, N}$, which we use tacitly in the following.

We define the $(\wh M + N) \times (\wh M + N)$ matrix
\begin{equation*}
\wh G \;\deq\; \lim_{\epsilon \downarrow 0} \begin{pmatrix}
\wh U & 0 \\ 0 & 1
\end{pmatrix}
\begin{pmatrix}
-(\wh S + \epsilon)^{-1} & X \\ X^* & -z
\end{pmatrix}^{-1}
\begin{pmatrix}
\wh U^* & 0 \\ 0 & 1
\end{pmatrix}
 \;=\;
\begin{pmatrix}
z \wh \Sigma^{1/2} \pb{\wh T X X^* \wh T^* - z}^{-1} \wh \Sigma^{1/2} & \wh T X R_N
\\
R_N X^* \wh T^* & R_N
\end{pmatrix}\,,
\end{equation*}
where we used \eqref{G=G}. Note that $\wh G$ has the block form
\begin{equation} \label{def_block_G}
\wh G \;=\;
\begin{pmatrix}
0 & 0
\\ 0 & G
\end{pmatrix}\,, \qquad
G \;\deq\; \begin{pmatrix}
z \Sigma^{1/2} R_M \Sigma^{1/2} & T X R_N
\\
R_N X^* T^* & R_N
\end{pmatrix}\,.
\end{equation}

Using this representation of $G$ as a block of $\wh G$, it is easy to drop the assumption \eqref{T_diag}. For example, 
suppose that Theorem \ref{thm:LL_G} has been proved under the assumption \eqref{T_diag}. Applying it to $\wh G$ and using a simple approximation argument in $\epsilon$, we find the following generalization of Theorem \ref{thm:LL_G}.

\begin{proposition}
Fix $\tau > 0$. Suppose that \eqref{T_diag} and Assumption \ref{ass:main} hold. Suppose moreover that every edge $k = 1, \dots, 2p$ satisfying $a_k \geq \tau$ and every bulk component $k = 1, \dots, p$ is regular in the sense of Definition \ref{def:regular}. Then \eqref{edge_G_result} and \eqref{edge_G_avg_result} hold for $G$ defined in \eqref{def_block_G} and $R_N$ defined in \eqref{def_RM_RN}.
\end{proposition}

In particular, Corollary \ref{cor:GM_GN} and Theorems \ref{thm:rigidity}, \ref{thm:edge}, \ref{thm:bulk}, \ref{thm:outside}, and \ref{thm:edge_univ} follow easily from their counterparts proved under the assumption \eqref{T_diag}. This concludes the discussion for the case $\wh M \geq M$.

Finally, we consider the case $\wh M < M$. We set $\wt T \deq (T, 0)$ and $\wt X \deq \binom{X}{Y}$, 
where $Y$ is an $(M - \wh M) \times N$ matrix, independent of $X$, with independent entries satisfying \eqref{cond on entries of X} and \eqref{moments of X-1}. Hence, $\wt T$ is $M \times M$ and $\wt X$ is $M \times N$. Now we have $TX = \wt T \wt X$, and we have reduced the problem to the case $\wh M = M$, which was dealt with above.

\subsection{Extension to $\dot Q$} \label{sec:Q_dot}
In this section we explain how our results have to be modified for $\dot Q$ from \eqref{def_Q}. For simplicity of presentation, we make the assumption \eqref{T_diag}; it may be easily removed as explained in Section \ref{sec:gen_T}.

The matrix $\dot Q$ is obtained from $Q$ by replacing $X$ with $\dot X \deq X (1 - \f e \f e^*)$. Generally, a dot on any quantity depending on $X$ means that $X$ has been replaced by $\dot X$ in its definition. For example, we have
\begin{equation*}
\dot G \;=\; \begin{pmatrix}
-\Sigma^{-1} & \dot X
\\
\dot X^* & -z
\end{pmatrix}^{-1}\,.
\end{equation*}
As before it is easy to obtain $\dot R_M$ and $\dot R_N$ from $\dot G$.

Moreover, in analogy to \eqref{def_Pi} we define
$
\dot \Pi \deq \Pi -  (m + z^{-1}) \f e \f e^*$.
(Recall the convention that $\f e \in \R^{\cal I}$ is the natural embedding of $\f e \in \cal R^{\cal I_N}$ obtained by adding zeros.)
Let $\f S \subset \f D$ be a spectral domain.
As in Definition \ref{def:local_laws}, we say that the anisotropic local law holds for $\dot G$ if
\begin{equation} \label{isotropic law dot}
\ul \Sigma^{-1} \pb{\dot G(z) - \dot \Pi(z)} \ul \Sigma^{-1} \;=\; O_\prec(\Psi(z))
\end{equation}
uniformly in $z \in \f S$, and that averaged local law for $\dot G$ holds if
\begin{equation} \label{averaged law dot}
\dot m_N(z) - m(z) \;=\; O_\prec \pbb{\frac{1}{N \eta}}
\end{equation}
uniformly in $z \in \f S$.

The local law for $\dot G$ reads as follows.
\begin{theorem}[Local laws for $\dot G$] \label{thm:local_laws_dot}
Fix $\tau > 0$. Suppose that $X$ and $\Sigma$ satisfy \eqref{T_diag} and Assumption \ref{ass:main}. Let $\f S \subset \f D$ be a spectral domain. Then the entrywise and averaged local laws hold for $\dot G$ in the sense of \eqref{isotropic law dot} and \eqref{averaged law dot} provided they hold for $G$ in the sense of Definition \ref{def:local_laws} (ii) and (iii).
\end{theorem}

\begin{proof}
To simplify notation, we omit the factor $\frac{N}{N-1}$ in the definition of $\dot Q$. It may easily be put back by scaling the argument $z$.
We prove the anisotropic local law by estimating the four blocks of $\dot G$ individually.
Using simple linear algebra and \eqref{defGM} and \eqref{G=G}, we get
\begin{equation} \label{G_M dot id}
\dot G_M \;=\; G_M + \frac{G_M X \f e \f e^* X^* G_M}{z - \f e^* X^* G_M X \f e} \;=\; G_M + \frac{G_M X \f e \f e^* X^* G_M}{z^2 \f e^* G_N \f e}\,.
\end{equation}
Since $G$ satisfies the anisotropic local law by assumption, it is easy to deduce from Lemma \ref{lem:idG} and \eqref{m_sim_1} that
\begin{equation*}
\scalar{\f v}{\dot G_{M} \f w} \;=\; \scalar{\f v}{G_{M} \f w} + O_\prec(\Psi^2 \abs{\Sigma \f v} \abs{\Sigma \f w})
\end{equation*}
for $\f v , \f w \in \R^{\cal I_M}$.

Next, define the orthogonal projection $\pi \deq \f e \f e^*$ in the space $\R^{\cal I_N}$, as well as its orthogonal complement $\ol \pi \deq I_N - \pi$. Now the upper-right block of $\dot G$ is equal to
\begin{equation*}
z^{-1} \dot G_M \dot X \;=\; z^{-1} G_M X \ol \pi + \frac{G_M X \f e \f e^* X^* G_M X \ol \pi}{z^2 \f e^* G_N \f e}\,.
\end{equation*}
From \eqref{G=G} we get $X^* G_M X = z^2 G_N + z$, so that, using the anisotropic local law for $G$, we conclude
\begin{equation*}
\scalar{\f v}{z^{-1} \dot G_M \dot X \f w} \;=\; O_\prec(\Psi \abs{\Sigma \f v} \abs{\f w})
\end{equation*}
for $\f v \in \R^{\cal I_M}$ and $\f w \in \R^{\cal I_N}$. The lower-left block is dealt with analogously.

Finally, using \eqref{G=G} for $\dot G$ as well as \eqref{G_M dot id}, we get
\begin{equation*}
\dot G_N \;=\; z^{-2} \dot X^* \dot G_M \dot X - z^{-1} \;=\; \ol \pi G_N \ol \pi - \pi z^{-1} +  \frac{\ol \pi G_N \pi G_N \ol \pi}{\f e^* G_N \f e}\,.
\end{equation*}
 Using the anisotropic law for $G$ and Lemma \ref{lem:idG}, we therefore get \nc
\begin{equation} \label{G_N_dot_est}
\scalar{\f v}{\dot G_N \f w} \;=\; \scalar{\f v}{\pb{\ol \pi G_N \ol \pi - \pi z^{-1}} \f w}  + O_\prec(\Psi^2 \abs{\f v}\abs{\f w})
\end{equation}
for $\f v , \f w \in \R^{\cal I_N}$. This concludes the proof of the anisotropic local law.

Moreover, the averaged local law for $\dot G$ is easy to deduce from \eqref{G_N_dot_est} by setting $\f v = \f w = \f e_\mu$ and summing over $\mu \in \cal I_N$. In this way, the averaged local law for $\dot G$ follows from the averaged and anisotropic local laws for $G$.
\end{proof}

We also obtain eigenvalue rigidity for the eigenvalues $\dot \lambda_1 \geq \dot \lambda_2 \geq \cdots \geq \dot \lambda_M$ of $\dot Q$. For instance, Theorem \ref{thm:rigidity} has the following counterpart.

\begin{theorem}[Eigenvalue rigidity for $\dot Q$] \label{thm:rigidity_dot}
Theorem \ref{thm:rigidity} remains valid if $\lambda_k$ is replaced with $\dot \lambda_k$.
\end{theorem}

\begin{proof}
The proof follows that of Theorem \ref{thm:rigidity} to the letter, using Theorem \ref{thm:local_laws_dot} as input. In fact, as explained around \eqref{improved_avg_G}, we need a stronger bound than \eqref{averaged law dot} outside of the spectrum; this stronger bound follows easily from \eqref{G_N_dot_est} and the analogous stronger bound for $G_N$ established in \eqref{improved_avg_G}.
\end{proof}

Finally, we obtain edge universality for $\dot Q$. The following result is proved exactly like Theorem \ref{thm:edge_univ}, using Theorems \ref{thm:local_laws_dot} and \ref{thm:rigidity_dot} as input.

\begin{theorem}[Edge universality for $\dot Q$]
Theorem \ref{thm:edge_univ} remains valid if $\lambda_i$ is replaced with $\dot \lambda_i$.
\end{theorem}

\section{Deformed Wigner matrices} \label{sec:Wig} 

In this section we apply our method to deformed Wigner matrices as a further illustration of its applicability. Since the statements and arguments are similar to those of the previous sections, we keep the presentation concise.

\subsection{Model and results}
Let $W = W^*$ be an $N \times N$ Wigner matrix whose upper-triangular entries $(W_{ij} \col 1 \leq i \leq j \leq N)$ are independent and satisfy the same conditions \eqref{cond on entries of X} and \eqref{moments of X-1} as $X_{i\mu}$. Let 
$A = A^*$ be a deterministic $N \times N$ matrix satisfying $\norm{A} \leq \tau^{-1}$. For definiteness, we suppose that $W$ and $A$ are real symmetric matrices, remarking that similar results also hold for complex Hermitian matrices.

The main result of this section is the anisotropic local law for the \emph{deformed Wigner matrix} $W + A$, analogous to Theorem \ref{thm:gen_iso}. As an application, we establish the edge universality of $W + A$. We remark that the entrywise local law and edge universality were previously established in \cite{LSY} under the assumption that $A$ is diagonal. Before that, edge universality was established in \cite{CP1,Joh2,Shc} under the assumption that $W$ is a GUE matrix. A somewhat different direction was pursued in \cite{BES15,BG15,Kar15}, where local laws for the sum of two matrices that are invariant under unitary or orthogonal conjugations were analysed.

In order to avoid confusion with similar quantities defined previously for sample covariance matrices, we use the superscript $W$ to distinguish quantities defined in terms of the deformed Wigner matrix $W+A$.
The Stieltjes transform $m^W$ of the asymptotic eigenvalue density of $W + A$ is defined as the unique solution of the equation
\begin{equation*}
m^W(z) \;=\; \frac { 1}N\tr\pb{-m^W(z) + A-z}^{-1}
\end{equation*}
satisfying $\im m^W(z) > 0$ for $\im z > 0$. (See e.g.\ \cite{Past} for details.) Note that $m^W$ only depends on the spectrum of $A$ and not on its eigenvectors.
For simplicity, following \cite{LSY} we assume that support of the asymptotic eigenvalue density $\varrho^W(E) \deq \lim_{\eta \downarrow 0} \pi^{-1} \im m^W(E + \ii \eta)$ is an interval, which we denote by $[L_-, L_+]$. This condition is however not necessary for our method, which may in particular easily be extended to the multi-cut case, using an argument similar to the one developed in the context of sample covariance matrices in Section \ref{sec:rigidity} and Appendix \ref{sec:varrho}.

We denote the eigenvalues of $W+A$ by $\lambda_1(W+A) \geq \lambda_2(W+A) \geq \cdots \geq \lambda_N(W+A)$.
Moreover, we define the resolvent $G^W(z) \deq (W+A - z)^{-1}$, as well as
\begin{equation*}
\Pi^W \;\deq\; \frac{1}{-m^W+A-z} \,, \qquad \Psi^W \;\deq\; \sqrt{\frac{\im m^W}{N\eta}}+\frac1{N\eta}\,.
\end{equation*}
The following definition is the analogue of Definition \ref{def:local_laws} for deformed Wigner matrices.

\begin{definition}[Local laws] \label{def:local_laws of Wigner} Define  $\f D^W \deq \hb{z \col \abs{E} \leq \tau^{-1} , N^{-1+\tau} \leq \im z \leq \tau^{-1}}$, and let $\f S \subset \f D^W$ be a spectral domain, i.e.\ for each $z \in \f S$ we have $\h{w \in \f D^W \col \re w = \re z, \im w \geq \im z} \subset \f S$.
\begin{enumerate}
\item\label{en}
We say that the \emph{entrywise local law holds with parameters $(W, A, \f S)$}  if
\begin{equation*}
\pb{G^W(z) - \Pi^W(z)}_{st} \;=\; O_\prec\pb{\Psi^W(z)}
\end{equation*}
uniformly in $z \in \f S$ and $1\leq s,t \leq N$.
\item
We say that the \emph{anisotropic local law holds with parameters $(W, A, \f S)$} if
\begin{equation*}
G^W(z) - \Pi^W(z) \;=\; O_\prec \pb{\Psi^W(z)}
\end{equation*}
uniformly in $z \in \f S$.
\item
We say that the \emph{averaged local law holds with parameters $(W, A, \f S)$} if
\begin{equation*}
\frac{1}{N} \tr G^W (z) - m^W(z) \;=\; O_\prec\pbb{\frac{1}{N \eta}}
\end{equation*}
uniformly in $z \in \f S$.
\end{enumerate}
\end{definition}

In analogy to \eqref{1+xm}, we always assume that
$\abs{-m^W(z) + a_i - z} \geq \tau$
for all $z \in \f S$ and $a_i \in \spec(A)$;
this assumption has been verified under general hypotheses on the spectrum of $A$ in \cite{LSY}. In particular, it implies that $\norm{\Pi} \leq \tau^{-1}$.

The following result is the analogue of Theorem \ref{thm:gen_iso}. Throughout the following we denote by $W^{\txt{Gauss}}$ a GOE matrix and by $D \equiv D_A$ the diagonalization of $A$.
\begin{theorem}[General local laws]\label{lem:general Wig}
Let $W$ and $A$ be as above. Fix $\tau > 0$ and let $\f S \subset \f D^W(\tau, N)$ be a spectral domain. Suppose that either (a) $\E W_{ij}^3=0$ for all $i,j$ or (b) there exists a constant $c_0 > 0$ such that $\Psi^W(z) \leq N^{-1/4-c_0}$ for all $z \in \f S$.
\begin{enumerate}
\item
If the entrywise local law holds with parameters $(W^{\txt{Gauss}}, D, \f S)$, then the anisotropic local law holds with parameters $(W, A, \f S)$.
\item
If the entrywise local law and the averaged local law hold with parameters $(W^{\txt{Gauss}}, D, \f S)$, then the averaged local law holds with parameters $(W, A, \f S) $.
\end{enumerate}
\end{theorem}

Note that the assumptions of Theorem \ref{lem:general Wig} depend on both the deterministic matrix $A$ and the spectral domain $\f S$.

The following result guarantees the rigidity of the extreme eigenvalues of $H+A$. As explained below, the rigidity of all eigenvalues will be a simple consequence of Theorems \ref{lem:general Wig} (ii) and \ref{thm:wigner_rig}.

\begin{definition}
We say the \emph{extreme eigenvalues of $W+A$ are rigid} if
$\qb{\lambda_1(W+A)-L_+}_+ \prec N^{-2/3}$ and $\qb{\lambda_N(W+A)-L_-}_- \prec N^{-2/3}$.
\end{definition}
\begin{theorem} \label{thm:wigner_rig}
Suppose that, for any fixed and small enough $\tau > 0$, the entrywise local law and the averaged local law hold with parameters $(W^{\txt{Gauss}}, D, \f D^W(\tau,N) )$. Moreover, suppose that the extreme eigenvalues of $W^{\txt{Gauss}} + D$ are rigid. Then the extreme eigenvalues of $W+A$ are rigid.
\end{theorem}

Together, Theorems \ref{lem:general Wig} (ii) and \ref{thm:wigner_rig} easily yield the rigidity of the eigenvalues, as explained in the proof of \cite[Theorem 2.2]{EYY3}. We may apply Theorem \ref{lem:general Wig} to extend the results of \cite{LSY} to arbitrary non-diagonal $A$. For instance, we obtain the following edge universality result.

\begin{theorem}[Edge universality] \label{lem:comp Wig}
Let $W$ and $A$ be as above. Suppose that the spectrum $D$ of the matrix $A$ satisfies the assumptions of Theorem \ref{thm:wigner_rig}. Then for any fixed $l$ and smooth $f \col \R^l \to \R$ there is a constant $c_0 > 0$ such that
\begin{equation} \label{Wigner_edge_univ}
f\qB{\pb{N^{2/3}(\lambda_i(W+A)-L_+)}_{i=1}^l} - f\qB{\pb{N^{2/3}(\lambda_i(W^{\txt{Gauss}} +D)-L_+)}_{i=1}^l} \;=\; O(N^{-c_0})\,.
\end{equation}
A similar result holds for the extreme eigenvalues near the left edge.
\end{theorem}

In \cite{LSY}, the assumptions of Theorem \ref{thm:wigner_rig} were verified for a large class of diagonal matrices $D$. Moreover, for such matrices $D$, it was proved that the limit of the second term on the left-hand side of \eqref{Wigner_edge_univ} is governed by the Tracy-Widom-Airy statistics. Theorem \ref{lem:comp Wig} therefore provides an extension of \cite[Theorem 2.8]{LSY} to non-diagonal matrices $A$. We refer to \cite{LSY} for the detailed statements about the distribution of the eigenvalues of $W^{\txt{Gauss}} +D$.

Further applications of Theorem \ref{lem:general Wig} include a study of the eigenvectors of $W+A$ and the outliers of finite-rank perturbations of $W+A$. We do not pursue these questions here.

The rest of this section is devoted to the proof of Theorems \ref{lem:general Wig}, \ref{thm:wigner_rig}, and \ref{lem:comp Wig}. To unburden notation, from now on we omit the superscripts $W$, with the understanding that all quantities in this section are defined in terms of deformed Wigner matrices.

\subsection{Proof of Theorem \ref{lem:general Wig} (i)} \label{sec:Wigner_ll}

Exactly as in Section \ref{sec:iso_gauss}, we first prove that the entrywise local law with parameters $(W^{\txt{Gauss}}, D, \f S)$ implies the anisotropic local law with parameters $(W^{\txt{Gauss}}, A, \f S)$. The proof is very similar to the one from Section \ref{sec:iso_gauss}. As explained there, the details may be taken over with trivial modifications from \cite[Sections 5 and 7]{BEKYY}, where the proof is given for $A=0$.

What remains, therefore, is the proof of the anisotropic local law with parameters $(W, A, \f S)$, starting from the anisotropic local law with parameters $(W^{\txt{Gauss}}, A, \f S)$.
The basic strategy is similar to the one in section \ref{sec:comparison1}.  Define $\wh {\f S}$ and $\wh {\f S}_m$ as in Section \ref{sec:comparison1}, and denote by $\bb S \subset \R^N$ the unit sphere of $\R^N$. The following definitions are analogous to \eqref{bound_assumption} and \eqref{bound_goal}.
\begin{itemize}
\item[\textbf{($\f A_m$)}]
For all $z \in \wh{\f S}_m$ and all deterministic $\bv \in \bb S$ we have
\begin{equation} \label{bound_assumption_W}
\im G_{\f v \f v}(z) \;\prec\;  \im m(z) + N^{C_0 \delta} \Psi(z)\,.
\end{equation}

\item[\textbf{($\f C_m$)}]
For all $z \in \wh{\f S}_m$ and all deterministic $\bv \in \bb S$ we have
\begin{equation} \label{bound_goal_W}
\absb{G_{\f v \f v}(z) - \Pi_{\f v \f v}(z)} \;\prec\; N^{C_0 \delta} \Psi(z)\,.
\end{equation}
\end{itemize}
Here $C_0$ is a constant that may depend only on $\tau$, and $\delta$ satisfies \eqref{cond_delta}.

Clearly, \textbf{($\f A_0$)} holds, and \textbf{($\f C_m$)} implies \textbf{($\f A_m$)}. As in Lemma \ref{lem:induction}, we only need to prove that $(\f A_{m-1}) $ implies  \textbf{($\f C_m$)}. The necessary a priori bound is summarized in the following result, which may be proved like Lemma \ref{cor:rough_bound}.

\begin{lemma}
If \textbf{($\f A_{m-1}$)} holds then
\begin{equation} \label{a_priori_wig}
G(z) \;=\; O_\prec(N^{2 \delta})\,, \qquad \frac{\im G_{\f v \f v}(z)}{N \eta} \;\prec\; N^{(C_0 + 1) \delta} \Psi^2
\end{equation}
for all $z \in \wh {\f S}_m$ and deterministic $\f v \in \bb S$.
\end{lemma}

Making minor adjustments to the argument of Sections \ref{sec:interpolation} and \ref{sec:expansion}, we find that it suffices to prove the following result, which is analogous to Lemma \ref{sc_step 3}.

\begin{lemma} \label{lem:Wig_main_est}
For $\f v \in \bb S$ define
\begin{equation*}
F_{\f v}(W,z) \;\deq\; \absb{G_{\f v \f v}(z) - \Pi_{\f v \f v}(z)}\,.
\end{equation*}
Let $z \in \wh{\f S}$ and suppose that \eqref{a_priori_wig} holds. Then we have
\begin{equation}\label{lowz}
N^{-n/2}   \sum_{i\leq j} \E \absbb{ \pbb{\frac{\partial}{\partial W_{ij}}}^n F^p_{\f v}(W,z)} \;=\; O \pB{(N^{C_0 \delta}\Psi(z))^p+  \E F_{\f v}^p(W,z)}
\end{equation}
for any $n = 4, 5, \dots, 4p$. Moreover, if in addition $\Psi(z) \leq N^{-1/4-c_0}$ then \eqref{lowz} holds also for $n = 3$.
\end{lemma}

Here we also used that the function $\eta \mapsto \Psi(E + \ii \eta)$ is decreasing, so that if the assumption $\Psi \leq N^{-1/4-c_0}$ holds at $z \in \f S$ then it also holds for all $w \in \f S$ satisfying $\re w = \re z$ and $\im w \geq \im z$.

The rest of this subsection is devoted to the proof of Lemma \ref{lem:Wig_main_est}. We note that the word structure describing the derivatives on the left-hand side of \eqref{lowz} is very similar to that of \eqref{sc_step 3} (given in Definition \ref{def:words}). Hence the proof of Lemma \ref{lem:Wig_main_est} is similar to that of Lemma \ref{lem:sc_step 3}.

The proof of \eqref{lowz} for $n \geq 4$ is a trivial modification of the argument given in Section \ref{sec:words}, whose details we omit. What remains therefore is the proof of \eqref{lowz} for $n = 3$ under the assumption $\Psi \leq N^{-1/4-c_0}$. From now on we omit the arguments $z$.

The main new ingredient of the proof for deformed Wigner matrices is a further iteration step at fixed $z$. Suppose that
\begin{equation} \label{oldPsi}
G - \Pi \;=\; O_\prec(N^{2 \delta} \Phi)
\end{equation}
for some $\Phi \leq 1$. Since $\norm{\Pi} \leq C$, the estimate \eqref{oldPsi} is stronger than the first estimate of \eqref{a_priori_wig}. Note that, by assumption \eqref{a_priori_wig}, the estimate \eqref{oldPsi} holds for $\Phi = 1$. Assuming \eqref{oldPsi}, we shall prove a self-improving bound of the form
\begin{equation} \label{Wig_self_improv}
N^{-3/2}   \sum_{i\leq j} \E \absbb{\pbb{\frac{\partial}{\partial W_{ij}}}^3 F^p_{\f v}(W)} \;=\; O \pB{(N^{C_0 \delta}\Psi)^p+(N^{-c_0/2} \Phi)^p+\E F_{\f v}^p(W,z)}\,.
\end{equation}
Once \eqref{Wig_self_improv} is proved, we may use it iteratively to obtain increasingly accurate bounds for the left-hand side of \eqref{bound_goal_W}. After each step, we obtain an improved a priori bound \eqref{oldPsi}, whereby $\Phi$ is reduced by powers of $N^{-c_0/2}$. After $O(1/c_0)$ iterations, \eqref{lowz} for $n = 3$ follows.

It therefore suffices to prove \eqref{Wig_self_improv} under the assumptions \eqref{a_priori_wig} and \eqref{oldPsi}. As in Section \ref{sec:words}, it suffices to prove
\begin{equation} \label{main Wig} 
N^{-3/2}\absBB{\sum_{i \leq j} A_{\f v,i, j}(w_0)^{p - q} \prod_{r = 1}^{q} A_{\f v,i,j }(w_r)} 
\;\prec\; F_{\f v}^{p-q}(W) \pb{N^{(C_0 -1) \delta} \Psi +N^{-c_0} \Phi}^q\,,
\end{equation}
where the words $w$ and their values $A_{\f v,i,j}(\,\cdot\,)$ are defined similarly to Definition \ref{def:words}. More precisely, abbreviate $\wt G \deq G - \Pi$. Then for $n(w) = 0$ we have $A_{\f v,i, j}(w) = \wt G_{\f v \f v}$, and for $n(w) \geq 1$ the random variable $A_{\f v,i,j }(w)$ is a product of entries $G_{\f v i}$, $G_{\f v j}$, $G_{i j}$, $G_{ji}$, $G_{i \f v}$, and $G_{j \f v}$, as in \eqref{def_w_1}. Clearly, to prove \eqref{main Wig} it suffices to prove
\begin{equation} \label{main Wig 2} 
N^{-3/2}\absBB{\sum_{i \leq j} \prod_{r = 1}^{q} A_{\f v,i,j }(w_r)} 
\;\prec\; \pb{N^{(C_0 -1) \delta} \Psi +N^{-c_0} \Phi}^q\,.
\end{equation}
We discuss the three cases $q = 1,2,3$ separately.

\subsubsection*{The case $q = 1$}
The single factor $A_{\f v,i,j }(w_r)$ is of the form
\begin{equation} \label{Wig_A_form}
G_{\f v i} G_{a b} G_{cd} G_{j \f v} \qquad \txt{or} \qquad G_{\f v i} G_{a b} G_{cd} G_{i \f v}
\end{equation}
or a term obtained from one of these two by exchanging $i$ and $j$;
here $a,b,c,d \in \{i,j\}$. We deal first with the first expression; the others are deal with analogously. We split it according to
\begin{equation} \label{q1_wig_split}
G_{\f v i} G_{a b} G_{cd} G_{j \f v} \;=\; G_{\f v i} \Pi_{a b} \Pi_{cd} G_{j \f v}
+ G_{\f v i} \wt G_{a b} \Pi_{cd} G_{j \f v}
+ G_{\f v i} \Pi_{a b} \wt G_{cd} G_{j \f v}
+ G_{\f v i} \wt G_{a b} \wt G_{cd} G_{j \f v}\,.
\end{equation}
We deal with the summation over $i$ using the estimate
\begin{equation} \label{shale}
\frac{1}{N} \sum_{i} \abs{G_{\f v i}}^2 \;=\; \frac{\im G_{\f v \f v}}{N \eta} \;\prec\; N^{(C_0 + 1) \delta} \Psi^2\,,
\end{equation}
where in the last step we used \eqref{a_priori_wig}; a similar estimate holds for the summation over $j$. Using \eqref{shale}, \eqref{oldPsi}, and $\norm{\Pi} \leq \tau^{-1}$, we may estimate the second term on the right-hand side of \eqref{q1_wig_split} as
\begin{equation*}
N^{-3/2} \sum_{i,j} \absb{G_{\f v i} \wt G_{a b} \Pi_{cd} G_{j \f v}} \;\prec\; N^{-3/2} \sum_{i,j} N^{2 \delta} \Phi \absb{G_{\f v i} G_{j \f v}} \;\prec\; N^{1/2} N^{(C_0 + 3) \delta} \Psi^2 \Phi \;\leq\; N^{-c_0} \Phi\,,
\end{equation*}
provided $\delta$ is chosen small enough, depending on $\tau$ and $c_0$. The third and fourth terms on the right-hand side of \eqref{q1_wig_split} are estimated in exactly the same way.

What remains is the first term on the right-hand side of \eqref{q1_wig_split}. Here taking the absolute value inside the sum is not affordable. Instead, we use the first A priori bound of \eqref{a_priori_wig} to estimate
$
\absb{\sum_{i : i \leq j} G_{si} \Pi_{ab} \Pi_{cd}} \prec N^{1/2 + 2 \delta}\,,
$
where we used that $\Pi$ is deterministic and satisfies the bound $\sum_{i} \abs{\Pi_{ab} \Pi_{cd}} \leq \tau^{-2} N$. Combining this estimate with $\sum_j \abs{G_{j \f v}} \prec N^{1 + (C_0/2 + 1) \delta} \Psi$ from \eqref{shale}, we find
\begin{equation*}
N^{-3/2} \absbb{\sum_{i \leq j} G_{\f v i} \Pi_{a b} \Pi_{cd} G_{j \f v}} \;\prec\; N^{(C_0 / 2 + 3) \delta} \Psi \;\leq\; N^{(C_0 - 1) \delta} \Psi\,,
\end{equation*}
provided $C_0 \geq 8$.

Finally, if $A_{\f v,i,j }(w_r)$ is the second expression of \eqref{Wig_A_form}, the argument is analogous. In this case at least one of the terms $G_{ab}$ and $G_{cd}$ is of the form $G_{ij}$ or $G_{ji}$, so that, in the analogue of the first term on the right-hand side of \eqref{q1_wig_split}, we may use the improved estimate $\sum_{i} \abs{\Pi_{ab} \Pi_{cd}} \leq \tau^{-1} \sum_i \abs{\Pi_{ij}}\leq \tau^{-2} N^{1/2}$, as follows from $\norm{\Pi} \leq \tau^{-1}$.

\subsubsection*{The case $q = 2$}
In this case the product $\prod_{r = 1}^{q} A_{\f v,i,j }(w_r)$ is of the form
\begin{equation} \label{Wig_A_form_2}
G_{\f v i} (GG_{j \f v}^*)_{jt} G_{\f v i}  G_{ji}G_{j \f v} \qquad \txt{or} \qquad  G_{\f v i}G_{j \f v} G_{\f v i}  G_{jj} G_{i \f v}\,,
\end{equation}
or an expression obtained from one of these two by exchanging $i$ and $j$. The contribution of the first expression of \eqref{Wig_A_form_2} is estimated using \eqref{a_priori_wig} and \eqref{shale} by
\begin{equation*}
N^{-3/2} \sum_{i,j} \absb{G_{\f v i}G_{j \f v} G_{\f v i}  G_{ji}G_{j \f v}} \;\prec\; N^{1/2 + (2 C_0 + 4) \delta} \Psi^4 \;\leq\; N^{(2 C_0 + 4) \delta - 2 c_0} \Psi^2 \;\leq\; \Psi^2\,,
\end{equation*}
provided $\delta$ is chosen small enough, depending on $\tau$ and $c_0$.

Next, in order to estimate the contribution of the second expression of \eqref{Wig_A_form_2}, we split
\begin{equation} \label{wq2a}
\sum_jG_{j \f v} G_{jj} \;=\; \sum_jG_{j \f v}\Pi_{jj} + \sum_jG_{j \f v}\wt G_{jj}
 \;=\; O_\prec(N^{1/2+2\delta})+O_\prec \pb{N^{1+(C_0 / 2 + 1) \delta} \Psi \Phi}\,,
\end{equation}
where we used \eqref{a_priori_wig}, \eqref{oldPsi}, and \eqref{shale}. Similarly, using \eqref{a_priori_wig} and \eqref{shale} we get
\begin{equation} \label{wq2b}
\absbb{\sum_i G_{\f v i}  G_{\f v i}G_{i \f v}} \;\prec\; N^{1 + (C_0 + 3) \delta} \Psi^2 \;\leq\; N^{1/2 - c_0}
\end{equation}
for small enough $\delta$. Putting \eqref{wq2a} and \eqref{wq2b} together, it is easy to deduce \eqref{main Wig 2}.
 
\subsubsection*{The case $q = 3$}
Now $\prod_{r = 1}^q A_{\f v,i,j }(w_r)$ is of the form $\pb{G_{\f v i} G_{j \f v}}^3$, or an expression obtained by exchanging $i$ and $j$ in some of the three factors. We estimate, using \eqref{oldPsi} and \eqref{shale},
\begin{multline*}
\absbb{\sum_i G_{\f v i}^3} \;\leq\; 4 \sum_i \abs{\wt G_{\f v i}}^3 + 4 \sum_i \abs{\Pi_{\f v i}}^3
\\
\prec\; \sum_i \pb{\abs{G_{\f v i}}^2 + \abs{\Pi_{\f v i}}^2} N^{2 \delta} \Phi + 1 \;\prec\; N^{1 + (C_0 + 1) \delta} \Psi^2 \Phi + N^{2 \delta} \Phi + 1\,,
\end{multline*}
where we used that $\sum_i \abs{\Pi_{\f v i}}^2 \leq \tau^{-2}$. Now \eqref{main Wig 2} is easy to conclude using $\Psi \geq N^{-1/2 + \tau/2}$.

\subsection{Proof of Theorem \ref{lem:general Wig} (ii)}
The proof is similar to that from Section \ref{sec:avg_proof}. As in Section \ref{sec:Wigner_ll}, it suffices to prove the following result.

\begin{lemma} \label{lem:Wig_main_est_avg}
Let $z \in \f S$ and suppose that the anisotropic local law holds at $z$. Define $\wt F(W) \deq \absb{\frac 1N  \sum_i G_{ii} - m}$. Then we have
\begin{equation}\label{lowzYY}
N^{-n/2}   \sum_{i\leq j} \absBB{\E \pbb{\frac{\partial}{\partial W_{ij}}}^n \wt F^p(W)}
\;=\; O \pBB{(N^{\delta}\Psi^2)^{p}+ \pbb{\frac{N^{-c_0/2}}{N\eta}}^{ p} +\E \wt F^p(W) }
\end{equation}
for any $n = 4, 5, \dots, 4p$. Moreover, if in addition $\Psi(z) \leq N^{-1/4-c_0}$ then \eqref{lowzYY} holds also for $n = 3$.
\end{lemma}

The case $n\geq 4$ can be easily proved as in covariance case. We therefore focus on the case $n = 3$ in Lemma \ref{lem:Wig_main_est_avg}. The proof is similar to the discussion below \eqref{main Wig}. The main difference is that for each $q$ we have some extra averaging
$N^{-q}\sum_{s_1, \dots, s_q} (\,\cdot\,)$, and we need to extract an extra factor $\Psi^q$ (or, alternatively, $(N^{-1-c_0/2}\eta^{-1}\Psi^{-1})^q$) from this average. We take over the notations from Sections \ref{sec:words} and \ref{sec:Wigner_ll} without further comment. We consider the three cases $q = 1,2,3$ separately, and tacitly use the anisotropic local law from Theorem \ref{lem:general Wig} (i).

\subsubsection*{The case $q = 1$}
Consider first the case $A_{s,s,i,j }(w_1) = G_{si}  G_{jj}G_{ij}G_{is}$. We estimate
\begin{equation*}
\absbb{\sum_{j} G_{ij}G_{jj}} \;\leq\;  \absbb{\sum_{j} G_{ij}\Pi_{jj}}+O_\prec (N\Psi^2) \;\prec\; N^{1/2}
\,, \qquad \sum_i \abs{G_{si}G_{is}} \;\prec\; N \Psi^2\,.
\end{equation*}
This gives
 $$
\absbb{\frac{1}{N}\sum_s\sum_{i, j}G_{si}  G_{jj}G_{ij}G_{is}} \;\prec\; N^{3/2}\Psi^2\,,
 $$
 as desired. Next, in the case $A_{s,s,i,j }(w_1) = G_{si} G_{ji} G_{ji}G_{js}$ we estimate
 $$
\absbb{\sum_{i, j}G_{si} G_{ji} G_{ji}G_{js}} \;\leq\; \sum_{i, j}(\abs{G_{si}}^2+\abs{G_{sj}}^2) \abs{G_{ji}}^2 \;\prec\; N^2\Psi^4 \;\prec\; N^{3/2}\Psi^2\,.
 $$
 Finally, in the case $A_{s,s,i,j }(w_1) = G_{si} G_{jj} G_{ii} G_{js}$ we estimate
\begin{equation*}
\absbb{\sum_i G_{si} G_{ii}} \;\leq\; \absbb{\sum_i G_{si} \wt G_{ii}} + \absbb{\sum_i G_{si} \Pi_{ii}} \;\prec\; N \Psi^2 + N^{1/2} \Psi \;\leq\; C N \Psi^2\,.
\end{equation*}
Using a similar bound for the sum over $j$, we find $\frac{1}{N} \sum_s \sum_{i,j} G_{si} G_{jj} G_{ii} G_{js} = O_\prec(N^2 \Psi^4) = O_\prec(N^{3/2} \Psi^2)$. All other terms are obtained from these three by exchanging $i$ and $j$.

\subsubsection*{The case $q = 2$}
In this case $N^{-1} \sum_{s_1, s_2} \prod_{r = 1}^2 A_{s_r,s_r,i,j }(w_r)$ is of the form
\begin{equation}
\frac{1}{N^2} \sum_{s_1, s_2} G_{s_1 i} G_{j s_1} G_{s_2 i} G_{ji} G_{j s_2} \qquad \txt{or} \qquad \frac{1}{N^2} \sum_{s_1, s_2} G_{s_1 i} G_{j s_1} G_{s_2 i} G_{jj} G_{i s_2}\,,
\end{equation}
or an expression obtained from one of these two by exchanging $i$ and $j$. These may be written as
\begin{equation} \label{piece}
\frac{1}{N^2} (G^2)_{ji}^2 G_{ji} \qquad \txt{or} \qquad  \frac{1}{N^2} (G^2)_{ji} (G^2)_{ii} G_{jj}\,.
\end{equation}
We estimate the contribution of the first expression by
\begin{multline*}
\absbb{\sum_{i,j}  \frac{1}{N^2} (G^2)_{ji}^2 G_{ji}} \;\prec\; \frac{1}{N^2} \tr|G|^4 \;=\; \frac{1}{N^2} \sum_k \frac{1}{\p{(\lambda_k - E)^2 + \eta^2}^2} \;\leq\; \frac{1}{N^2 \eta^3} \sum_k \frac{\eta}{(\lambda_k - E)^2 + \eta^2}
\\
\;\prec\; N^2 \frac{\im m + \Psi}{(N\eta)^3} \;\leq\; 2 N^2 \frac{1}{(N \eta)^2} \Psi^2 \;\leq\; 2 N^{3/2} \pbb{\frac{N^{-c_0}}{N \eta}}^2\,.
\end{multline*}
Next, we split the contribution of the second expression of \eqref{piece} as
  $$
 \sum_{i,j}  \frac{1}{N^2} (G^2)_{ji}   (G^2)_{ii}  G_{jj} \;=\;
  \sum_{i,j}   \frac{1}{N^2} (G^2)_{ji}   (G^2)_{ii}  \wt G_{jj}+  
  \sum_{i,j}   \frac{1}{N^2} \Pi_{jj}(G^2)_{ji}   (G^2)_{ii}\,.
  $$
Using the anisotropic local law, it is easy to prove that $\abs{(G^2)_{ij}} \prec N\Psi^2$ and $\absb{\sum_j \Pi_{jj} (G^2)_{ji}} \prec N^{3/2}\Psi^2$. Therefore 
\begin{multline*}
\absbb{\sum_{i,j}  \frac{1}{N^2} (G^2)_{ji}   (G^2)_{ii}  G_{jj}}
\;\prec\; \Psi^3(\tr|G|^4)^{1/2}+ N^{3/2}\Psi^4
\\
\prec\; N^2\frac{1}{N \eta}\Psi^4+ N^{3/2}\Psi^4 \;\leq\; N^{3/2} \Psi^2 \frac{N^{- c_0}}{N \eta} + N^{3/2} \Psi^4\,,
\end{multline*}
where we estimate $\tr \abs{G}^4$ as above. This concludes the proof in the case $q = 2$.

\subsubsection*{The case $q = 3$}
In this case $\sum_{s_1, s_2, s_3} \prod_{r = 1}^3 A_{s_r, s_r,i,j }(w_r)$ is of the form $(G^2)_{ij}^3$, or an expression obtained by exchanging $i$ and $j$ in some of the three factors. We estimate its contribution by
 $$
\absbb{\frac{1}{N^3} \sum_{i,j} (G^2)_{ij}^3} \;\prec\; N^{-2}\tr |G|^4\Psi^2
\;\leq\; N^2 \frac{1}{(N\eta)^2} \Psi^4 \;\prec\; N^{3/2} \pbb{\frac{N^{-c_0}}{N \eta}}^2 \Psi^2\,,
 $$
 which concludes the proof.
  
\subsection{Proof of Theorem \ref{thm:wigner_rig}}
The proof is analogous to that of Theorem \ref{thm:rigidity} in Section \ref{sec:rigidity}. Define the domain
$$
\f S_\tau \;\deq\; \hB{ z\col N^{-2/3+\tau} \leq \dist(E, [L_-, L_+]) \leq \tau^{-1} \,,\, N^{-\tau} \kappa(E) \leq \eta \leq  N^{-\delta/2} \kappa(E)}\,.
$$
From the assumptions on $W^{\txt{Gauss}} + D$, it is not hard to deduce the estimate
$$
\frac{1}{N} \tr \pb{W^{\txt{Gauss}} + D - z}^{-1} - m(z) \;=\; O_\prec \pbb{\frac{N^{-c}}{N\eta}}
$$
for all $z \in \f S_\tau$, where $c > 0$ is a positive constant that depends only on $\tau$.

Next, from Theorem \ref{lem:general Wig} (ii), we have $N^{-1} \tr G (z)- m(z) = O_\prec((N\eta)^{-1})$ for $z\in\f S _\tau$, which is not enough to establish the rigidity of the extreme eigenvalues. However, analogously to Proposition \ref{prop:gen_avg}, our proof of Theorem \ref{lem:general Wig} (ii) in fact yields a stronger result. Indeed, Lemma \ref{lem:Wig_main_est_avg} implies that
$$
\absbb{\frac{1}{N} \tr G (z) - m(z)} \;\prec\;  \frac{N^{-c}}{N\eta} + \Psi^2 \;\prec\; \frac{N^{-c}}{N\eta}
$$
for $z\in\f S _\tau$. We may now repeat the argument starting at \eqref{tur1}, with trivial modifications, to deduce that $[\lambda_1 (W+A)-L_+]_+\prec N^{-2/3}$. This concludes the proof of Theorem \ref{thm:wigner_rig}.

\subsection {Proof of Theorem \ref{lem:comp Wig}}
Analogously to the proof of Theorem \ref{thm:edge_univ}, the proof is a routine application of the Green function comparison method near the edge \cite[Section 6]{EYY3}. The key technical inputs are Theorems \ref{lem:general Wig} and \ref{thm:wigner_rig}. Note that Theorem \ref{lem:general Wig} is applicable since the Green function comparison argument only involves $z$ satisfying $\abs{\Psi(z)} \leq N^{-1/3-c}$ with some small constant $c>0$. Hence the assumption (b) from Theorem \ref{lem:general Wig} is satisfied.

\appendix

\section{Properties of $\varrho$ and Stability of \eqref{def_m}} \label{sec:varrho}

This appendix is devoted the proofs of the basic properties of $\varrho$ and the stability of \eqref{def_m} in the sense of Definition \ref{def:stability}. In this appendix we abbreviate $r_i \deq \phi \pi(\{s_i\})$.

\subsection{Properties of $\varrho$ and proof of Lemmas \ref{lem:crit points}--\ref{lem:rho_struct}}

In this subsection we establish the basic properties of $\varrho$ and prove Lemmas \ref{lem:crit points}--\ref{lem:rho_struct}.

\begin{proof}[Proof of Lemma \ref{lem:crit points}]
We have
\begin{equation} \label{f_der}
f'(x) \;=\; \frac{1}{x^2} - \sum_{i = 1}^n \frac{r_i}{(x + s_i^{-1})^2}\,, \qquad f''(x) \;=\; \frac{-2}{x^3} + 2 \sum_{i = 1}^n \frac{r_i}{(x + s_i^{-1})^3}\,.
\end{equation}
From \eqref{f_der} we find that $(x^3 f''(x))' > 0$ on $I_0 \cup \cdots \cup I_n$. Therefore for $i = 2, \dots, n$ there is at most one point $x \in I_i$ such that $f''(x) = 0$. We conclude that $I_i$ has at most two critical points of $f$. Using the boundary conditions of $f$ on $\partial I_i$, we conclude the proof of $\abs{\cal C \cap I_i} \in \{0,2\}$ for $i = 2, \dots n$.

From \eqref{f_der} we also find that $(x^2 f'(x))' < 0$ for $x \in I_1$. We conclude that there exists at most one point $x \in I_1$ such that $f'(x) = 0$. Using the boundary conditions of $f'$ on $\partial I_1$, we deduce that $\abs{\cal C \cap I_1} = 1$.

Finally, if $\phi \neq 1$ we have $f(x) = (\phi - 1) x^{-1} + O(x^{-2})$ as $x \to \infty$. From the boundary conditions of $f$ on $\partial I_0$ we therefore deduce that $\abs{\cal C \cap I_0} = 1$. Moreover, if $\phi = 1$ we find from \eqref{f_der} that $f'(x) \neq 0$ for all $x \in I_0 \setminus \{\infty\}$. This concludes the proof.
\end{proof}

By multiplying both sides of the equation $z = f(m)$ in \eqref{def_m} with the product of all denominators on the right-hand side of \eqref{def_fx}, we find that $z = f(m)$ may be also written as $P_z(m) = 0$, where $P_z$ is a polynomial of degree $n+1$, whose coefficients are affine linear functions of $z$. (Here we used that all $s_i$ are distinct and that $m + s_i^{-1} \neq 0$.) This polynomial characterization of $m$ is useful for the proofs of Lemmas \ref{lem:a_order} and \ref{lem:rho_struct}.

From Lemma \ref{lem:gen_prop_m} we find that the measure $\varrho$ has a bounded density in $[\tau,\infty)$ for any fixed $\tau > 0$. Hence, we may extend the definition of $m$ down to the real axis by setting $m(E) \deq \lim_{\eta \downarrow 0} m(E + \ii \eta)$.

\begin{proof}[Proof of Lemma \ref{lem:a_order}]
For $i = 0, \dots, n$ define the subset $J_i \deq \h{x \in I_i \col f'(x) > 0}$. The graph of $f$ restricted to $J_0 \cup \cdots \cup J_n$ is depicted in red in Figures \ref{fig:plot_f} and \ref{fig:plot_f2}. From Lemma \ref{lem:crit points} we deduce that if $i = 2, \dots, n$ then $J_i \neq \emptyset$ if and only if $I_i$ contains two distinct critical points of $f$, in which case $J_i$ is an interval. Moreover, we always have $J_1 \neq \emptyset \neq J_0$.

Next, we observe that for any $0 \leq i < j \leq n$ we have $f(J_i) \cap f(J_j) = \emptyset$. Indeed, if there were $E \in f(J_i) \cap f(J_j)$ then we would have $\abs{\h{x \col f(x) = E}} > n+1$ (see Figure \ref{fig:plot_f}). Since $f(x) = E$ is equivalent to $P_E(x) = 0$ and $P_E$ has degree $n+1$, we arrive at the desired contradiction. We conclude that the sets $f(J_i)$, $0 \leq i \leq n$, may be strictly ordered.

The claim $a_1 \geq a_2 \geq \cdots \geq a_{2p}$ may now be reformulated as
\begin{equation} \label{f_incr}
f(J_i) \;>\; f(J_j) \;>\; f(J_0) \qquad \txt{if } 1 \leq i < j \leq n\,.
\end{equation}

In order to show \eqref{f_incr}, we use a continuity argument in $\phi$. Let $t \in (0,1]$ and denote by $f^t$ the right-hand side of \eqref{def_fx} with $\phi$ is replaced by $t \phi$. We use the superscript $t$ to denote quantities defined using $f^t$ instead of $f = f^1$. Using directly the definition of $f^t$, it is easy to check that \eqref{f_incr} holds for small enough $t > 0$.

We focus on the first inequality of \eqref{f_incr}; the proof of the second one is similar. We have to show that $f(J_i) > f(J_j)$ whenever $1 \leq i < j \leq n$ and $J_i \neq \emptyset \neq J_j$. First, we claim that for $1 \leq i \leq n$ we have
\begin{equation} \label{J_remains}
J_i \neq \emptyset \quad \Longrightarrow \quad J_i^t \neq \emptyset \txt{ for all } t \in (0,1]\,.
\end{equation}
To prove \eqref{J_remains}, we remark that if $2 \leq i \leq n$ then $J_i^t \neq \emptyset$ is equivalent to $I_i$ containing two distinct critical points. Moreover, $\partial_t \partial_x f^t(x) < 0$ in $I_2 \cup \cdots \cup I_n$, from which we deduce that the number of distinct critical points in each $I_i$, $i = 2, \dots, n$, does not decrease as $t$ decreases. Recalling that $I_1^t \neq \emptyset$, we deduce \eqref{J_remains}.

Next, suppose that there exist $1 \leq i < j \leq n$ satisfying $J_i \neq \emptyset \neq J_j$ and $f(J_i) < f(J_j)$. From \eqref{J_remains} we deduce that $J_i^t \neq \emptyset \neq J_j^t$ for all $t \in (0,1]$. Moreover, by a simple continuity argument and the fact that for each $t$ we have either $f^t(J_i^t) < f^t(J_j^t)$ or  $f^t(J_i^t) > f^t(J_j^t)$, we conclude that $f^t(J_i^t) < f^t(J_j^t)$ for all $t \in (0,1]$. As explained after \eqref{f_incr}, this is impossible for small enough $t > 0$. This concludes the proof of \eqref{f_incr}, and hence also of the first sentence of Lemma \ref{lem:a_order}.

Next, we prove the second sentence of Lemma \ref{lem:a_order}. Suppose that $\phi \neq 1$ and that $E \neq 0$ lies in the (open) set $\bigcup_{i = 0}^n f(J_i) = f \pb{\bigcup_{i = 0}^n J_i}$. Then the set $\h{x \col f(x) = E} = \h{x \col P_E(x) = 0} \subset \R$ has $n+1$ points. Recall that $m(E)$ is a solution of $P_{E}(m(E)) = 0$. Since $n+1 = \deg P_E$, we deduce that $m(E)$ is real. Since $m$ is the Stieltjes transform of $\varrho$, we conclude that $E \notin \supp \varrho$ and that $m'(E) > 0$. This means that $m \vert_{f(J_i)}$ is increasing, and is therefore the inverse function of $f \vert_{J_i}$. We conclude that $f$ is a bijection from $\bigcup_{i = 0}^n J_i$ onto $\bigcup_{i = 0}^n f(J_i)$ with inverse function $m$ (here we use the convention that for $\phi \leq 1$ we have $m(0) = \infty$ and $f(\infty) = 0$). By definition, we have $a_k = f(x_k)$ and the inverse relation $x_k = m(a_k)$ follows by a simple limiting argument from $\C_+$ to $\R$.
The case $\phi = 1$ is handled similarly; we omit the minor modifications.

Finally, in order to prove the third sentence of Lemma \ref{lem:a_order} we have to show that $a_{2p} \geq 0$ and $a_1 \leq C$. It is easy to see that if $a_{2p} < 0$ then there is an $E < 0$ with $\abs{\{x \col f(x) = E\}} > n+1$, which is impossible since $\{x \col f(x) = E\} = \{x \col P_E(x) = 0\}$ and $\deg P_E = n+1$. Moreover, the estimate $a_1 \leq C$ is easy to deduce from the definition of $f$ and the estimates $\phi \leq \tau^{-1}$ and $s_1^{-1} \geq \tau$.
\end{proof}

\begin{proof}[Proof of Lemma \ref{lem:rho_struct}]
In the proof of Lemma \ref{lem:a_order} we showed that if $E \in (0,\infty) \cap \bigcup_{i = 0}^n f(J_i)$ then $E \notin \supp \varrho$. It therefore suffices to show that if $E \in \bigcup_{k = 1}^p (a_{2k}, a_{2k - 1})$ then $\im m(E) > 0$. In that case the set of real preimages $\{x \col f(x) = E\} = \{x \col P_E(x) = 0\}$ has $n - 1$ points (see Figure \ref{fig:plot_f}). Since $P_E$ has degree $n+1$ and real coefficients, we conclude that $P_E$ has a root with positive imaginary part. By the uniqueness of the root of $P_{E + \ii \eta}$ in $\C_+$ (see \eqref{def_m}) and the continuity of the roots of $P_{E + \ii \eta}$ in $\eta$, we therefore conclude by taking $\eta \downarrow 0$ that $\im m(E) > 0$.
\end{proof}

For the following counting argument, in order to avoid extraneous complications we assume that $\sigma_i > 0$ for all $i$. Recall the definition of $N_k$ from \eqref{def_N_k}.

\begin{lemma} \label{lem:N_i}
Suppose that $\sigma_j > 0$ for all $j$.
Then $N \varrho(\{0\}) = (N-M)_+$ and
\begin{equation} \label{Ni_sum}
\sum_{k = 1}^p N_k \;=\; N - (N - M)_+ \;=\; N \wedge M\,.
\end{equation}
Moreover, for $k = 1, \dots, p - 1$ we have
\begin{equation} \label{Ni_rep}
N_k \;=\; \sum_{i \in \cal I_M} \ind{x_{2 k} \leq - \sigma_i^{-1} \leq x_{2k - 1}}\,.
\end{equation}
\end{lemma}

\begin{proof}
Suppose first that $N > M$. Then $x_{2p} < 0$ and $a_{2p} > 0$ (see Figure \ref{fig:plot_f}). Let $1 \leq k \leq p$ and define $\gamma$ as the positively oriented rectangular contour joining the points $a_{2k - 1} \pm \ii$ and $a_{2k} \pm \ii$. Then we have $N_k = - \frac{N}{2 \pi \ii} \oint_{\gamma} m(z) \, \dd z$. We now do a change of coordinates by writing $w = m(z)$. Define $\tilde \gamma$ as the positively oriented rectangular contour joining the points $x_{2k - 1} \pm \ii$ and $x_{2k} \pm \ii$. Since the imaginary parts of $w$ and $z$ have the same sign, we conclude using Cauchy's theorem that
\begin{multline*}
N_k \;=\; - \frac{N}{2 \pi \ii} \oint_{\gamma} m(z) \, \dd z \;=\; -\frac{N}{2 \pi \ii} \oint_{f(\tilde \gamma)} m(z) \, \dd z
\\
=\; - \frac{N}{2 \pi \ii} \oint_{\tilde \gamma} w \, f'(w) \, \dd w \;=\; \sum_{i = 1}^n \ind{x_{2 k} \leq - s_i^{-1} \leq x_{2k - 1}} \, \pi (\{s_i\}) M\,,
\end{multline*}
where in the last step we used the definition of $f$ from \eqref{def_fx}. Hence \eqref{Ni_rep} follows. Note that we proved \eqref{Ni_rep} also for $k = p$. Now \eqref{Ni_sum} follows easily using $\sigma_i > 0$ for all $i$. Since $\supp \varrho \subset \{0\} \cup [a_{2p}, a_{2p - 1}] \cup \cdots \cup [a_2,a_1]$ (recall Lemma \ref{lem:rho_struct}), we get $N = N \varrho(\{0\}) + \sum_{k = 1}^p N_k$. Therefore summing over $k = 1, \dots, p$ in \eqref{Ni_rep} yields $N \varrho(\{0\}) = N - M$, as desired. This concludes the proof in the case $N > M$.

Next, suppose that $N \leq M$. Then we deduce directly from \eqref{def_fx} that $\varrho(\{0\}) = 0$ (see Figure \ref{fig:plot_f2}). Hence, \eqref{Ni_sum} follows. Finally, exactly as above, we get \eqref{Ni_rep} for $i = 1, \dots, p-1$. This concludes the proof.
\end{proof}

\subsection{Stability near a regular edge}
The rest of this appendix is devoted to the proofs of the assumption \eqref{1+xm} and the stability of \eqref{def_m} in the sense of Definition \ref{def:stability}. In fact, we shall prove the following stronger notation of stability, which will be needed to establish eigenvalue rigidity. Recall the definition of $\kappa$ from \eqref{def_kappa}.
\begin{definition}[Strong stability of \eqref{def_m} on $\f S$] \label{def:strong_stab}
We say that \eqref{def_m} is \emph{strongly stable on} $\f S$ if Definition \ref{def:stability} holds with \eqref{u-m_stability} replaced by
\begin{equation} \label{strong_stab}
\abs{u - m} \;\leq\; \frac{C \delta}{\sqrt{\kappa + \eta} +  \sqrt{\delta}}\,.
\end{equation}
\end{definition}
That the notion of stability from Definition \ref{def:strong_stab} it is stronger than that of Definition \ref{def:stability} follows from the estimate
\begin{equation} \label{m_sq_root}
\im m \;\asymp\; 
\begin{cases}
\sqrt{\kappa + \eta} & \text{if $E \in \supp \varrho$}
\\
\frac{\eta}{\sqrt{\kappa + \eta}} & \text{if $E \notin \supp \varrho$}
\end{cases}
\end{equation}
for $z \in \f S$, which we shall establish under the regularity assumptions of Definition \ref{def:regular}. Under the same regularity assumptions, we shall also prove
\begin{equation} \label{mz_reg}
\min_i \abs{m + s_i^{-1}} \;\geq\; c
\end{equation}
for $z \in \f S$ and some constant $c > 0$ depending only on $\tau$. Using Lemma \ref{lem:gen_prop_m}, we find that \eqref{mz_reg} implies $\abs{1 + m \sigma_i} \geq c$, which is the estimate from \eqref{1+xm} (after a possible renaming of the constant $\tau$).

In this subsection we deal with a regular edge; the case of a regular bulk component is dealt with in the next subsection.
We begin with basic estimates on the behaviour of $f$ near a critical point.

\begin{lemma} \label{lem:edge_tools}
Fix $\tau > 0$. Suppose that the edge $k$ is regular in the sense of Definition \ref{def:regular} (i). Then there exists a $\tau' > 0$, depending only on $\tau$, such that
\begin{equation} \label{f_der_est}
\abs{x_k} \;\asymp\; 1\,, \qquad \abs{f''(x_k)} \;\asymp\; 1\,, \qquad \abs{f'''(\zeta)} \;\leq\; C \quad \txt{for} \quad \abs{\zeta - x_k} \;\leq\; \tau'\,.
\end{equation}
\end{lemma}

\begin{proof}
From Lemma \ref{lem:a_order} we get $x_k = m(a_k)$. Since by assumption $a_k \in [\tau,C]$ for some constant $C$, we find from Lemma \ref{lem:gen_prop_m} that $\abs{x_k} \leq C$. The lower bound $\abs{x_i} \geq c$ follows from $f'(x_k) = 0$ and \eqref{f_der}, using that $\abs{s_i} \leq C$. This proves the first estimate of \eqref{f_der_est}.

Next, from $f'(x_k)$ and \eqref{f_der} we find
\begin{equation} \label{f''_expr}
f''(x_k) \;=\; - \frac{2}{x_k} \sum_i \frac{s_i^{-1} r_i}{(x_k + s_i^{-1})^3}\,.
\end{equation}
Using the first estimate of \eqref{f_der_est} and \eqref{reg: edge} we find from \eqref{f''_expr} that $\abs{f''(x_k)} \leq C$. Moreover for $k \in \{1, 2p\}$ (i.e.\ $x_k \in I_0 \cup I_1$) all terms in \eqref{f''_expr} have the same sign, and it is easy to deduce that $\abs{f''(x_k)} \geq c$.

Next, we prove the lower bound $\abs{f''(x_k)} \geq c$ for $k = 2, \dots, 2p - 1$. Suppose for definiteness that $k$ is odd. (The case of even $k$ is handled in exactly the same way.) For $x \in [x_k, x_{k-1}]$ we have $f'(x) \geq 0$ and
\begin{equation*}
f'(x) \;=\; \int_{x_k}^x (-y^{-3}) (- y^3 f''(y)) \, \dd y \;\leq\; \int_{x_k}^x (-y^{-3}) (- x_k^3 f''(x_k)) \, \dd y \;\leq\; C f''(x_k) (x_{k-1} - x_k)\,,
\end{equation*}
where in the second step we used that $(-x^3 f''(x))' < 0$ as follows from \eqref{f_der}, and in the last step the estimate $\abs{x_k} \asymp \abs{x_{k-1}} \asymp 1$. We therefore find
\begin{equation*}
a_{k-1} - a_k \;=\; \int_{x_k}^{x_{k-1}} f'(x) \, \dd x \;\leq\; C f''(x_k) (x_{k-1} - x_k)^2\,.
\end{equation*}
Since $a_{k-1} - a_k \geq \tau$ by the assumption in Definition \ref{def:regular} (i) and $x_{k-1} - x_k \leq \abs{x_{k-1}} + \abs{x_k} \leq C$, we deduce the lower bound $f''(x_k) \geq c$.

Finally, the last estimate of \eqref{f_der_est} follows easily from the assumption \eqref{reg: edge} and by differentiating \eqref{f_der}.
\end{proof}

We record the following easy consequence of \eqref{f_der_est}. Recalling that $f(x_k) = a_k$ and $f'(x_k) = 0$, we find, using \eqref{f_der_est} and the fact that $f$ is continuous in a neighbourhood of $x_k$, that after choosing $\tau' > 0$ small enough (depending on $\tau$) we have
\begin{equation} \label{z-a_taylor}
z - a_k \;=\; \frac{f''(x_k)}{2} (m(z) - x_k)^2 + O \pb{\abs{m(z) - x_k}^3} \qquad \txt{for} \quad \abs{z - a_k} \leq \tau'\,.
\end{equation}

\begin{lemma} \label{lem:stab_edge1}
Fix $\tau > 0$. Suppose that the edge $k$ is regular in the sense of Definition \ref{def:regular} (i). Then there exists a constant $\tau' > 0$, depending only on $\tau$, such that for $z \in \f D_k^e$ we have \eqref{m_sq_root} and \eqref{mz_reg}.
\end{lemma}

\begin{proof}
We first prove \eqref{m_sq_root}.
From \eqref{z-a_taylor} we find that there exists a $\tau' > 0$ such that for $\abs{E - a_k} \leq \tau'$ we have
\begin{equation} \label{E-a_exp}
E - a_k \;=\; \frac{f''(x_k)}{2} \pb{m(E) - m(a_k)}^2 \pb{1 + O(\abs{m(E) - m(a_k)})}\,.
\end{equation}
From \eqref{f_der_est} we find that $\abs{f''(x_k)} \geq c$. For definiteness, suppose that $k$ is odd. Then $\im m(E) = 0$ for $E \geq a_k$ (i.e.\ $E \notin \supp \varrho$). From \eqref{E-a_exp} we therefore deduce that for $E \leq a_k$ we have $\im m(E) \asymp \sqrt{\abs{E - a_k}}$. We conclude that for $\abs{E - a_k} \leq \tau'$ we have the square root behaviour $\im m(E) \asymp \sqrt{\kappa(E)} \ind{E \in \supp \varrho}$ for the density of $\varrho$. Now \eqref{m_sq_root} easily follows from the definition of the Stieltjes transform.

Finally, \eqref{mz_reg} follows from the assumption \eqref{reg: edge} and \eqref{z-a_taylor}.
\end{proof}

\begin{lemma}\label{lem:stab_edge2}
Fix $\tau > 0$. Suppose that the edge $k$ is regular in the sense of Definition \ref{def:regular} (i). Then there exists a constant $\tau'$, depending only on $\tau$, such that \eqref{def_m} is strongly stable on $\f D^e_k$ in the sense of Definition \ref{def:strong_stab}.
\end{lemma}

\begin{proof}
We take over the notation of Definition \ref{def:stability}. Abbreviate $w(z) \deq f(u(z))$, so that $\abs{w(z) - z} \leq \delta(z)$.

Suppose first that $\im z \geq \tau'$ for some constant $\tau' > 0$. Then, from the assumption $\delta(z) \leq (\log N)^{-1}$ we get $\im w(z) \geq \tau'/2$, and therefore from the uniqueness of \eqref{def_m} we find $u(z) = m(w(z))$. Hence,
\begin{equation*}
\abs{u(z) - m(z)} \;=\; \abs{m(w(z)) - m(z)} \;\leq\; \int_{w(z)}^z \abs{m'(\zeta)} \, \abs{\dd \zeta} \;\leq\; 2 (\tau')^{-2} \delta(z)\,,
\end{equation*}
where we used the trivial bound $\abs{m'(\zeta)} \leq (\im \zeta)^{-2}$. This yields \eqref{strong_stab} at $z$.

What remains therefore is the case $\abs{z - a_k} \leq 2 \tau'$, which we assume for the rest of the proof. We drop the arguments $z$ and write the equation $f(u) - f(m) = w - z$ as
\begin{equation} \label{quadr_edge}
\alpha \, (u - m)^2 + \beta \, (u - m) \;=\; u \, m \, (w-z)\,,
\end{equation}
where
\begin{equation} \label{def_al_be}
\alpha \;\deq\; - \sum_i \frac{m s_i^{-1} r_i}{(m + s_i^{-1})^2 (u + s_i^{-1})}\,, \qquad
\beta \;\deq\; 1 - \sum_i \frac{m^2 r_i}{(m + s_i^{-1})^2}\,.
\end{equation}
Note that $\alpha$ and $\beta$ depend on $z$, and $\alpha$ also depends on $u$. We suppose that
\begin{equation} \label{edge_assump}
\abs{z - a_k} \;\leq\; 2 \tau' \,, \qquad \abs{u - m} \;\leq\; (\log N)^{-1/3}\,.
\end{equation}
Then we claim that for small enough $\tau'$ we have
\begin{equation} \label{al_be_claim}
\abs{\alpha} \;\asymp\; 1 \,, \qquad \abs{\beta} \;\asymp\; \sqrt{\kappa + \eta}.
\end{equation}

In order to prove \eqref{al_be_claim}, we note that, by Lemma \ref{lem:edge_tools}, the statement of \eqref{z-a_taylor} holds under the assumption \eqref{edge_assump} provided $\tau' > 0$ is chosen small enough.
Using \eqref{edge_assump}, \eqref{mz_reg} (see Lemma \ref{lem:edge_tools}) , and \eqref{f''_expr} we get
\begin{equation*}
\alpha \;=\; - \sum_i \frac{x_k s_i^{-1} r_i}{(x_k + s_i^{-1})^3} + O \pb{\abs{m - x_k} + (\log N)^{-1}} = \frac{x_k^2}{2} f''(x_k) + O \pb{\abs{m - x_k} + (\log N)^{-1}}\,.
\end{equation*}
Using \eqref{f_der_est} and \eqref{z-a_taylor} we conclude that for small enough $\tau' > 0$ we have $\abs{\alpha} \asymp 1$. This concludes the proof of the first estimate of \eqref{al_be_claim}.

In order to prove the second estimate of \eqref{al_be_claim}, we note that $\beta = m^2 f'(m)$, so that for small enough $\tau' > 0$ we get from \eqref{z-a_taylor} and Lemma \ref{lem:edge_tools} that
\begin{equation*}
m^{-2} \beta \;=\; \int_{x_k}^m f''(\zeta) \, \dd \zeta \;=\; \int_{x_k}^m \pb{f''(x_k) + O(\abs{m - x_k})} \, \dd \zeta \;=\; (m - x_k) f''(x_k) + O(\abs{m - x_k}^2)\,.
\end{equation*}
Using \eqref{z-a_taylor} and Lemmas \ref{lem:edge_tools} and \eqref{lem:gen_prop_m}, we conclude for small enough $\tau' > 0$ that
\begin{equation*}
\abs{\beta} \;\asymp\; \abs{m - x_k} \;\asymp\; \sqrt{\abs{z - a_k}} \;\asymp\; \sqrt{\kappa + \eta}\,,
\end{equation*}
This concludes the proof of the second estimate of \eqref{al_be_claim}.

Using \eqref{m_sim_1} and the estimate $\abs{w - z} \leq \delta$, we write \eqref{quadr_edge} as $\alpha (u - m)^2 + \beta (u - m) = O(\delta)$. We now proceed exactly as in the proof of \cite[Lemma 4.5]{BEKYY}, by solving the quadratic equation for $u - m$ explicitly. We select the correct solution by a continuity argument using that \eqref{strong_stab} holds by assumption at $z + \ii N^{-5}$. The second assumption of \eqref{edge_assump} is obtained by continuity from the estimate on $\abs{u - m}$ at the neighbouring value $z + \ii N^{-5}$. We refer to \cite[Lemma 4.5]{BEKYY} for the full details of the argument. (Note that the fact $\alpha$ depends on $u$ does not change this argument. This dependence may be easily dealt with using a simple fixed point argument whose details we omit.) This concludes the proof.
\end{proof}

\subsection{Stability in a regular bulk component}

In this subsection we establish the basic estimates \eqref{m_sq_root} and \eqref{mz_reg} as well as the stability of \eqref{def_m} in a regular bulk component $k$, i.e.\ for $z \in \f D^b_k$.

\begin{lemma} \label{lem:stab_bulk}
Fix $\tau, \tau' > 0$ and suppose that the bulk component $k = 1, \dots, p$ is regular in the sense of Definition \ref{def:regular} (ii). Then for $z \in \f D^b_k$ we have \eqref{m_sq_root} and \eqref{mz_reg}. Moreover, \eqref{def_m} is strongly stable on $\f D^b_k$ in the sense of Definition \ref{def:strong_stab}.
\end{lemma}

\begin{proof}
The estimates \eqref{m_sq_root} and \eqref{mz_reg} follow trivially from the assumption in Definition \ref{def:regular} (ii), using $\kappa \leq C$ as follows from Lemma \ref{lem:a_order}.

In order to show the strong stability of \eqref{def_m}, we proceed as in the proof of Lemma \ref{lem:stab_edge2}. Mimicking the argument there, we find that it suffices to prove that $\abs{\alpha} \leq C$ and $\abs{\beta} \asymp 1$ for $z \in \f D^e_k$, where $\alpha$ and $\beta$ were defined in \eqref{def_al_be} and $u$ satisfies $\abs{u - m} \leq (\log N)^{-1/3}$.

Note that Definition \ref{def:regular} (ii) immediately implies that $\im m \geq c$ for some constant $c > 0$.
The upper bound $\abs{\alpha} + \abs{\beta} \leq C$ easily follows from the definition \eqref{def_al_be} combined with Lemma \ref{lem:gen_prop_m}.  What remains is the proof of the lower bound $\abs{\beta} \geq c$. To that end, we take the imaginary part of \eqref{def_m} to get
\begin{equation} \label{im_def_m}
\sum_i \frac{r_i \abs{m}^2}{\abs{m + s_i^{-1}}^2} \;=\; 1 - \frac{\abs{m}^2 \eta}{\im m}\,.
\end{equation}
Using that $\im m \geq c$, a simple analysis of the arguments of the expressions $(m + s_i^{-1})^2$ on the left-hand side of \eqref{im_def_m} yields
\begin{equation*}
\absBB{\sum_i \frac{r_i m^2}{(m + s_i^{-1})^2}} \;\leq\; \sum_i \frac{r_i \abs{m}^2}{\abs{m + s_i^{-1}}^2} - c \;\leq\; 1 - c\,,
\end{equation*}
where in the last step we used \eqref{im_def_m}. Recalling the definition of $\beta$ from \eqref{def_al_be}, we conclude that $\abs{\beta} \geq c$. This concludes the proof.
\end{proof}

\begin{remark} \label{rem:bulk_stab}
The bulk regularity condition from Definition \ref{def:regular} (ii) is stable under perturbation of $\pi$. To see this, define the shifted empirical density $\pi_t \deq M^{-1} \sum_{i = 1}^M \delta_{\sigma_i + t}$, and the associated Stieltjes transform $m_t(E)$ and function $f_t(x)$. Differentiating $f_t(m_t(E)) = E$ in $t$ yields $(\partial_t m)_t(E) = - (\partial_t f)_t(m_t(E)) / (\partial_x f)_t(m_t(E))$. Using that
\begin{equation*}
(\partial_t f)_0(m(E)) \;=\; \frac{1}{N} \sum_{i = 1}^M \frac{\sigma_i^{-2}}{(m(E) + \sigma_i^{-1})^2} \;=\; O(1)
\end{equation*}
by \eqref{mz_reg}, and $(\partial_x f)_t(m_t(E)) = m(E)^{-2} \beta(E)$, we find from $\abs{\beta(E)} \asymp \abs{m(E)} \asymp 1$ that $(\partial_t m)_0(E) = O(1)$ for $E \in [a_{2k} + \tau', a_{2k - 1} - \tau']$. A simple extension of this argument shows that if Definition \ref{def:regular} (ii) holds for $\pi$ then it holds for all $\pi_t$ with $t$ in some ball of fixed radius around zero.
\end{remark}

\subsection{Stability outside of the spectrum}

In this subsection we establish the basic estimates \eqref{m_sq_root} and \eqref{mz_reg} as well as the stability of \eqref{def_m} outside of the spectrum, i.e.\ for $z \in \f D^o$.

\begin{lemma} \label{lem:stab_outside}
Fix $\tau, \tau' > 0$. Then for $z \in \f D^o$ we have \eqref{m_sq_root} and \eqref{mz_reg}. Moreover, \eqref{def_m} is strongly stable on $\f D^o$ in the sense of Definition \ref{def:strong_stab}.
\end{lemma}

\begin{proof}
Since $m$ is the Stieltjes transform of a measure $\varrho$ with bounded density (see Lemma \ref{lem:gen_prop_m}), we find that
\begin{equation} \label{im_m_outside}
\im m \;\asymp\; \eta
\end{equation}
on $\f D^o$. This proves \eqref{m_sq_root}. 

We first establish \eqref{mz_reg} for $z = E \in \R$ satisfying $\dist(E, \supp \varrho) \geq \tau'$. For definiteness, let $z = E \in [a_{2l}, a_{2l+1}]$ for some $l$, and let $i = 2, \dots, n$ satisfy $\{a_{2l}, a_{2l+1}\} = f(\partial J_i)$, where we recall the set $J_i \deq \h{x \in I_i \col f'(x) > 0}$ from the proof of Lemma \ref{lem:a_order}. Then there exists an $x \in J_i$ such that $E = f(x)$. Moreover, from \eqref{f_der} we get $f'(y) \leq y^{-2}$ for $y \in \R$, so that for $y \in J_i$ we have $0 \leq f'(y) \leq C$ by \eqref{f_der_est}. Since $\dist(E, \{a_{2l}, a_{2l+1}\}) \geq \tau'$, we conclude that $\dist(x, \partial J_i) \geq c$, which implies $\dist(x, \partial I_i) \geq c$. The cases $E > a_1$ and $E < a_{2p}$ are handled similarly. This concludes the proof of \eqref{mz_reg} for $z = E \in \R$ satisfying $\dist(E, \supp \varrho) \geq \tau'$.

Next, we note that if $\eta \geq \epsilon$ for some fixed $\epsilon > 0$ then \eqref{mz_reg} is trivially true by \eqref{im_m_outside}. On the other hand, if $\im m \leq \epsilon$ then we get
\begin{equation*}
\min_i \abs{m(E + \ii \eta) + s_i^{-1}} \;\geq\; \min_i \abs{m(E) + s_i^{-1}} - C \epsilon \;\geq\; c - C \epsilon \;\geq\; c/2
\end{equation*}
for $\epsilon$ small enough, where we used the case $z = E \in \R$ satisfying $\dist(E, \supp \varrho) \geq \tau'$ established above and the estimate $\abs{m'(z)} = \absb{\int \frac{1}{(x - z)^2} \, \varrho(\dd x)} \leq C$ (since $\abs{x - z} \geq \tau'$ for $z \in \f D^o$ and $x \in \supp \varrho$). This concludes the proof of \eqref{mz_reg} for $z \in \f D^o$.

As in the proof of Lemma \ref{lem:stab_bulk}, to prove the strong stability of \eqref{def_m} on $\f D^o$, it suffices to show that $\abs{\alpha} \leq C$ and $\abs{\beta} \asymp 1$ for $z \in \f D^e_k$, where $\alpha$ and $\beta$ were defined in \eqref{def_al_be} and $u$ satisfies $\abs{u - m} \leq (\log N)^{-1/3}$. The upper bound $\abs{\alpha} + \abs{\beta} \leq C$ follows immediately from \eqref{mz_reg}, and the lower bound $\abs{\beta} \geq c$ follows from the definition \eqref{def_al_be} combined with \eqref{im_def_m}, \eqref{im_m_outside}, and Lemma \ref{lem:gen_prop_m}. This concludes the proof.
\end{proof}

\section*{Acknowledgements}
A.K.\ was partially supported by Swiss National Science Foundation grant 144662 and the SwissMAP NCCR grant. J.Y.\ was partially supported by NSF Grant DMS-1207961. We are very grateful to the Institute for Advanced Study, Thomas Spencer, and Horng-Tzer Yau for their kind hospitality during the academic year 2013-2014. We also thank the Institute for Mathematical Research (FIM) at ETH Z\"urich for its generous support of J.Y.'s visit in the summer of 2014. We are indebted to Jamal Najim for stimulating discussions.

{\small

\providecommand{\bysame}{\leavevmode\hbox to3em{\hrulefill}\thinspace}
\providecommand{\MR}{\relax\ifhmode\unskip\space\fi MR }
\providecommand{\MRhref}[2]{%
  \href{http://www.ams.org/mathscinet-getitem?mr=#1}{#2}
}
\providecommand{\href}[2]{#2}

}


\begin{thebibliography}{10}

\bibitem{BS2}
Z.~Bai and J.~W. Silverstein, \emph{No eigenvalues outside the support of the
  limiting spectral distribution of large-dimensional sample covariance
  matrices}, Ann. Prob. \textbf{26} (1998), 316--345.

\bibitem{BS3}
\bysame, \emph{Exact separation of eigenvalues of large dimensional sample
  covariance matrices}, Ann. Prob. \textbf{27} (1999), 1536--1555.

\bibitem{BSbook}
Z.~Bai and J.~W. Silverstein, \emph{Spectral analysis of large dimensional
  random matrices}, Springer, 2010.

\bibitem{BES15}
Z.~Bao, L.~Erd\H{o}s, and K.~Schnelli, \emph{Local law of addition of random
  matrices on optimal scale}, Preprint arXiv:1509.07080.

\bibitem{BaoPanZhou}
Z.~Bao, G.~Pan, and W.~Zhou, \emph{Universality for the largest eigenvalue of a
  class of sample covariance matrices}, Preprint arXiv:1304.5690v4.

\bibitem{BG15}
F.~Benaych-Georges, \emph{Local single ring theorem}, Preprint
  arXiv:1501.07840.

\bibitem{BEKYY}
A.~Bloemendal, L.~Erd{\H{o}}s, A.~Knowles, H.-T. Yau, and J.~Yin,
  \emph{Isotropic local laws for sample covariance and generalized {W}igner
  matrices}, Electron. J. Probab \textbf{19} (2014), 1--53.

\bibitem{KYY}
A.~Bloemendal, A.~Knowles, H.-T. Yau, and J.~Yin, \emph{On the principal
  components of sample covariance matrices}, Prob. Theor. Rel. Fields
  \textbf{164} (2016), 459--552.

\bibitem{CP1}
M.~Capitaine and S.~P\'ech\'e, \emph{Fluctuations at the edges of the spectrum
  of the full rank deformed {GUE}}, Prob. Theor. Rel. Fields \textbf{165}
  (2016), 117--161.

\bibitem{Chat}
S.~Chatterjee, \emph{A generalization of the {L}indeberg principle}, Ann.\
  Prob. \textbf{34} (2006), 2061--2076.

\bibitem{ElK2}
N.~El~Karoui, \emph{{T}racy-{W}idom limit for the largest eigenvalue of a large
  class of complex sample covariance matrices}, Ann. Prob. \textbf{35} (2007),
  663--714.

\bibitem{EKY2}
L.~Erd{\H{o}}s, A.~Knowles, and H.-T. Yau, \emph{Averaging fluctuations in
  resolvents of random band matrices}, Ann.\ H.\ Poincar{\'e} \textbf{14}
  (2013), 1837--1926.

\bibitem{EKYY2}
L.~Erd{\H{o}}s, A.~Knowles, H.-T. Yau, and J.~Yin, \emph{Spectral statistics of
  {E}rd{\H{o}}s-{R}\'enyi graphs {II}: Eigenvalue spacing and the extreme
  eigenvalues}, Comm.\ Math.\ Phys.\ \textbf{314} (2012), 587--640.

\bibitem{EKYY4}
\bysame, \emph{The local semicircle law for a general class of random
  matrices}, Electron. J. Probab \textbf{18} (2013), 1--58.

\bibitem{EKYY1}
\bysame, \emph{Spectral statistics of {E}rd{\H{o}}s-{R}\'enyi graphs {I}: Local
  semicircle law}, Ann. Prob. \textbf{41} (2013), 2279--2375.

\bibitem{ESY2}
L.~Erd{\H{o}}s, B.~Schlein, and H.-T. Yau, \emph{Local semicircle law and
  complete delocalization for {W}igner random matrices}, Comm.\ Math.\ Phys.\
  \textbf{287} (2009), 641--655.

\bibitem{ESY4}
\bysame, \emph{Universality of random matrices and local relaxation flow},
  Invent. Math. \textbf{185} (2011), no.~1, 75--119.

\bibitem{ESYY}
L.~Erd{\H{o}}s, B.~Schlein, H.-T. Yau, and J.~Yin, \emph{The local relaxation
  flow approach to universality of the local statistics of random matrices},
  Ann.\ Inst.\ . Poincar{\'e} (B) \textbf{48} (2012), 1--46.

\bibitem{EYY1}
L.~Erd{\H{o}}s, H.-T. Yau, and J.~Yin, \emph{Bulk universality for generalized
  {W}igner matrices}, Prob.\ Theor.\ Rel.\ Fields \textbf{154} (2012),
  341--407.

\bibitem{EYY3}
\bysame, \emph{Rigidity of eigenvalues of generalized {W}igner matrices}, Adv.
  Math \textbf{229} (2012), 1435--1515.

\bibitem{HHN}
W.~Hachem, A.~Hardy, and J.~Najim, \emph{Large complex correlated {W}ishart
  matrices: Fluctuations and asymptotic independence at the edges}, Preprint
  arXiv:1409.7548.

\bibitem{HLNP}
W.~Hachem, P.~Loubaton, J.~Najim, and P.~Vallet, \emph{On bilinear forms based
  on the resolvent of large random matrices}, Ann. Inst. H. Poincar{\'e} (B)
  \textbf{49} (2013), 36--63.

\bibitem{Joh2}
K.~Johansson, \emph{From {G}umbel to {T}racy-{W}idom}, Prob. Theor. Rel. Fields
  \textbf{138} (2007), 75--112.

\bibitem{Kar15}
V.~Kargin, \emph{Subordination for the sum of two random matrices}, Ann. Prob.
  \textbf{43} (2015), 2119--2150.

\bibitem{KY2}
A.~Knowles and J.~Yin, \emph{The isotropic semicircle law and deformation of
  {W}igner matrices}, Comm. Pure Appl. Math. \textbf{66} (2013), 1663--1749.

\bibitem{SL}
J.~O. Lee and K.~Schnelli, \emph{{T}racy-{W}idom distribution for the largest
  eigenvalues of real sample covariance matrices with general population},
  Preprint arXiv:1409.4979.

\bibitem{LSY}
J.~O. Lee and K.~Schnelli, \emph{Edge universality for deformed {W}igner
  matrices}, Rev. Math. Phys. \textbf{27} (2015), 1550018.

\bibitem{LP}
A.~Lytova and L.~Pastur, \emph{Central limit theorem for linear eigenvalue
  statistics of random matrices with independent entries}, Ann. Prob.
  \textbf{37} (2009), 1778--1840.

\bibitem{MP}
V.~A. Marchenko and L.~A. Pastur, \emph{Distribution of eigenvalues for some
  sets of random matrices}, Mat. Sbornik \textbf{72} (1967), 457--483.

\bibitem{Ona}
A.~Onatski, \emph{The {T}racy-{W}idom limit for the largest eigenvalues of
  singular complex {W}ishart matrices}, Ann. Appl. Prob. \textbf{18} (2008),
  470--490.

\bibitem{Past}
L.~A. Pastur, \emph{On the spectrum of random matrices}, Teor. Math. Phys.
  \textbf{10} (1972), 67--74.

\bibitem{PY}
N.~S. Pillai and J.~Yin, \emph{Universality of covariance matrices}, Ann. Appl.
  Probab. \textbf{24} (2014), 935--1001.

\bibitem{Shc}
T.~Shcherbina, \emph{On universality of bulk local regime of the deformed
  {G}aussian unitary ensemble}, Math. Phys. Anal. Geom. \textbf{5} (2009),
  396--433.

\bibitem{BS4}
J.~W. Silverstein and Z.~Bai, \emph{On the empirical distribution of
  eigenvalues of a class of large dimensional random matrices}, J. Multivar.
  Anal. \textbf{54} (1995), 175--192.

\bibitem{SC}
J.~W. Silverstein and S.-I. Choi, \emph{Analysis of the limiting spectral
  distribution of large dimensional random matrices}, J. of Multivariate Anal.
  \textbf{54} (1995), 295--309.

\bibitem{TV1}
T.~Tao and V.~Vu, \emph{Random matrices: {U}niversality of local eigenvalue
  statistics}, Acta Math. \textbf{206} (2011), 1--78.

\bibitem{TW1}
C.~Tracy and H.~Widom, \emph{Level-spacing distributions and the {A}iry
  kernel}, Comm.\ Math.\ Phys.\ \textbf{159} (1994), 151--174.

\bibitem{TW2}
\bysame, \emph{On orthogonal and symplectic matrix ensembles}, Comm.\ Math.\
  Phys.\ \textbf{177} (1996), 727--754.

\bibitem{Wig}
E.P. Wigner, \emph{Characteristic vectors of bordered matrices with infinite
  dimensions}, Ann. Math. \textbf{62} (1955), 548--564.

\end{thebibliography}
\end{document}